\documentclass[11pt]{amsart}

\usepackage{lipsum}
\usepackage{amsfonts}
\usepackage{graphicx}

\usepackage{amssymb,amsmath}

\usepackage{dsfont}
\usepackage[colorlinks, citecolor=red]{hyperref}
\usepackage{color}

\usepackage{mathrsfs}

\usepackage{subfigure}

\usepackage{caption}
\usepackage{bm}
\usepackage{algorithm}
\usepackage{algorithmicx}
\usepackage{algpseudocode}

\newcommand{\xkh}[1]{\left(#1\right)}
\newcommand{\xxkh}[1]{(#1)}
\newcommand{\dkh}[1]{\left\{#1\right\}}
\newcommand{\zkh}[1]{\left[#1\right]}

\newcommand{\nj}[1]{\langle {#1} \rangle}

\newcommand{\norm}[1]{\left\|{#1}\right\|_2}

\newcommand{\norminf}[1]{\left\|{#1}\right\|_{\infty}}

\newcommand{\norms}[1]{\left\|{#1}\right\|}
\newcommand{\normsg}[1]{\left\|{#1}\right\|_{\psi_2}}
\newcommand{\normse}[1]{\left\|{#1}\right\|_{\psi_1}}

\newcommand{\abs}[1]{\left\lvert#1\right\rvert}
\newcommand{\aabs}[1]{\lvert#1\rvert}

\newcommand{\argmin}[1]{{\rm arg\;}\mathop{\rm min}\limits_{#1}}

\newcommand{\minm}[1]{\mathop{\rm min}\limits_{#1}}
\newcommand{\maxm}[1]{\mathop{\rm max}\limits_{#1}}

\newcommand{\sumin}{\sum_{i=1}^n}

\newcommand{\dist}[1]{\mathrm {dist}\left(#1\right)}

\newcommand{\dd}{{\mathrm d}}
\newcommand{\e}{{\mathrm e}}

\newcommand{\E}{{\mathbb E}}

\newcommand{\PP}{{\mathbb P}}

\newcommand{\1}{{\mathds 1}}

\newcommand{\R}{{\mathbb R}}

\newcommand{\C}{{\mathbb C}}

\newcommand{\betat}{{\tilde{\beta}}}

\newcommand{\real}[1]{\mathrm{Re}\left(#1\right)}

\newcommand{\hh}{\mathsf{H}}
\newcommand{\lkh}{{(l)}}
\newcommand{\T}{\top}

\newcommand{\zxi}{z_i^*}
\newcommand{\zxis}{z_i^{*2}}
\newcommand{\zxic}{\bar{z}_i^*}

\newcommand{\zxj}{z_j^*}

\newcommand{\zxjc}{\bar{z}_j^*}

\newcommand{\wi}{w_i}

\newcommand{\wic}{\bar{w}_i}
\newcommand{\wics}{\bar{w}_i^{2}}

\newcommand{\wj}{w_j}

\newcommand{\wjc}{\bar{w}_j}

\newcommand{\va}{{\bm a}}

\newcommand{\vu}{{\bm u}}

\newcommand{\vx}{{\bm x}}
\newcommand{\vy}{{ \bm{ y}}}
\newcommand{\vone}{{\bm 1}}

\newcommand{\vz}{{\bm z}}
\newcommand{\vzt}{{\bm z}^t}
\newcommand{\vztu}{\tilde{\bm z}^t}

\newcommand{\vztl}{{\bm z}^{t,(l)}}
\newcommand{\vztlu}{\tilde{\bm z}^{t,(l)}}
\newcommand{\vztlh}{\hat{\bm z}^{t,(l)}}

\newcommand{\vztt}{{\bm z}^{t+1}}
\newcommand{\vzttu}{\tilde{\bm z}^{t+1}}

\newcommand{\vzttl}{{\bm z}^{t+1,(l)}}
\newcommand{\vzttlu}{\tilde{\bm z}^{t+1,(l)}}
\newcommand{\vzttlh}{\hat{\bm z}^{t+1,(l)}}

\newcommand{\vzx}{{\bm z}^*}
\newcommand{\vzxt}{{\bm z}^{*\top}}
\newcommand{\vzxh}{{\bm z}^{*\hh}}
\newcommand{\vzxc}{\bar{{\bm z}}^{*}}
\newcommand{\vxi}{{\bm \xi}}

\newcommand{\vw}{{\bm w}}
\newcommand{\vwt}{{\bm w}^{\top}}
\newcommand{\vwh}{{\bm w}^{\hh}}
\newcommand{\vwc}{\overline{{\bm w}}}

\newcommand{\tz}{{\tilde{\bm{z}}}}

\newcommand{\ma}{{\bm A}}
\newcommand{\mb}{{\bm B}}

\newcommand{\md}{{\bm D}}
\newcommand{\mi}{{\bm I}}
\newcommand{\mm}{{\bm M}}

\newcommand{\mzero}{{\bm 0}}
\newcommand{\mml}{{\bm M}^{\lkh}}

\newcommand{\ve}{{\bm e}}

\newcommand{\ep}{\varepsilon}

\newcommand{\mdiag}{{\rm diag}}

\renewcommand{\omega}{\eta}
\newcommand{\phih}{\hat{\phi}}

\newcommand{\sqrttwo}{\sqrt{2}}
\newcommand{\sqrtonetwo}{\frac{1}{\sqrt{2}}}
\newcommand{\sqrtn}{\sqrt{n}}
\newcommand{\sqrtm}{\sqrt{m}}
\newcommand{\sqrtlm}{\sqrt{\log m}}
\newcommand{\sqrtnlm}{\sqrt{n \log m}}
\newcommand{\sqrtnlmmm}{\sqrt{n \log^3 m}}

\newcommand{\gl}{\gamma_l}

\newcommand{\RNum}[1]{\uppercase\expandafter{\romannumeral #1\relax}}

\newcommand{\deltaztx}{\begin{bmatrix}
		\vztu-\vzx\\
		\overline{\vztu-\vzx}
\end{bmatrix}}

\newtheorem{definition}{Definition}[section]

\newtheorem{theorem}[definition]{Theorem}
\newtheorem{lemma}[definition]{Lemma}

\newtheorem{remark}[definition]{Remark}

\date{}

\begin{document}

\author{Haiyang Peng}
\address{School of Mathematical Sciences, Beihang University, Beijing, 100191, China }
\email{haiyangpeng@buaa.edu.cn}

\author{Deren Han}
\address{School of Mathematical Sciences, Beihang University, Beijing, 100191, China }
\email{handr@buaa.edu.cn}

\author{Linbin Li}
\address{School of Mathematical Sciences, Beihang University, Beijing, 100191, China }
\email{linbinli@buaa.edu.cn}

\author{Meng Huang}
\address{School of Mathematical Sciences, Beihang University, Beijing, 100191, China }
\email{menghuang@buaa.edu.cn}


\baselineskip 18pt
\bibliographystyle{plain}

\title{Noisy phase retrieval from Subgaussian measurements}

\maketitle

\begin{abstract}
This paper aims to address the phase retrieval problem from subgaussian measurements with arbitrary noise, with a focus on devising robust and efficient  algorithms for solving non-convex problems.	 To ensure uniqueness of solutions in the subgaussian setting, we explore two commonly used assumptions: either the subgaussian measurements satisfy a fourth-moment condition or the target signals exhibit non-peakiness. For each scenario, we introduce a novel spectral initialization method that yields robust initial estimates.    Building on this, we employ leave-one-out arguments to show that the classical Wirtinger flow algorithm achieves a linear rate of convergence for both real-valued and complex-valued cases,  provided the sampling complexity  $m\ge O(n \log^3 m)$, where $n$ is  the dimension of the underlying signals.   In contrast to existing work, our algorithms are regularization-free, requiring no truncation, trimming, or additional penalty terms, and they permit the algorithm step sizes as large as  $O(1)$, compared to the $O(1/n)$ in previous literature.  Furthermore, our results accommodate arbitrary noise vectors that meet certain statistical conditions, covering a wide range of noise scenarios, with sub-exponential noise as a notable special case. The effectiveness of our algorithms is validated through various numerical experiments.  We emphasize that our findings provide the first theoretical guarantees for recovering non-peaky signals using non-convex methods from Bernoulli measurements, which is of independent interest.
\end{abstract}

\keywords{Keywords: noisy phase retrieval, subgaussian measurements, Wirtinger flow, leave-one-out analysis}


\section{Introduction}
\subsection{Phase retrieval}
We consider the \textit{phase retrieval} problem, where the goal is to recover the signal $\vzx \in \C^n$ from  measurements
\begin{equation}\label{def:measurements}
	y_j:=\abs{\nj{\va_j,\vzx}}^2+\xi_j, \quad  j=1,\ldots,m.
\end{equation}
Here,  $ \va_j \in \C^n$ are known sensing vectors and $\bm{\xi}:=(\xi_1,\ldots,\xi_m)^\T \in \R^m$ is a noise vector.  Throughout this paper, we make the assumption that the noise vector  is either fixed or random, and is
independent to all sensing vectors $\va_j$. Phase retrieval is a significant research topic in fields such as X-ray crystallography \cite{Harrison1993, Millane1990}, optics \cite{Gonsalves1982, Walther1963}, quantum information \cite{Heinosaari2013}, quantum mechanics \cite{Corbett2006}, and interferometry \cite{Demanet2017}. The central challenge in phase retrieval is to reconstruct phase information solely from intensity-based measurements.


Based on the least squares criterion, one can employ the following Wirtinger-based empirical loss to
estimate $\vzx$:
\begin{equation}\label{def:f}
	\minm{\vz\in\C^n}\quad f(\vz) = \frac{1}{2m} \sum_{j=1}^{m}\xkh{\abs{\nj{\va_j,\vz}}^2-y_j}^2.
\end{equation}
In scenarios where the measurements $\va_j$ are standard Gaussian random vectors, numerous robust phase retrieval algorithms have been adopted  to efficiently solve \eqref{def:f}, with notable examples including works by \cite{CaiJianfeng2022, Candes2015WF, ChenYuxin2019, Macong2020, SunJu2018, HuangMeng2022}.  However,  when considering a broader class of subgaussian measurements, there are few algorithms with robust theoretical guarantees, especially for phase retrieval from symmetric Bernoulli measurements.  This paper focus on noisy phase retrieval with subgaussian measurements and aims  to investigate the rate of convergence  of the classical Wirtinger flow method based on two prevalent assumptions: either the subgaussian measurements satisfy a fourth-moment condition or the target signals exhibit non-peakiness.

\subsection{Related work}
Over the past decade, researchers have devised numerous robust phase retrieval algorithms under standard Gaussian measurements. Cand{\`e}s et al. introduced the PhaseLift algorithm, demonstrating its capability for ensuring precise recovery of phase information in the absence of noises \cite{Candes2013Phaselift}. Waldspurger et al. formalized phase retrieval  as a quadratic optimization problem on the unit complex torus and solved it using the block coordinate descent algorithm following convex relaxation \cite{Waldspurger2015}. However, those convex approaches involve lifting the problem to a higher-dimensional space, which is  inefficient for addressing large-scale problems. Further explorations of convex optimization techniques for phase retrieval can be found in \cite{Bahmani2017, Doelman2018, Goldstein2018}.

The non-convex methods for phase retrieval operate directly on the original space, achieving significantly improved computational performance. A prominent example of non-convex algorithms is the Wirtinger flow (WF), introduced by Cand{\`e}s et al. in 2015 \cite{Candes2015WF}. Its recovery guarantee hinges on a well-designed initialization and a gradient-based refinement procedure, enabling a linear rate of convergence with a moderate number of measurements. Ma et al. revisited WF and refined the theoretical step size significantly using the leave-one-out technique \cite{Macong2020}. Employing the same approach, Chen et al. demonstrated that WF can effectively converge to a global minimizer even with random initialization in the real-valued scenario \cite{ChenYuxin2019}.
To improve the sampling complexity and the convergence performance of WF,
Chen et al. proposed truncated Wirtinger flow algorithm (TWF) by devising rational truncation rules, and proved that it can provide more accurate initial guess and better theoretical guarantees \cite{ChenYuxin2015TWF}.  Other algorithms for solving phase retrieval problem under standard Gaussian measurements includes incremental truncated Wirtinger flow (ITWF) \cite{Kolte2016ITWF},  median-TWF \cite{ZhangHuishuai2018median-TWF}, and amplitude-based methods \cite{Luoqi2020, SamuelPinilla2018, Wanggang2017sto, Wanggang2017, Wanggang2018,  ZhangHuishuai2016Reshaped}, just to name a few.
The idea of WF for phase retrieval has inspired research in diverse domains. These include blind deconvolution \cite{ChenYuxin2023,LiXiaodong2019,Macong2020}, blind demixing \cite{PeterJung2017,LingShuyang2019}, and matrix completion \cite{ChenYuxin2020}.

It is important to note that algorithms for phase retrieval discussed earlier rely on i.i.d. Gaussian random samples, representing only a limited set of measurement types.  When extending from Gaussian to subgaussian measurements, additional assumptions are necessary to ensure uniqueness. For instance, if the measurements $\va_j$ comes from i.i.d.  symmetric Bernoulli distribution - where each entry of $\va_j$ takes the values $+1$ and $-1$ with probability $1/2$ - the vector $\vz_1:=(1,0,\ldots,0)^\T$ can never be distinguished from the vector $\vz_2:=(0,1,0,\ldots,0)^\T$.  Further elucidations on ambiguities in phase retrieval from subgaussian measurements are provided in \cite{Krahmer2020}.  To make the problem well-posed, some research focuses on real-valued signals and requires that  each entry of measurement vectors satisfies the so-called small-ball probability \cite{ChenYuxin2015Convex} or
\begin{equation}\label{eq:assum1}
	\E \zkh{\abs{a_{jk}}^4}>\E \zkh{\abs{a_{jk}}^2}
\end{equation}
for all $j \in [m]$ and $k \in [n]$. Other research approaches the problem from the perspective of target signals, assuming that the peak-to-average power ratio of $\vzx$ satisfies
\begin{equation}\label{eq:assum2}
	\frac{\norminf{\vzx}}{\norm{\vzx}} \le \mu <1
\end{equation}
for some constant $\mu>0$ \cite{Krahmer2018}. Under the fourth-moment assumption, the authors in \cite{LiHuiping2020PhaseMax} proposed a PhaseMax method to tackle the phase retrieval problem through convex relaxation. Subsequently,  \cite{LiHuiping2021} proposed a variant of gradient descent method based on Riemannian optimization, termed truncated Riemannian gradient descent algorithm (TRGrad), and demonstrated its linear rate of convergence. Gao et al. explored complex  subgaussian measurements, and demonstrated that the Wirtinger flow method with spectral initialization converges linearly to the underlying signals \cite{Gaobing2021}.  Revealing a profound insight, Krahmer et al. demonstrated that a broad range of vectors $\vx\in\R^n$ can be uniquely reconstructed from subgaussian measurements without necessitating the small-ball probability assumption or the fourth-moment condition \cite{Krahmer2018}, enabling phase retrieval from symmetric Bernoulli measurements.   They classified these vectors as meeting the \(\mu\)-flatness criterion as in \eqref{eq:assum2} and established that the PhaseLift method can effectively recover such \(\mu\)-flatness signals with high probability.  Subsequently,  Krahmer et al. extended this results from the real-valued case to the complex-valued case \cite{Krahmer2020}. To the best of our knowledge, PhaseLift \cite{Krahmer2018,Krahmer2020} stands out as the sole method for phase retrieval from symmetric Bernoulli measurements,  with no non-convex algorithm currently accessible for solving this specific type of phase retrieval problem.

\subsection{Our contributions}\label{sec:contributions}
This paper aims to propose efficient algorithms with theoretical guarantees for noisy phase retrieval from sub-Gaussian measurements under either Assumption \eqref{eq:assum1} or Assumption \eqref{eq:assum2}. To this end, our approach involves (i) designing robust spectral initialization procedures to provide a reliable initial guess for each assumption, and (ii) refining this initial guess using the classical Wirtinger flow method. For each assumption, we establish a linear rate of convergence  for the proposed algorithm through leave-one-out analysis. In contrast to existing methods, our algorithms are regularization-free, requiring no truncation/trimming or additional penalty terms, and permit a step size of   $O(1)$. Moreover, our results apply to any noise vector $\vxi$ that satisfies certain statistical properties, covering a broad spectrum of noise scenarios, with sub-exponential noise as a notable special case.
The results are summarized below:
\subsubsection{Results under the fourth-moment condition (\ref{eq:assum1})}
Suppose that the measurement vectors $\va_1,\cdots,\va_m\in \C^n$ are independent copies of a vector $\va=(a_1,a_2,\cdots,a_m)^\T \in \C^n$.  We assume that  the vector $\va$ obeys the following assumption: \medskip
\paragraph{\textit{Assumption A}} The entries of $\va $ are i.i.d. subguassian random variables with $\normsg{a_1}\le K$ for a constant $K>0$, $\E a_1=0$, $\E\abs{a_1}^2=1$, $\E\abs{a_1}^4= 1+\beta_1$ and $\abs{\E a_1^2}^2 = 1-\beta_2$ for some $\beta_1 \in (0,+\infty)$ and $\beta_2 \in (0,1]$. \medskip

\begin{algorithm}
	\caption{Vanilla gradient descent for phase retrieval (with Assumption A)}\label{alg:1}
	\begin{algorithmic}[0]
		\State \textbf{Input:} The sensing vectors $\dkh{\va_j}_{1\le j\le m}$, the observations  $\dkh{y_j}_{1\le j\le m}$, the parameters $\beta_1$, $\beta_2$ and $\eta_0$, the maximum number of iterations $T$.
		\State \textbf{Spectral initialization: } Let $\check{\vz}^0$ be the left singular vector for the leading singular value of \begin{equation*}
			\mm = \frac{1}{\beta_2}\mm_0+\frac{2-\beta_1-\beta_2}{\beta_1(2-\beta_2)}\mdiag\xkh{\mm_0}-\frac{2-2\beta_2}{\beta_2(2-\beta_2)}\real{\mm_0} + \frac{\beta_1-1}{\beta_1} \gamma \mi,
		\end{equation*}
		where
		\[
		\mm_0 = \frac{1}{m}\sum_{j=1}^{m} y_j\va_j\va_j^\hh  \qquad \mbox{and} \qquad  \gamma = \frac{1}{m}\sum_{j=1}^{m} y_j.
		\]
		Set $\vz^0 = \sqrt{\gamma}\check{\vz}^0$.
		\State \textbf{Gradient updates: for} $t = 0,1,2\cdots,T-1$ \textbf{do}\begin{equation*}
			\vztt = \vzt - \eta \nabla_{\vz} f(\vzt),
		\end{equation*}
		where the step size $\eta = \eta_0 /\gamma$. Here, $\nabla_{\vz} f $ is the Wirtinger gradient of $f$ in \eqref{def:f}.
	\end{algorithmic}
\end{algorithm}
We emphasize that the assumption $\abs{\E a_1^2}^2 = 1-\beta_2$ with $\beta_2 \in (0,1]$ cannot be avoided for complex subgaussian measurements. Otherwise, the signal $\vzx$ can never be distinguished from its complex-conjugate $\bar{\vz}^*$ \cite{Krahmer2020}.  In particular, if each entry of $\va $ is independently generated according to $a_j \sim x+iy$, where $x$ and $y$ are zero-mean independent random variables with the same distribution, then $\beta_2=1$. Notably,  the condition $\beta_1>0$ corresponding to  the fourth-moment condition \eqref{eq:assum1}.
Under Assumption A, one has the following result:

\begin{theorem} \label{th:A}
	Suppose that Assumption A holds. For any fixed $\vzx\in\C^n$, assume that the noise vector $\xi \in \R^m$ obeys $\abs{\vone^\T \vxi}\le \tilde{C}_{1} m\norm{\vzx}^2,\ \norm{\vxi}\lesssim \sqrtm \norm{\vzx}^2$, and $\norminf{\vxi}\lesssim \log m\norm{\vzx}^2$ for some sufficiently small constant $\tilde{C}_{1}>0$. Then with probability at least $1 - m\exp (-c_0 n) - O(m^{-10})$,   Algorithm \ref{alg:1} with step size \(\eta = c/\norm{\vzx}^2\)  satisfy
	\begin{align*}
		\dist{\vzt,\vzx} \le (1-\frac{\betat c}{8})^t C_{1}\norm{\vzx}+\frac{8\abs{\vone^\T \vxi}+C_{2}K^2\xkh{\sqrt{n} \norm{\vxi} + n \norminf{\vxi}}}{m\betat\norm{\vzx}}
	\end{align*}	
	simultaneously for all $0\le t\le t_0\le m^{10}$, provided \(m \ge C K^8n \log^3 m\) for some sufficiently large constant \(C>0\) depending only on $\beta_1, \beta_2$. Here, $\vone=\xkh{1,1,\ldots,1}^\T\in\R^m$, $\betat=\min\dkh{\beta_1,\beta_2}$; $C_{2}, c_0>0$ are universal constants, and $C_{1},\tilde{C}_{1}>0$ and $0<c<\betat \le 1$ are constants depending only on $\beta_1, \beta_2$.	
\end{theorem}

\begin{remark}
	In the realm of phase retrieval with subgaussian measurements, the work by Gao et al. \cite{Gaobing2021} exhibits similarities to our Theorem \ref{th:A}. They demonstrated that when $m \ge O(n \log^2 n)$, the Wirtinger flow (WF) can attain a linear rate of convergence  with high probability.  However, their results are established in the absence of noise, whereas our results accommodate  arbitrary noise vector obeys certain statistical conditions. Furthermore, their findings necessitate a relatively small step size of $O(1/n)$, leading to a computational complexity of $O(n \log(1/\ep))$ to achieve an $\ep$-accuracy. In contrast,  our results allow the step size as large as $O(1)$, thereby reducing the computational complexity to $O(\log(1/\ep))$.
\end{remark}

\begin{remark}
	Since Theorem \ref{th:A} is valid for any arbitrary noise vector $\vxi$, a natural step is to leverage it to derive convergence guarantees for random noise. Specifically, for zero-mean  subexponential noise $\vxi$ with $\max_{1\le j\le m} \normse{\xi_j}\le \sigma \norm{\vzx}^2$,  according to \cite[Lemma A.7]{Cai2016optimal}, it holds with probability at least $1- \frac3{m}$ that $\norminf{\vxi} \lesssim \sigma \log m \norm{\vzx}^2, ~\norm{\vxi} \lesssim  \sigma \sqrt{m} \norm{\vzx}^2$, and $\abs{\vone^\T \vxi} \lesssim \sigma \sqrt{m \log m} \norm{\vzx}^2$.
Therefore, by applying Theorem \ref{th:A},  the iterations in Algorithm \ref{alg:1} with subexponential noise obey	
	\begin{align*}
		\dist{\vzt,\vzx}\le& (1-\frac{\betat c}{8})^t C_{1}\norm{\vzx}+\frac{C_{2}' K^2\sigma}{\betat} \sqrt{\frac{n}{m}}\norm{\vzx}
	\end{align*}
	for all $0\le t\le t_0\le m^{10}$, provided $m\ge O(n \log^3 m)$. Here, $C_{2}'>0$ is a universal constant.  It is worth noting that the error bound $O(\sqrt{n/m})$ on the right-hand side is optimal \cite{Eldar2014,Lecue2015,Plan2013}.
\end{remark}

\begin{remark}\label{rm:parameter}
	The constants $C,C_{1},\tilde{C}_{1}$ and $c$ in Theorem \ref{th:A} depend on $\beta_1$ and $\beta_2$.
	The exact values  are:
	\[C=\frac{C_0 (1+\beta_1)^4}{\betat^{4}},\quad  C_{1}=\frac{C_{1}'\betat}{1+\beta_1},\quad \tilde{C}_{1}=\frac{C_{1}'\betat^2}{22(1+\beta_1)^2},\quad  c \le \frac{\betat}{36(4+\beta_1)^{2}}.
	\]
	Here,  $C_0,C'_{1}>0$ are absolute constants.
	This implies that larger values of  $\beta_1$ and $\beta_2$ reduce the required sampling complexity $m$. Additionally, as
	 $\beta_1$ and $\beta_2$ increase, the algorithm converges more quickly. This will be confirmed in our numerical experiments.	
\end{remark}

\subsubsection{Results for non-peakiness signals without the fourth-moment condition}
In scenarios where the measurements $\va_j$ do not satisfy the fourth moment condition \eqref{eq:assum1}, as discussed previously, we are limited to recovering those \(\mu\)-flatness signals as in \eqref{eq:assum2}.  Assume that the measurement vectors $\va_j\in \C^n$ are independent copies of a vector $\va=(a_1,a_2,\cdots,a_m)^\T\in \C^n$.  Then we require the following assumption: \medskip
\paragraph{\textit{Assumption B}} The entries of $\va$ are i.i.d. subguassian random variables with $\normsg{a_1}\le K$,  $\E a_1=0$, $\E\abs{a_1}^2=\E\abs{a_1}^4=1$, and $\abs{\E a_1^2}^2 = 1-\beta_2$ for some $\beta_2 \in (0,1]$. Furthermore, the underlying signal $\vzx$ satisfies $\norminf{\vzx}^2/\norm{\vzx}^2\le \mu<1$ for some constant $\mu>0$.



\begin{algorithm}
	\caption{Vanilla gradient descent for phase retrieval (with Assumption B)}\label{alg:2}
	\begin{algorithmic}[0]
		\State \textbf{Input:} The sensing vectors $\dkh{\va_j}_{1\le j\le m}$, the observations $\dkh{y_j}_{1\le j\le m}$, the parameters $\beta_2$ and $\eta_0$, the maximum number of iterations $T$.
		\State \textbf{Spectral initialization: } Let $\check{\vz}^0$ be the left singular vector for the leading singular value of \begin{equation*}
			\mm = \frac{1}{\beta_2}\mm_0-\frac{2-2\beta_2}{\beta_2(2-\beta_2)}\real{\mm_0} + \frac{1-\beta_2}{2-\beta_2}\gamma \mi
		\end{equation*}
		where
		\[
		\mm_0 = \frac{1}{m}\sum_{j=1}^{m} y_j\va_j\va_j^\hh \qquad \mbox{and} \qquad  \gamma = \frac{1}{m}\sum_{j=1}^{m} y_j.
		\]
		Set $\vz^0 = \sqrt{\gamma}\check{\vz}^0$.
		\State \textbf{Gradient updates: for} $t = 0,1,2\cdots,T-1$ \textbf{do}\begin{equation*}
			\vztt = \vzt - \eta \nabla_{\vz} f(\vzt),
		\end{equation*}
		where $\eta = \eta_0 /\gamma$.
	\end{algorithmic}
\end{algorithm}
Our results for this assumption are presented as follows.

\begin{theorem} \label{th:B}
	For any fixed $\vzx\in\C^n$, suppose that Assumption B holds with $\mu\le 1/10$ and  the noise vector $\xi \in \R^m$ obeys $\abs{\vone^\T \vxi}\le \tilde{C}_{11}  m\norm{\vzx}^2,\ \norm{\vxi}\lesssim \sqrtm \norm{\vzx}^2$, and $\norminf{\vxi}\lesssim \log m\norm{\vzx}^2$ for some sufficiently small  constant $\tilde{C}_{11}>0$.   Then with probability at least $1 - m\exp (-c_{0} n)-O(m^{-10})$ the Wirtinger flow method given in Algorithm \ref{alg:2} with step size \(\eta = c/\norm{\vzx}^2\) obeys
	\begin{align*}
		\dist{\vzt,\vzx}\le (1-\frac{\beta_2 c}{8})^t C_{11}\norm{\vzx}+\frac{8\abs{\vone^\T \vxi}+C_{12}K^2\xkh{\sqrt{n} \norm{\vxi} + n \norminf{\vxi}}}{m\beta_2\norm{\vzx}}
	\end{align*}
	simultaneously  for all $0\le t\le t_0\le m^{10}$, provided \(m \ge CK^8n \log^3 m\). Here, \(C>0\) is a  sufficiently large constant depending only on $\beta_2$, $C_{12}>0$ is a universal constant, and
	$C_{11}, \tilde{C}_{11}>0$, $0<c<\beta_2\le 1$ are constants depending only on $\beta_2$.
\end{theorem}	

\begin{remark}
	The exact values of constants $C,C_{11},\tilde{C}_{11}$ and $c$ in Theorem \ref{th:B} are
	\[C=\frac{C_{10}}{\beta_2^{4}},\quad  C_{11}=C_{11}'\beta_2,\quad  \tilde{C}_{11}=\frac{C_{11}'\beta_2^4}{14},\quad  c \le \frac{\beta_2}{324},  \]
	where $C_{10},C_{11}'>0$ are some absolute constants.
\end{remark}

It is worth noting that Theorem \ref{th:A} and Theorem \ref{th:B} are also valid for the real-valued case where the measurements $\va_j \in \R^n$ and the underlying signals $\vzx\in \R^n$. To make the paper more concise, we present the results for the real-valued case in Section \ref{sec:realcase}.  In particular, when applied our results to symmetric Bernoulli measurements, we can show that the Wirtinger flow method is able to recover non-peaky signals with a linear  rate of convergence.  To the best of our knowledge, this is the first non-convex algorithm with theoretical guarantees  for recovering non-peaky signals from symmetric Bernoulli measurements.

\subsection{Notations} \label{sec:notation}
For any vector $\vz$,  we denote the Euclidean norm of a vector as $\norm{\vz}$. Moreover, $\vz^\T$, $\vz^\hh$ and $\bar{\vz}$ represent the transpose, the conjugate transpose, and the entrywise conjugate of $\vz$ respectively.
For a matrix $\ma$, $\sigma_j(\ma)$ and $\lambda_j(\ma)$ denote its $j$-th largest singular value and eigenvalue, while $\norm{\ma}$ denotes its spectral norm. Additionally, $\ma^\T$, $\ma^\hh$ and $\bar{\ma}$ stand for the transpose, the conjugate transpose, and the entrywise conjugate of $\ma$, respectively. We use the notation  $\real{\cdot}$ to denote the real part of a complex number or a complex matrix(i.e., the real part of each entrywise element). The symbol $\mdiag\xkh{\ma}$ represents the diagonal part of matrix $\ma$, while $\mdiag\xkh{a_1,a_2,\cdots,a_n}$ signifies a diagonal matrix with $a_1,a_2,\cdots,a_n\in\C$ as the diagonal elements.
For any vector  \( \vz=\xkh{z_1,z_2,\cdots,z_n}^\T \in\C^n \), the diagonal matrices $\md_1(\vz)$ and $ \md_2(\vz)$, which will be repeatedly utilized in this paper, are defined as follows:
\begin{equation} \label{def:md}
		\md_1(\vz) = \mdiag\xkh{\abs{z_1}^2,\abs{z_2}^2,\cdots,\abs{z_n}^2},\qquad
	\md_2(\vz) = \mdiag\xkh{z_1^2,z_2^2,\cdots,z_n^2}.
\end{equation}
Furthermore, the notation \(f(m,n)=O(g(m,n))\) or \(f(m,n)\lesssim g(m,n)\) denotes the existence of a constant \(c>0\) such that \( f(m,n)\le cg(m,n) \), while \(f(m,n)\gtrsim g(m,n)\) indicates that there exist a constant \(c>0\) such that \(f(m,n)\ge cg(m,n)\).

Since $\vz^* \vz^{*\hh}=\e^{i \phi}\vz^* \xkh{\e^{i \phi}\vz^{*}}^\hh$ for any $\phi\in \R$, we cannot distinguish between $\vz^*$ and $\e^{i \phi}\vz^{*}$ by intensity-only measurements  $\aabs{\va_j^\hh \vzx}$. So we define the distance between $\vz_1$ and $\vz_2$ as
\begin{equation}\label{def:dist}
	\dist{\vz_1, \vz_2} = \minm{\phi\in\R} \norm{\e^{i\phi}\vz_1-\vz_2}.
\end{equation}
For convenience,  let
\begin{equation}\label{def:ztilde:a}
		\phi(t) = \argmin{\phi\in \R} \norm{\e^{i \phi}\vz^t-\vz^*} \quad \mbox{and}\quad
		\tz^t = e^{i \phi(t)}\vz^t.
\end{equation}
We say $\vz_1$ is aligned with $\vz_2$, if
\begin{equation*}
	\norm{\vz_1-\vz_2} = \minm{\phi\in\R} \norm{\e^{i\phi}\vz_1-\vz_2}.
\end{equation*}

\subsection{Organization}
The remainder of this paper is organized as follows: Section \ref{sec:preliminaries} introduces definitions and properties of subgaussian and subexponential random variables. Section \ref{sec:pf:A} and Section \ref{sec:pf:B} present the proofs of our two main results, respectively. In Section \ref{sec:realcase}, we provide two conclusions for noisy phase retrieval from real-valued subgaussian measurements. Section \ref{sec:experiments} evaluates the empirical performance of our algorithms through a series of numerical experiments. Finally, Section \ref{sec:discussion} offers a discussion and outlook for this work. Technical lemmas and detailed proofs are provided in the supplementary materials.

\section{Preliminaries: subgaussian and subexponential random variables}\label{sec:preliminaries}
Let $X$ be a random variable. For $1\le \alpha \le 2$, define $\norms{\cdot}_{\psi_\alpha} $ as
\begin{equation}\label{def:normsg normse}
	\norms{X}_{\psi_\alpha} =\inf \dkh{t>0: \E\exp(|X/t|^\alpha)\le 2}.
\end{equation}
\begin{definition}
	For a random variable $X$, if $\norms{X}_{\psi_1} \le K_1$ for some constant $K_1>0$ then we say $X$ is a subgaussian random variable with parameter $K_1$; if $\norms{X}_{\psi_2} \le K_2$ for some constant $K_2>0$  then we say $X$ is a subexponential random variable with parameter $K_2$.
\end{definition}

It has been shown in \cite[Exercise 2.5.7]{Vershynin2018} that if $X$ is a subguassian random variable, then the tails of $X$ satisfy
\begin{equation} \label{eq:sgtail}
\PP\dkh{\abs{X}\ge t} \le 2\exp\xkh{- \frac{c_1 t^2}{\norms{X}_{\psi_2}^2 }} \quad \mbox{for all} \quad t\ge 0,
\end{equation}
and the moments of $X$ obey
\begin{equation} \label{eq:subEp}
	\xkh{\E\abs{X}^p}^{1/p} \le c_2 \norms{X}_{\psi_2} \sqrt p \quad \mbox{for all} \quad p\ge 1,
\end{equation}
where $c_1, c_2 >0$ are universal constants.  Furthermore, if $X_1, X_2$ are subgaussian random variables, then $X_1X_2$ is a sub-exponential random variable with $\norms{X_1X_2}_{\psi_1}  \le  \norms{X_1}_{\psi_2}\norms{X_2}_{\psi_2}$.

\begin{definition}
	A random vector $X \in \C^n$ is called subgaussian with parameter $K>0$  if the one-dimensional marginals $\nj{X,\vw}$ are subgaussian random variables with parameter $K>0$ for all $\vw \in \C^n$, or equivalently, $\norms{X}_{\psi_2}:= \sup_{\norm{\vw}=1} \norms{ \nj{X,\vw}}_{\psi_2} \le K$.
\end{definition}

Assume that $X=(x_1,\ldots, x_n) \in \C^n$ is a random vector with independent, zero-mean, subgaussian coordinates $x_j$.  Then $X$ is a subgaussian random vector  obeying  $\norms{X}_{\psi_2} \le C \max_{1\le j \le n} \norms{x_j}_{\psi_2}$. Here, $C>0$ is a universal constant.

\section{Proof of Theorem \ref{th:A}}\label{sec:pf:A}
In this section, we start from Assumption A and derive Theorem \ref{th:A}.
Before proceeding, we provide some intuitions for the spectral initialization given in Algorithm \ref{alg:1}.
Under Assumption A, we deduce from Lemma \ref{le:e} that
\begin{equation*}
	\E \zkh{\frac 1m \sum_{j=1}^m y_j\va_j\va_j^\hh} = \norm{\vzx}^2\mi + \vzx\vzxh + (1-\beta_2)\vzxc\vzxt - \xkh{2-\beta_1-\beta_2}\md_1(\vzx) + \frac{\vone^\T \vxi}{m}\mi
\end{equation*}
where $\md_1(\vzx)$ is defined in \eqref{def:md}. 
Due to the existence of $\md_1(\vzx)$, the singular values of $\E \zkh{\frac 1m \sum_{j=1}^m y_j\va_j\va_j^\hh}$ is challenging to estimate.  However, through some simple transformations, we can construct a matrix $\mm$ such that it obeys
\[\E \mm = \norm{\vzx}^2\mi + \vzx\vzxh + \frac{\vone^\T \vxi}{m}\mi.\]
Therefore, a meaningful approximation of $\vzx$ can be obtained by computing the leading eigenvector of $\mm$.

%

%

Throughout this section, without loss of generality, we assume that $\norm{\vzx}=1$. Recall the minimization program we consider is
\[
\minm{\vz\in\C^n}\quad f(\vz) = \frac{1}{2m} \sum_{j=1}^{m}\xkh{\abs{\nj{\va_j,\vz}}^2-y_j}^2.
\]
Given that  $f$ is a real-valued complex function,  we proceed to consider its Wirtinger gradient, which is expressed as:
\begin{eqnarray}
	\label{def:df}&&\nabla f_\vz(\vz) = \frac1m \sum\limits_{j=1}^m \xkh{\aabs{\va_j^\hh \vz}^2-y_j}\va_j\va_j^\hh \vz
\end{eqnarray}
Similarly,  the Hessian matrix is
\begin{equation}\label{def:Hessian}
	\nabla^2 f(\vz)=
	\begin{bmatrix}
		\frac 1m\sum\limits_{j=1}^m \xkh{2\abs{\va_j^\hh \vz}^2-y_j}\va_j\va_j^\hh & \frac 1m\sum\limits_{j=1}^m \xkh{\va_j^\hh \vz}^2\va_j\va_j^\top \\
		\frac 1m\sum\limits_{j=1}^m \xkh{\vz^\hh \va_j}^2\bar{\va}_j\va_j^\hh &\frac 1m\sum\limits_{j=1}^m \xkh{2\abs{\va_j^\hh \vz}^2-y_j} \bar{\va}_j\va_j^\top
	\end{bmatrix}.
\end{equation}
For a more detailed exploration of the Wirtinger calculus, please refer to \cite{Candes2015WF,Wirtinger1,SunJu2018}.

\subsection{Error Contraction}
To show the Wirtinger flow method given in Algorithm \ref{alg:1} possesses the linear rate of convergence, we first demonstrate that the loss function $f$ has restricted strong convexity and smoothness condition in some region of incoherence and contraction (RIC), as stated below.
\begin{lemma}\label{le:ric:a}
Denote $\betat=\min \dkh{\beta_1,\beta_2}$.	For any fixed $\vzx \in \C^n$ with $\norm{\vzx} = 1$, suppose that Assumption A holds true and the noise $\vxi$ satisfies
	\begin{equation}\label{cond:noise:a}
		\abs{\vone^\T \vxi}\le \tilde{C}_{1} m,\ \norm{\vxi}\lesssim \sqrtm ,\ \mbox{and}\ \norminf{\vxi}\lesssim \log m
	\end{equation}
	for some sufficiently small constant $\tilde{C}_{1}>0$.
	Assume that  $m\ge C_0 \betat^{-2} K^8 n \log^3 m$ for some absolute large constant $C_0 > 0$. Then with probability at least $1-m\exp (-c_0 n)-O(m^{-10})$,
	 \begin{subequations}\label{a:ric}
	 	\begin{align}
	 		\label{a:ric:1}\norm{\nabla^2 f(\vz)}&\le 12+3\beta_1,\\
	 		\label{a:ric:2}\quad \vu^\hh \nabla^2 f(\vz) \vu&\ge \frac{\betat}{4}\norm{\vu}^2
	 	\end{align}
	 \end{subequations}
	holds for all  $\vz \in \C^n$ obeying
	\begin{eqnarray}
		\label{cond1:ric:a}\norm{\vz-\vz^*}&&\le \delta_1;\\
		\label{cond2:ric:a}\max\limits_{1\le j\le m} \aabs{\va_j^\hh \xkh{\vz-\vz^*}}&&\le \tilde{C}_{3} K \sqrt{\log m};
	\end{eqnarray}
	and for all $\vu=\begin{bmatrix}
			\vz_1-\vz_2\\
			\overline{\vz_1-\vz_2}
		\end{bmatrix}  \in \C^{2n}$ obeying
		\[
		\max\dkh{\norm{\vz_1-\vz^*},\norm{\vz_2-\vz^*}}\le \delta_1
		\]	
and $\vz_1$ is aligned with $\vz_2$.
	Here, $\tilde{C}_{3}, c_0>0$ are universal constants, and $\delta_1>0$ is some sufficiently small constant.
\end{lemma}
\begin{proof}
	See Section \ref{pf:ric:a}.
\end{proof}

The condition \eqref{a:ric} is referred to as the region of incoherence and contraction in \cite{Macong2020}, where \eqref{a:ric:1} represents the smoothness property and \eqref{a:ric:2} represents the restricted strong convexity.
With Lemma \ref{le:ric:a} in place, we next show that if the iterate $\vzt$ resides within the RIC, the subsequent iterate $\vztt$ exhibits a contracted distance from the true signal $\vzx$.

\begin{lemma} \label{le:error:a}
Suppose that $m\ge C_0\betat^{-1}\beta_2^{-1}K^8 n \log^3 m$ for some sufficiently large constant $C_0>0$, and the noise vector $\vxi$ satisfies \eqref{cond:noise:a}.  There exists an event that does not depend on $t$ and has probability at least $1-m\exp (-c_0 n)-O(m^{-20})$, such that when it happens and
\begin{equation} \label{eq:le32}
\dist{\vzt,\vzx} \le \delta_1 , \qquad  \qquad \max\limits_{1\le j\le m} \aabs{\va_j^\hh \xkh{\vztu-\vz^*}} \le \tilde{C}_{3} K \sqrt{\log m}
\end{equation}
hold for some some sufficiently small constant $\delta_1>0$ and a universal constant $\tilde{C}_{3}>0$,
one has
	\begin{equation*}
		\dist{\vztt,\vzx}\le (1-\frac{\betat}{8} \eta) \dist{\vzt,\vzx}+\eta\frac{\abs{\vone^\T \vxi}+C_{9} K^2\xkh{\sqrtn\norm{\vxi} + n\norminf{\vxi}}}{m}, 
	\end{equation*}
	provided the step size satisfies $0< \eta <\frac{\betat}{36(4+\beta_1)^{2}}$ where $C_{9}>0$ is a sufficiently large constant. 
	Here, $\vztu$ is defined in \eqref{def:ztilde:a}, and $\beta_1,\beta_2>0$ are parameters in Assumption A.
\end{lemma}
\begin{proof}
	See Section \ref{pf:error:a}.
\end{proof}

From Lemma \ref{le:error:a}, one sees that  if the condition \eqref{eq:le32} holds for the previous $t$ iterations, then
\begin{eqnarray}
	\dist{\vztt,\vzx} &\le & (1-\frac{\betat}{8} \eta) \dist{\vzt,\vzx}+\eta\frac{\abs{\vone^\T \vxi}+C_{9} K^2\xkh{\sqrtn\norm{\vxi} + n\norminf{\vxi}}}{m} \notag \\
	& \le& (1-\frac{\betat}{8} \eta)^{t+1} \dist{\vz^0,\vzx} + \frac{8\abs{\vone^\T \vxi} + 8 C_9 K^2\xkh{\sqrt{n}\norm{\vxi} + n\norminf{\vxi}}}{m\betat}, \label{eq:con17}
\end{eqnarray}
which implies the conclusion of Theorem \ref{th:A}.

It remains to demonstrate the validity of the condition \eqref{eq:le32} for all $0\le t\le t_0\le m^{10}$. Our proof will be inductive in nature.  Specifically, we prove that if  the condition \eqref{eq:le32} is satisfied for the iteration $\vzt$,  it will continue to hold for the subsequent iteration $\vztt$ with high probability. This in turn concludes the proof of Theorem \ref{th:A}, as long as the condition \eqref{eq:le32} is valid for the base case $\vz^0$.

For the first part of the condition \eqref{eq:le32},   it follows from \eqref{eq:con17} that if $ \dist{\vz^0,\vzx} \le \delta$ for a positive constant $\delta < \delta_1/2$, we have
\begin{equation}\label{cond:result:1}
	\dist{\vztt,\vzx}  \le \delta + \frac{8\abs{\vone^\T \vxi} + 8 C_9 K^2\xkh{\sqrt{n}\norm{\vxi} + n\norminf{\vxi}}}{m\betat} \le \delta_1,
\end{equation}
provided $m\gtrsim K^2\betat^{-1} n\log^3 m $ and $\vxi$ obeying the condition \eqref{cond:noise:a}.
However, verifying the second part of condition \eqref{eq:le32} is  challenging  due to the statistical dependence between $\vzt$ and the sampling vectors $\dkh{\va_j}$. To address this, we employ the leave-one-out technique.

\subsection{Leave-one-out Sequences for Auxiliary Purposes}
In this section, we aim to show that the second part of condition \eqref{eq:le32} holds true with high probability for all $0\le t\le t_0\le m^{10}$.  To decouple the  statistical dependence between \(\{\vzt\}_{t \ge 0}\) and the sampling vectors $\dkh{\va_j}$, we introduce an auxiliary sequence \(\{\vztl\}_{t \ge 0}\) for each $1\le l \le m$, which excludes the \(l\)-th  measurement from consideration.  Specifically,  for each $1\le l \le m$, we employ the following leave-one-out loss function
\begin{equation}\label{loss:leave:a}
	f^{(l)} (\vz) = \frac{1}{2m} \sum_{j\neq l} \xkh{\abs{\nj{\va_j,\vz}}^2 - y_j}^2,
\end{equation}
and the sequence \(\{\vztl\}_{t \ge 0}\) is constructed by running Algorithm \ref{alg:a:leave one out} with respect to $f^{(l)} (\vz)$.
\begin{algorithm}[H]
	\caption{The $l$-th iterative sequence for Assumption A}\label{alg:a:leave one out}
	\begin{algorithmic}[0]
		\State \textbf{Input:} $\dkh{\va_j}_{1\le j\le m,j\ne l}$, $\dkh{y_j}_{1\le j\le m,j\ne l}$, $\beta_1$, $\beta_2$ and $\eta_0$.
		\State \textbf{Spectral initialization: } Let $\check{\vz}^{0,\lkh}$ be the leading singular vector  of
		\begin{equation*}
			\mm^\lkh = \frac{1}{\beta_2}\mm_0^\lkh+\frac{2-\beta_1-\beta_2}{\beta_1(2-\beta_2)}\mdiag\xkh{\mm_0^\lkh}-\frac{2-2\beta_2}{\beta_2(2-\beta_2)}\real{\mm_0^\lkh} + \frac{\beta_1-1}{\beta_1} \gl\mi
		\end{equation*}
		where
		\[
		\mm_0^\lkh = \frac{1}{m}\sum_{j=1,j\ne l}^{m} y_j\va_j\va_j^\hh \qquad \mbox{and} \qquad  \gl = \frac{1}{m}\sum_{j=1,j\ne l}^{m} y_j.
		\]
		Set $\vz^{0,\lkh} = \sqrt{\gl}\check{\vz}^{0,\lkh}$.
		\State \textbf{Gradient updates: for} $t = 0,1,2\cdots,T-1$ \textbf{do}
		\begin{equation} \label{eq:uprulel}
			\vzttl = \vztl - \eta \nabla_{\vz} f^{(l)} (\vztl),
		\end{equation}
		where $\eta = \eta_0 /\gl$.
	\end{algorithmic}
\end{algorithm}

Next, we will demonstrate the incoherence condition  \eqref{eq:le32}  by using the following inductive hypotheses:
\begin{subequations}\label{induct:a}
	\begin{align}
		\label{induct:err:a}\dist{\vzt,\vzx}&\le C_{3}+ C_{4} \frac{\abs{\vone^\T \vxi}}{m\betat} + C_{5} K^2\frac{\sqrt{n}\norm{\vxi} + n\norminf{\vxi}}{m\betat},\\
		\label{induct:ztl:a}\max_{1\le l\le m}\dist{\vztl,\vztu}&\le C_{6} \frac{K^7\sqrtnlmmm}{m\beta_2}\xkh{1+\frac{\abs{\vone^\T \vxi}+\sqrtn \norm{\vxi}}{m} + \frac{\norminf{\vxi}}{K^4 \log m}},\\
		\label{induct:ztlu:a}\max_{1\le l\le m}\norm{\vztlu-\vztu}&\lesssim C_{6} \frac{K^7\sqrtnlmmm}{m\beta_2}\xkh{1+\frac{\abs{\vone^\T \vxi}+\sqrtn \norm{\vxi}}{m} + \frac{\norminf{\vxi}}{K^4 \log m}},\\
		\label{induct:inh:a}\max_{1\le l\le m}\abs{\va_l^\hh \xkh{\vztu-\vzx}}&\le C_{7}K\sqrtlm \xkh{1+\frac{\abs{\vone^\T \vxi}+ K^2\sqrtn\norm{\vxi} + K^4 n\norminf{\vxi}}{m\betat}}.
	\end{align}
\end{subequations}
Here, $\vztu$ is defined in \eqref{def:ztilde:a}, and $C_{4}>0$ is a universal constant,  $C_{3}>0$ is a sufficiently small constant, $C_{5},C_{6},C_{7}>0$ are sufficiently large constants.  We next  demonstrate that when \eqref{induct:a} holds true up to  the $t$-th iteration,  the inductive hypotheses \eqref{induct:a} still holds for the $(t+1)$-th iteration.  It is easy to see that \eqref{induct:err:a} is a direct consequence of \eqref{cond:result:1}. Subsequently, \eqref{induct:ztl:a} and \eqref{induct:ztlu:a} are justified by the following lemma.
\begin{lemma} \label{le:ztl:a}
	Suppose that $m\ge C_0\betat^{-1}\beta_2^{-1}K^8 n \log^3 m$ for some sufficiently large constant $C_0>0$, the noise vector $\vxi$ satisfies \eqref{cond:noise:a}, and the step size $\eta>0$ is some sufficiently small constant. Let $\mathcal{E}_t$ be the event where \eqref{induct:err:a}-\eqref{induct:inh:a} hold for the $t$-th iteration. Then there exist an event $\mathcal{E}_{t+1} \subseteq \mathcal{E}_t$ obeying $\PP\xkh{\mathcal{E}_t \cap \mathcal{E}_{t+1}^C}\le m\exp (-c_0 n) + O(m^{-20})$, one has
	\begin{equation} \label{eq:le331}
		\max\limits_{1\le l\le m}\dist{\vzttl,\vzttu}\le C_{6} \frac{K^7\sqrtnlmmm}{m\beta_2}\xkh{1+\frac{\abs{\vone^\T \vxi}+\sqrtn \norm{\vxi}}{m} + \frac{\norminf{\vxi}}{K^4\log m}},
	\end{equation}
	and
	\begin{equation}  \label{eq:le332}
		\max\limits_{1\le l\le m}\norm{\vzttlu-\vzttu}\lesssim C_{6} \frac{K^7\sqrtnlmmm}{m\beta_2}\xkh{1+\frac{\abs{\vone^\T \vxi}+\sqrtn \norm{\vxi}}{m} + \frac{\norminf{\vxi}}{K^4\log m}}.
	\end{equation}
\end{lemma}
\begin{proof}
	See Section \ref{pf:ztl:a}.
\end{proof}

With Lemma \ref{le:ztl:a} in place,  the inductive hypotheses \eqref{induct:inh:a} can be readily verified. Specifically,  it follows from Lemma \ref{le:order:statistics} that with probability at least $1-O(m^{-20})$,  one has
 \begin{align}\label{ineq:pf:inh:a}
		\maxm{1\le l\le m} \abs{\va_l^\hh\xkh{\vzttlu-\vzx}}&\le \maxm{1\le l\le m} c_2 K\sqrtlm \norm{\vzttlu-\vzx}\notag\\
		&\le \maxm{1\le l\le m} c_2 K\sqrtlm \xkh{\norm{\vzttlu-\vzttu} + \norm{\vzttu-\vzx}}
\end{align}
for a universal constant $c_2>0$.
Then with probability at least $1-m\exp(-c_0 n) - O(m^{-20})$,  we obtain
\begin{eqnarray}
	&&\maxm{1\le l\le m} \abs{\va_l^\hh\xkh{\vzttu-\vzx}} \label{eq:alzttrue}\\
	& \stackrel{\mbox{(i)}}{\le}  & \maxm{1\le l\le m}  \abs{\va_l^\hh\xkh{\vzttu-\vzttlu}} + \maxm{1\le l\le m} \abs{\va_l^\hh\xkh{\vzttlu-\vzx}} \notag \\
	& \stackrel{\mbox{(ii)}}{\le}  &  \maxm{1\le l\le m} \norm{ \va_l } \norm{\vzttu-\vzttlu}+ \maxm{1\le l\le m} \abs{\va_l^\hh\xkh{\vzttlu-\vzx}} \notag \\
	&\stackrel{\mbox{(iii)}}{\le}& \maxm{1\le l\le m} \xkh{c_3 K^2\sqrtn + c_2 K\sqrtlm } \norm{\vzttu-\vzttlu} +  c_2 K\sqrtlm \norm{\vzttu-\vzx} \notag\\
	&\stackrel{\mbox{(iv)}}{\le} & 2c_3 C_{6} K^9 \frac{n\sqrt{\log^3 m}}{m\beta_2} \xkh{1+\frac{\abs{\vone^\T \vxi}+\sqrtn \norm{\vxi}}{m} + \frac{\norminf{\vxi}}{K^4 \log m}} \notag\\
	&&+  c_2 K\sqrtlm \xkh{C_{3}+ C_{4} \frac{\abs{\vone^\T \vxi}}{m\betat} + C_{5} K^2\frac{\sqrt{n}\norm{\vxi} + n\norminf{\vxi}}{m\betat}} \notag \\
	&\stackrel{\mbox{(v)}}{\le}  & C_{7}K\sqrtlm \xkh{1+\frac{\abs{\vone^\T \vxi}+ K^2\sqrtn\norm{\vxi} + K^4 n\norminf{\vxi}}{m\betat}}. \notag
\end{eqnarray}
Here, (i) comes from the triangle inequality,  (ii) arises from Cauchy-Schwarz, (iii) follows from Lemma \ref{le:order:statistics} and \eqref{ineq:pf:inh:a}.
Besides, (iv) holds because of Lemma \ref{le:ztl:a} and  the inductive hypotheses \eqref{induct:err:a}.
Lastly, (v) holds since $C_{7}>0$ is sufficiently large and $m\gtrsim K^8 n \log^3 m$ and $\betat = \min\dkh{\beta_1,\beta_2}\le 1$.

\subsection{Spectral Initialization}\label{sec:initial of A}
To complete the induction steps, we need to verify the induction hypotheses for the base cases. Specifically, we must show that the spectral initializations $\vz^0$ and $\vz^{0,\lkh}$ satisfy the induction hypothesis \eqref{induct:a} at $t=0$.
First, we demonstrate that the spectral initial point $\vz^0$ is sufficiently close to $\vzx$, as shown below.

\begin{lemma} \label{le:initial1:a}
Assume that the noise vector $\vxi$ satisfies \eqref{cond:noise:a}.  For any $0<\delta\le 1$, with probability at least $1 - m\exp (-c_0 n) - O(m^{-10})$, we have
\[
		\dist{\vz^0,\vzx} \le \delta \qquad \mbox{and} \qquad
		\maxm{1\le l\le m}\dist{\vz^{0,\lkh},\vzx} \le \delta,
\]
provided $m\ge C_0(1+\beta_1)^2\betat^{-2}\delta^{-2} K^8 n \log^3 m$ for some absolute constant $C_0>0$.
\end{lemma}
\begin{proof}
	This follows from Wedin's sin$\Theta$ theorem \cite{WedinInitial}. See Section \ref{pf:initial1:a}.
\end{proof}

\begin{lemma} \label{le:initial2:a}
	With probability at least $1 - m\exp (-c_0 n) - O(m^{-10})$, one has
\[
		\maxm{1\le l\le m} \dist{\vz^{0,\lkh},\tilde{\vz}^0} \le C_{6} \frac{K^7\sqrtnlmmm + K^7\sqrtnlm\norminf{\vxi}}{m\beta_2}
\]
and
\[
		\maxm{1\le l\le m} \norm{\tilde{\vz}^{0,\lkh}-\tilde{\vz}^0} \lesssim C_{6} \frac{K^7\sqrtnlmmm + K^7\sqrtnlm\norminf{\vxi}}{m\beta_2},
\]
provided  $m\ge C_0 (1+\beta_1)^2 \betat^{-2} K^8 n \log^3 m$ for some sufficiently large constant $C_0>0$.
\end{lemma}
\begin{proof}
	See Section \ref{pf:initial2:a}.
\end{proof}

The final claim \eqref{induct:inh:a} for $t=0$ can be proved using the same argument as in deriving \eqref{eq:alzttrue} and hence is omitted.

\subsection{Summary of the above lemmas}\label{subsec:sum of A}
Let's summarize the proof of Thoerem \ref{th:A} below. First, the guarantees for spectral initialization are established by Lemma  \ref{le:initial1:a} and Lemma \ref{le:initial2:a}, indicating that condition \eqref{induct:a} holds for $\vz^0$ with probability at least $1 - m\exp (-c_0 n) - O(m^{-10})$. Next, using Lemma  \ref{le:ztl:a} along with    \eqref{cond:result:1} and \eqref{eq:alzttrue}, we observe that if the  $t$-th iteration $\vzt$ satisfies \eqref{induct:a}, then it will continue to hold for $\vztt$ with probability at least $1 - m\exp (-c_0 n) - O(m^{-20})$. Taking the union bound over $0\le t\le t_0\le m^{10}$, we arrive at  the conclusion that
\begin{equation*}
	\dist{\vzt,\vzx} \le (1-\frac{\betat}{8} \eta)^{t} \dist{\vz^0,\vzx} + \frac{8\abs{\vone^\T \vxi} + 8 C_9 K^2\xkh{\sqrt{n}\norm{\vxi} + n\norminf{\vxi}}}{m\betat}
\end{equation*}
holds with probability at least $1 - m\exp (-c_0 n) - O(m^{-10})$.

\section{Proof of Theorem \ref{th:B} }\label{sec:pf:B}
The proof of Theorem \ref{th:B} follows a similar approach to that of Theorem \ref{th:A} given in Section \ref{sec:pf:A}. Specifically,  we first show in Lemma \ref{le:ric:b} that,  under Assumption B,   the loss function $f$  satisfies the restricted strong convexity and smoothness conditions within a defined region of incoherence and contraction. The main result is then established by ensuring that the iterations generated by Algorithm \ref{alg:2} remain within this region throughout. To achieve this, we apply the leave-one-out technique combined with mathematical induction to confirm that the iterations indeed stay within this region. Notably, only Lemma \ref{le:ric:b} differs substantially from Lemma  \ref{le:ric:a}, for which we provide a detailed proof. The remaining lemmas closely mirror those under Assumption A, so their proofs are omitted for brevity.
Without loss of generality, we assume throughout this section that $\norm{z^*}=1$.

\subsection{Characterizing Local Geometry in the RIC and Error Contraction}
First, we establish the smoothness and restricted strong convexity properties of the loss function $f$ under Assumption B within a specified region of incoherence and contraction, as stated below.
\begin{lemma}\label{le:ric:b}
	Suppose that $\norm{\vzx}=1$, Assumption B holds true with $\mu\le 1/10$, and the noise vector $\vxi$ satisfies
	\begin{equation}\label{cond:noise:b}
		\aabs{\vone^\T \vxi}\le \tilde{C}_{11} m,\ \norm{\vxi}\lesssim \sqrtm ,\ \mbox{and}\ \norminf{\vxi}\lesssim \log m
	\end{equation}
	for some sufficiently small constant $\tilde{C}_{11}>0$.
	Assume that $m\ge C_{10}\beta_2^{-2} K^8 n \log^3 m$ for some sufficiently large constant $C_{10}>0$. Then with probability at least $1 - m\exp (-c_0 n)-O(m^{-10})$ for some universal constant $c_{0}>0$, the Wirtinger Hessian $\nabla^2 f(\vz)$ obeys
	\begin{equation}\label{b:ric}
		\norm{\nabla^2 f(\vz)}\le 9 \quad and \quad \vu^\hh \nabla^2 f(\vz) \vu\ge \frac{\beta_2}{4}\norm{\vu}^2
	\end{equation}
	simultaneously for all
	\begin{equation*}
		\vz \quad \mbox{and} \quad \vu=
		\begin{bmatrix}
			\vz_1-\vz_2\\
			\overline{\vz_1-\vz_2}
		\end{bmatrix},
	\end{equation*}
	where $\vz$ satisfies
	\begin{eqnarray}
		\label{cond1:ric:b}\norm{\vz-\vz^*}&\le&  \delta_{11};\\
		\label{cond2:ric:b}\max\limits_{1\le j\le m}\aabs{\va_j^\hh \xkh{\vz-\vz^*}}&\le &  \tilde{C}_{13} K \sqrt{\log m};
	\end{eqnarray}
	$\vz_1$ is aligned with $\vz_2$, and they satisfy
	\begin{equation} \label{eq:z1z2}
		\max\dkh{\norm{\vz_1-\vz^*},\norm{\vz_2-\vz^*}}\le \delta_{11}.
	\end{equation}
	Here, $\tilde{C}_{13}>0$ is a sufficiently large constant, and $\delta_{11}>0$ is a sufficiently small constant.
\end{lemma}
\begin{proof}
	See Section \ref{pf:ric:b}.
\end{proof}
Next, we characterize the contraction of \(l_2\) error.

\begin{lemma} \label{le:error:b}
	Suppose that $m\ge C_0 K^8 n \log^3 m$ for some sufficiently large constant $C_0>0$, and the step size $\eta>0$ is some sufficiently small constant. The noise vector $\vxi$ satisfies \eqref{cond:noise:b}. Then there exists an event that does not depend on $t$ and has probability at least $1-m\exp (-c_0 n)-O(m^{-20})$, such that when conditions \eqref{cond1:ric:b} and \eqref{cond2:ric:b} hold for $\vz=\vztu$, one has
	\begin{equation*}
		\dist{\vztt,\vzx}\le (1-\frac{\beta_2}{8} \eta) \dist{\vzt,\vzx}+\eta\frac{\abs{\vone^\T \vxi}+C_{19} K^2\xkh{\sqrt{n}\norm{\vxi} + n\norminf{\vxi}}}{m}
	\end{equation*}
	for some sufficiently large constant $C_{19}>0$. Here, $\beta_2>0$ is the  parameter given in Assumption B, and $\vztu$ is defined in \eqref{def:ztilde:a}.
\end{lemma}
\begin{proof}
	The proof follows from the same argument as in the proof of Lemma \ref{le:error:a}, and is thus omitted here for brevity.
\end{proof}

If the conditions \eqref{cond1:ric:b} and \eqref{cond2:ric:b} hold for the previous $t$ iterations, it implies that
\begin{align*}
	\dist{\vztt,\vzx}\le& (1-\frac{\beta_2}{8} \eta) \dist{\vzt,\vzx}+\eta\frac{\abs{\vone^\T \vxi}+C_{19} K^2\xkh{\sqrt{n}\norm{\vxi} + n\norminf{\vxi}}}{m}\\
	\le& (1-\frac{\beta_2}{8} \eta)^{t+1} \dist{\vz^0,\vzx} + \frac{8\abs{\vone^\T \vxi}+8C_{19}K^2\xkh{\sqrt{n}\norm{\vxi} + n\norminf{\vxi}}}{m\beta_2}.
\end{align*}
This immediately gives
\begin{align*}
	\dist{\vztt,\vzx} \le \delta_{11},
\end{align*}
provided that $m\gtrsim K^2 \betat^{-1} n \log^3 m$, $\dist{\vz^0,\vzx} \le \delta_{11}/2$, and $\vxi$ obeys  \eqref{cond:noise:b}.

\subsection{Leave-one-out Sequences and Approximate Independence}
Here, we employ the following loss function
\begin{equation}\label{loss:leave:b}
	f^{(l)} (\vz) = \frac{1}{2m} \sum_{j,j\neq l} \xkh{\abs{\nj{\va_j,\vz}}^2 - y_j}^2.
\end{equation}
Subsequently, the sequence \(\{\vztl \}_{t \ge 0}\) is then generated by applying Algorithm \ref{alg:b:leave one out} to the loss function \eqref{loss:leave:b}.

\begin{algorithm}[H]
	\caption{The $l$-th iterative sequence for Assumption B}\label{alg:b:leave one out}
	\begin{algorithmic}[0]
		\State \textbf{Input:} $\dkh{\va_j}_{1\le j\le m,j\ne l}$, $\dkh{y_j}_{1\le j\le m,j\ne l}$, $\beta_2$ and $\eta_0$.
		\State \textbf{Spectral initialization: } Let $\check{\vz}^{0,\lkh}$ be the left singular vector for the leading singular value of \begin{equation*}
			\mm^\lkh = \frac{1}{\beta_2}\mm_0^\lkh-\frac{2-2\beta_2}{\beta_2(2-\beta_2)}\real{\mm_0^\lkh} + \frac{1-\beta_2}{2-\beta_2} \gl \mi
		\end{equation*}
		with
		 \[
		 \mm_0^\lkh = \frac{1}{m}\sum_{j=1,j\ne l}^{m} y_j\va_j\va_j^\hh \quad \mbox{and} \quad  \gl = \frac{1}{m}\sum_{j=1,j\ne l}^{m} y_j.
		 \]
		Set $\vz^{0,\lkh} = \sqrt{\gl}\check{\vz}^{0,\lkh}$, $\eta = \eta_0 /\gl $.
		\State \textbf{Gradient updates: for} $t = 0,1,2\cdots,T-1$ \textbf{do}\begin{equation*}
			\vzttl = \vztl - \eta \nabla_{\vz} f^{(l)} (\vztl).
		\end{equation*}
	\end{algorithmic}
\end{algorithm}

Next, the incoherence condition between the iterations $\dkh{\vzt}$ and the measurement vectors $\dkh{\va_j}$ is described by the following inductive assumptions:
\begin{subequations}\label{induct:b}
	\begin{align}
		\label{induct:err:b}\dist{\vzt,\vzx}&\le C_{13}+ C_{14} \frac{\abs{\vone^\T \vxi}}{m\beta_2} + C_{15} K^2\frac{\sqrt{n} \norm{\vxi} + n \norminf{\vxi}}{m\beta_2},\\
		\label{induct:ztl:b}\max_{1\le l\le m}\dist{\vztl,\vztu}&\le C_{16} \frac{K^7\sqrtnlmmm}{m\beta_2}\xkh{1+\frac{\abs{\vone^\T \vxi}+\sqrtn\norm{\vxi}}{m} + \frac{\norminf{\vxi}}{K^4 \log m}},\\
		\max_{1\le l\le m}\norm{\vztlu-\vztu}&\lesssim C_{16} \frac{K^7\sqrtnlmmm}{m\beta_2}\xkh{1+\frac{\abs{\vone^\T \vxi}+\sqrtn\norm{\vxi}}{m} + \frac{\norminf{\vxi}}{K^4 \log m}},\\
		\label{induct:inh:b}\max_{1\le l\le m}\abs{\va_l^\hh \xkh{\vztu-\vzx}}&\le C_{17}K\sqrtlm \xkh{1 + \frac{\abs{\vone^\T \vxi} + K^2\sqrtn \norm{\vxi} + K^4 n\norminf{\vxi}}{m\beta_2}}.
	\end{align}
\end{subequations}
Here, $C_{14}>0$ is a universal constant, $C_{13}>0$ is a sufficiently small constant, and $C_{15},C_{16},C_{17}>0$ are sufficiently large constants.

Our aim is to demonstrate that these assumptions hold inductively from the $t$-th step to the $(t+1)$-th step with high probability,  as established in the following lemma.
\begin{lemma} \label{le:ztl:b}
	Suppose that $m\ge C_{10} K^8 n \log^3 m$ for some sufficiently large constant $C_{10}>0$, and the step size $\eta>0$ is some sufficiently small constant. The noise vector $\vxi$ satisfies \eqref{cond:noise:b}. Let $\mathcal{E}_t$ be the event where \eqref{induct:err:b}-\eqref{induct:inh:b} hold for the $t$-th iteration. Then there exist an event $\mathcal{E}_{t+1,1} \subseteq \mathcal{E}_t$ obeying $\PP\xkh{\mathcal{E}_t \cap \mathcal{E}_{t+1,1}^C}\le m\exp (-c_0 n) + O(m^{-20})$ such that it holds
	\begin{equation*}
		\max\limits_{1\le l\le m}\dist{\vzttl,\vzttu}\le C_{16} \frac{K^7\sqrtnlmmm}{m\beta_2}\xkh{1+\frac{\abs{\vone^\T \vxi}+\sqrtn\norm{\vxi}}{m} + \frac{\norminf{\vxi}}{K^4 \log m}}
	\end{equation*}
	and \begin{equation*}
		\max\limits_{1\le l\le m}\norm{\vzttlu-\vzttu}\lesssim C_{16} \frac{K^7\sqrtnlmmm}{m\beta_2}\xkh{1+\frac{\abs{\vone^\T \vxi}+\sqrtn\norm{\vxi}}{m} + \frac{\norminf{\vxi}}{K^4 \log m}}.
	\end{equation*}
	Furthermore,   there exist an event $\mathcal{E}_{t+1,2} \subseteq \mathcal{E}_t$ obeying $\PP\xkh{\mathcal{E}_t \cap \mathcal{E}_{t+1,2}^C}\le m\exp (-c_0 n) + O(m^{-20})$ such that it holds
	\begin{equation*}
		\max_{1\le l\le m}\abs{\va_l^\hh \xkh{\vzttu-\vzx}} \le C_{17}K\sqrtlm \xkh{1 + \frac{\abs{\vone^\T \vxi} + K^2\sqrtn \norm{\vxi} + K^4 n\norminf{\vxi}}{m\beta_2}}.
	\end{equation*}
\end{lemma}
This lemma  can be established using the same arguments as those for Lemma \ref{le:ztl:a} and \eqref{eq:alzttrue};  we omit the details here for brevity.

\subsection{Spectral Initialization}
It remains to verify the induction hypotheses for the base cases, i.e., $t=0$.

\begin{lemma} \label{le:initial1:b}
	Suppose that the noise vector $\vxi$ satisfies \eqref{cond:noise:b}. For any $0 < \delta \le 1$,  with probability at least $1 - m\exp (-c_0 n) - O(m^{-10})$, we have
	\begin{align*}
		\dist{\vz^0,\vzx} \le& \delta, \\
		\maxm{1\le l\le m} \dist{\vz^{0,\lkh},\vzx} \le& \delta,
	\end{align*}
	as long as $m\ge C_{10}\beta_2^{-2}\delta^{-2} K^8 n \log^3 m$ for some sufficiently large constant $C_{10}>0$.
\end{lemma}


\begin{lemma} \label{le:initial2:b}
	With probability at least $1 - m\exp (-c_0 n) - O(m^{-10})$, one has
	\begin{align*}
		\maxm{1\le l\le m} \dist{\vz^{0,\lkh},\tilde{\vz}^0} \le C_{16} \frac{K^7\sqrtnlmmm + K^7\sqrtnlm\norminf{\vxi}}{m\beta_2}, \\
		\maxm{1\le l\le m} \norm{\tilde{\vz}^{0,\lkh}-\tilde{\vz}^0} \lesssim C_{16} \frac{K^7\sqrtnlmmm + K^7\sqrtnlm\norminf{\vxi}}{m\beta_2},
	\end{align*}
	provided that $m\ge C_{10}\beta_2^{-2} K^8 n \log^3 m$ for some sufficiently large constant $C_{10}>0$.
\end{lemma}

The proofs of the two lemmas above are identical to those of Lemma \ref{le:initial1:a} and Lemma \ref{le:initial2:a}, respectively, and are thus omitted here.

Similar to Section \ref{subsec:sum of A}, one can summarize the above lemmas to arrive at the conclusion of Theorem \ref{th:B}, namely,
\begin{align*}
	\dist{\vzt,\vzx}\le (1-\frac{\beta_2 c}{8})^t C_{11}\norm{\vzx}+\frac{8\abs{\vone^\T \vxi}+C_{12}K^2\xkh{\sqrt{n} \norm{\vxi} + n \norminf{\vxi}}}{m\beta_2\norm{\vzx}}
\end{align*}
holds with probability exceeding $1-m\exp(-c_0 n) - O(m^{-10})$, provided \(m \ge CK^8n \log^3 m\).

\section{Results for the real-valued case} \label{sec:realcase}
In this section, we give a brief description  demonstrating that the results stated in Theorem \ref{th:A} and Theorem \ref{th:B} also hold in the real-valued case.  These results can be established using techniques similar to those in the complex case, so we omit the detailed proofs. Suppose the random measurement vectors $\va_j\in \R^n$  are  independent copies of a random vector $\va=(a_1,a_2,\cdots,a_m)^\T\in \R^n$,  and assume that the unknown object $\vzx\in\R^n$ obeys one of the following assumptions: \medskip
\paragraph{\textit{Assumption C}}\label{para:c} The entries of $\va=(a_1,a_2,\cdots,a_m)^\T \in \R^n$ are assumed to be i.i.d. subguassian random variables with $\maxm{1\le j\le m} \normsg{a_j}\le K$,  $\E a_1=0$, $\E a_1^2=1$, $\E a_1^4= 1+\beta_1$ for some $\beta_1 \in (0,+\infty)$. \medskip
\paragraph{\textit{Assumption D}} The entries of $\va=(a_1,a_2,\cdots,a_m)^\T \in \R^n$ are assumed to be i.i.d. subguassian random variables with $\maxm{1\le j\le m} \normsg{a_j}\le K$,  $\E a_1=0$, $\E a_1^2 = \E a_1^4=1$. Assume that $\vzx\in\R^n$ satisfies $\norminf{\vzx}^2/\norm{\vzx}^2\le \mu<1$ for some constant $\mu>0$. \bigskip

The algorithms for recovering the underlying signal $\vzx$ from noisy phase retrieval under each assumption are provided in Algorithm \ref{alg:3} and Algorithm \ref{alg:4}, respectively.

\begin{algorithm}[H]
	\caption{Vanilla gradient descent for phase retrieval (with Assumption C)}\label{alg:3}
	\begin{algorithmic}[0]
		\State \textbf{Input:} $\dkh{\va_j}_{1\le j\le m}$, $\dkh{y_j}_{1\le j\le m}$, $\beta_1$ and $\eta_0$.
		\State \textbf{Spectral initialization: } Let $\check{\vz}^0$ be the left singular vector for the leading singular value of \begin{equation*}
			\mm = \mm_0+\frac{2-\beta_1}{\beta_1}\mdiag\xkh{\mm_0} + \frac{\beta_1-2}{\beta_1}\gamma \mi
		\end{equation*}
		where
		\[
		\mm_0 = \frac{1}{m}\sum_{j=1}^{m} y_j\va_j\va_j^\hh \quad \mbox{and} \quad  \gamma = \frac{1}{m}\sum_{j=1}^{m} y_j.
		\]
		Set $\vz^0 = \sqrt{\gamma}\check{\vz}^0$, and $\eta = \eta_0 /\gamma$.
		\State \textbf{Gradient updates: for} $t = 0,1,2\cdots,T-1$ \textbf{do}\begin{equation*}
			\vztt = \vzt - \eta \nabla f(\vzt).
		\end{equation*}
	\end{algorithmic}
\end{algorithm}

The convergence result for Assumption C is stated as follows.

\begin{theorem}\label{th:C}
	For any fixed $\vzx\in\R^n$, suppose that Assumption C holds and the noise $\vxi$ satisfies $\abs{\vone^\T \vxi}\le \tilde{C}_{21} m\norm{\vzx}^2,\ \norm{\vxi}\lesssim \sqrtm \norm{\vzx}^2,\ \mbox{and}\ \norminf{\vxi}\lesssim \log m\norm{\vzx}^2$. Assume that \(m \ge CK^8n \log^3 m\).
	Then with probability at least $1 - m\exp (-c_0 n)-O(m^{-10})$ the iterations in Algorithm \ref{alg:3} with  step size \(\eta = c/\norm{\vzx}^2\) obey
	\begin{align*}
		\min\dkh{\norm{\vzt-\vzx}, \norm{\vzt+\vzx}} &\le (1-\frac{\betat c}{4})^t C_{21}\norm{\vzx}+ \frac{4\abs{\vone^\T \vxi}+C_{22}K^2\xkh{\sqrt{n} \norm{\vxi} + n \norminf{\vxi}}}{m\betat\norm{\vzx}}
	\end{align*}
	simultaneously for all $0\le t\le t_0\le m^{10}$.  Here, $\betat=\min\dkh{\beta_1, 2}$, $C_{22}>0$ is a universal constant, and
	\begin{equation*}
		C = \frac{C_{20}(1+\beta_1)^4}{\betat^4}, \quad C_{21}=\frac{C_{21}'\betat}{1+\beta_1}, \quad \tilde{C}_{21}=\frac{C_{21}'\betat^2}{8(1+\beta_1)^2}, \quad c \le \frac{1}{12+3\beta_1},
	\end{equation*}
	where \(C_{20},C_{21}'>0\) are absolute constants.
\end{theorem}

\begin{algorithm}[H]
	\caption{Vanilla gradient descent for phase retrieval (with Assumption D)}\label{alg:4}
	\begin{algorithmic}[0]
		\State \textbf{Input:} $\dkh{\va_j}_{1\le j\le m}$, $\dkh{y_j}_{1\le j\le m}$ and $\eta_0$.
		\State \textbf{Spectral initialization: } Let $\check{\vz}^0$ be the left singular vector for the leading singular value of
		\[
		\mm = \frac{1}{m}\sum_{j=1}^{m} y_j\va_j\va_j^\hh.
		\]
		Set $\vz^0 = \sqrt{\gamma}\check{\vz}^0$, and $\eta = \eta_0 /\gamma$.
		\State \textbf{Gradient updates: for} $t = 0,1,2\cdots,T-1$ \textbf{do}\begin{equation*}
			\vztt = \vzt - \eta \nabla f(\vzt).
		\end{equation*}
	\end{algorithmic}
\end{algorithm}

The result for Assumption D is:
\begin{theorem}\label{th:D}
	For any fixed $\vzx\in\R^n$, suppose that Assumption D holds with $\mu\le 1/10$ and the noise $\vxi$ satisfies $\abs{\vone^\T \vxi}\le C_{31}m\norm{\vzx}^2/5,\ \norm{\vxi}\lesssim \sqrtm \norm{\vzx}^2$, and $\norminf{\vxi}\lesssim \log m\norm{\vzx}^2$. Assume that  \(m \ge CK^8n \log^3 m\) for some sufficiently large constant \(C>0\), then with probability at least $1 - m\exp (-c_{0} n)-O(m^{-10})$  the iterations in Algorithm \ref{alg:4} with step size \(\eta = c/\norm{\vzx}^2\) obey
	\begin{align*}
		\min\dkh{\norm{\vzt-\vzx}, \norm{\vzt+\vzx}} \le& (1-\frac{c}{4})^t C_{31}\norm{\vzx}+ \frac{4\abs{\vone^\T \vxi}+C_{32}K^2\xkh{\sqrt{n} \norm{\vxi} + n \norminf{\vxi}}}{m\norm{\vzx}}
	\end{align*}
	simultaneously for all $0\le t\le t_0\le m^{10}$. Here, $C, C_{31}, C_{32},c>0$ are some absolute constants, and $c\le 1/9$.
\end{theorem}



\section{Numerical Experiments}\label{sec:experiments}
In this section, a series of numerical experiments are conducted to validate the effectiveness of our results.  Several measurement models are considered:
\begin{itemize}
\item[(1)] {\bf Uniform Distribution}: Each element of $\va_j \in \C^n$ is drawn i.i.d. from a uniform distribution, specifically,  $a_{j,k} \sim \frac{\sqrt{6}}{2}\mathcal{U}[-1,1]+ i \frac{\sqrt{6}}{2}\mathcal{U}[-1,1]$ for all $j=1,\ldots,m$ and $k=1,\ldots,n$.
\item[(2)] {\bf Discrete Uniform Distribution}:   The entries of $\va_j \in \C^n$  follow a discrete uniform distribution:
\begin{equation*}
	a_{j,k}=\left\{ \begin{aligned}
		+1&\quad \rm{with\ prob.}\ 1/4\\
		-1&\quad \rm{with\ prob.}\ 1/4\\
		+i&\quad \rm{with\ prob.}\ 1/4\\
		-i&\quad \rm{with\ prob.}\ 1/4.
	\end{aligned}\right.
\end{equation*}
\item[(3)] {\bf Gaussian Distribution}:  the measurement vectors $\va_j \sim \frac{1}{\sqrttwo}\mathcal{N}\xkh{\mzero,\mi_n}+i\frac{1}{\sqrttwo}\mathcal{N}\xkh{\mzero,\mi_n}$.
\end{itemize}
The underlying signal $\vzx$ is generated according to the Gaussian distribution, and the noise vector is set to be $\xi_j\stackrel{\text{i.i.d}}{\sim}\mathcal{N}\xkh{0,\sigma^2}$. To evaluate the performance of algorithms, we examine the relative error, defined as
\[
\mbox{Relative error}:=\frac{\dist{\vzt,\vzx}}{\norm{\vzx}},
\]
where $\dist{\vzt,\vzx} = \minm{\phi\in\R} \norm{\e^{i\phi}\vzt-\vzx}$.

The first experiment  aims to report that our algorithms can recover the underlying signal $\vzx$  exactly when the measurements $y_j$ are noise-free.
Two sets of experiments are conducted here for the aforementioned scenarios where the observation vectors follow the uniform distribution and the discrete distribution, respectively. We set $n\in\{100, 200, 500, 1000\}$, with $m = 10n$, and the maximum number of iterations is $1000$. The step size  is set to be $\eta = \eta_0/\xkh{\sum_j y_j/m}$ with $\eta_0=0.2$. Figure \ref{S1} illustrates the relative error versus the number of iterations, indicating that our algorithms converge linearly to the underlying signal.

\begin{figure}[t]
	\setlength\tabcolsep{1pt}
	\centering
	\begin{tabular}{cc}
		\includegraphics[width=0.49\textwidth]{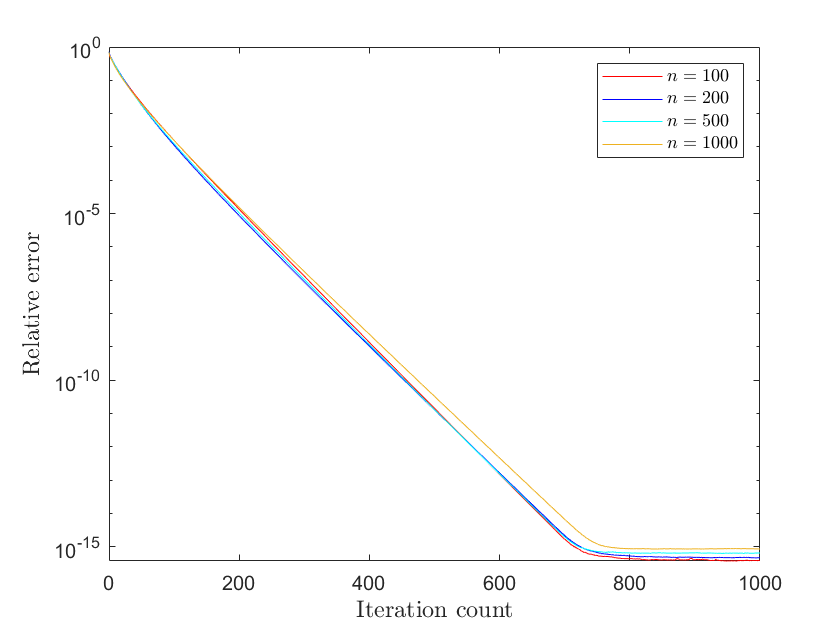}
		&\includegraphics[width=0.49\textwidth]{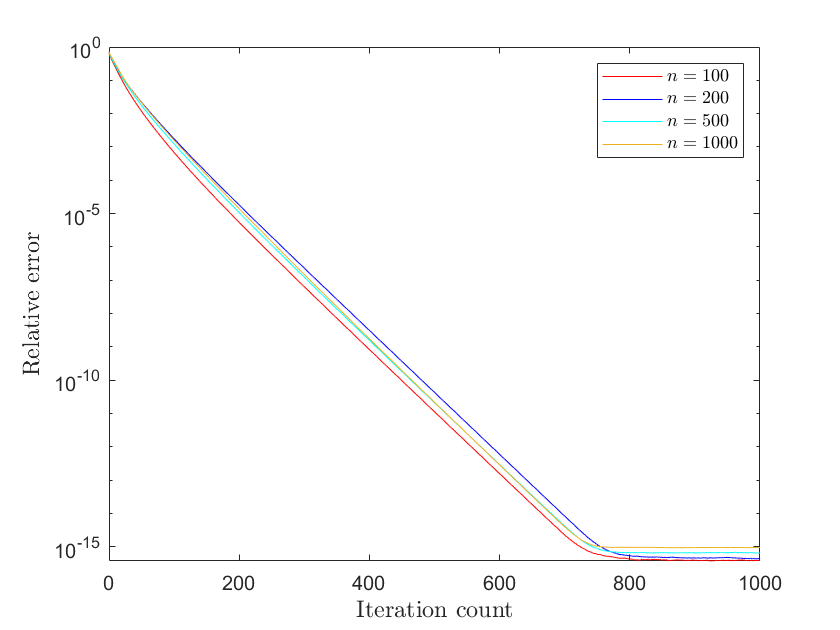}\\
		{\footnotesize (a) The uniform distribution scenario.}&{\footnotesize (b) The discrete distribution scenario.}
	\end{tabular}
	\caption{Convergence of WF with noise-free measurements.}
	\label{S1}
\end{figure}



The second experiment reports the robustness of our algorithms in the presence of noise. We again choose $n\in\{100,200,500,1000\}$, with $m = 10n$, and a fixed noise level of $\sigma=0.05\norm{\vzx}^2$. The step size  is set to be $\eta = \eta_0/\xkh{\sum_j y_j/m}$ with $\eta_0=0.1$.
Figure \ref{S2} demonstrates that our algorithms exhibit good performance in the presence of noise, typically converging in approximately 200 iterations.

\begin{figure}[tbhp]
	\setlength\tabcolsep{1pt}
	\centering
	\begin{tabular}{cc}
		\includegraphics[width=0.49\textwidth]{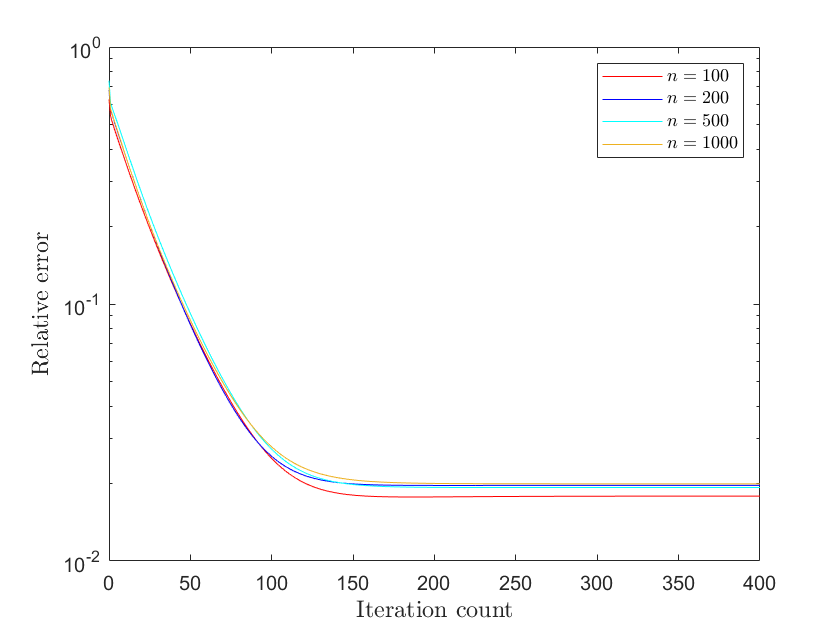}
		&\includegraphics[width=0.49\textwidth]{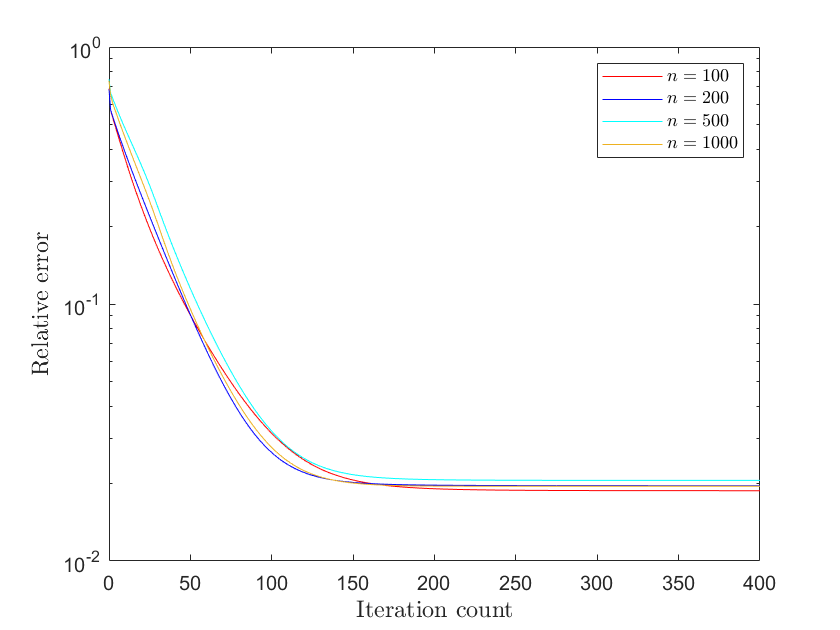}\\
		{\footnotesize (a) The uniform distribution scenario.}&{\footnotesize (b) The discrete distribution scenario.}
	\end{tabular}
	\caption{Convergence of WF for Gaussian noise with $\sigma=0.05\norm{\vzx}^2$.}
	\label{S2}
\end{figure}

The third experiment aims to report the success rate under  different distributions. For complex-valued case, we use the complex Gaussian distribution and the complex uniform distribution. And for real-valued case, each row $a_{j,k}$ of the measurement vectors $\va_j$ are independently generated following either the Gaussian distribution $\mathcal{N}\xkh{\mzero,\mi_n}$ or a uniform distribution $\sqrt{3}\, \mathcal{U}[-1,1]$.  Then we set \( n = 1000 \), \( m/n = \dkh{5,6,\cdots,12} \), and utilize Gaussian noise with $\sigma=0.05\norm{\vzx}^2$.  The step size  is set to be $\eta = \eta_0/\xkh{\sum_j y_j/m}$ with $\eta_0=0.1$.  The maximum number of iterations is set to be $400$. The success rate is calculated based on $100$ trials. In the presence of noise with a noise level of $\sigma=0.05\norm{\vzx}^2$, we say a trial has successfully reconstructed the target signal if the relative error is less than $0.1$.  As depicted in Figure \ref{S3}, it is observed that when $m\ge 12n$, the success rate almost reaches  $100\%$. For both complex and real scenarios, when the ratio \(m/n\) falls within the range of 6 to 9, the success rate associated with our uniform distribution is approximately 2\% to 9\% higher compared to that of the Gaussian distribution. We list the differences in relevant parameters in Table \ref{tab:para of distr}, and for the sake of simplicity, we only consider the case in the real number field. This indicates that in this experiment, the influence generated by the subgaussian norm $K$ is stronger than that caused by parameter $\beta_1$. 
This difference in success rate validates the analysis in Remark \ref{rm:parameter}.

\begin{figure}[t]
	\setlength\tabcolsep{1pt}
	\centering
	\begin{tabular}{cc}
		\includegraphics[width=0.49\textwidth]{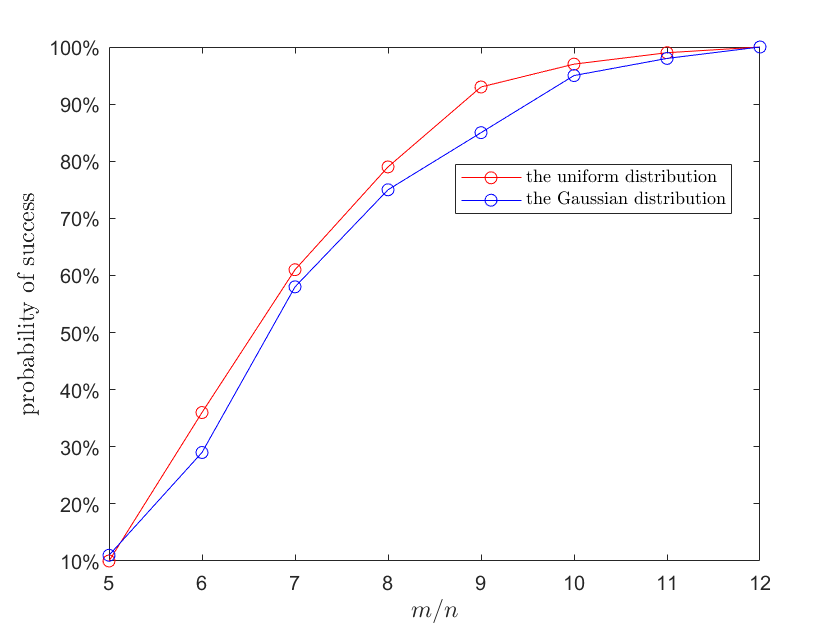}
		&\includegraphics[width=0.49\textwidth]{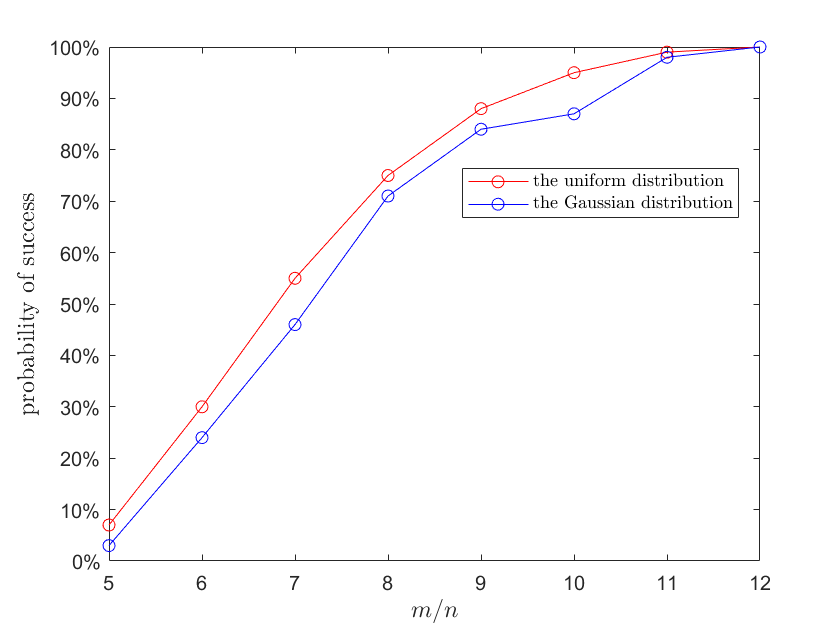}\\
		{\footnotesize (a) The complex-valued case.}&{\footnotesize (b) The real-valued case.}
	\end{tabular}
	\caption{The plot of success rates of our uniform and discrete measurements in real and complex fields in relation to $m/n$.}
	\label{S3}
\end{figure}

\begin{table}[b]
	\footnotesize
	\caption{The parameters of two distributions in real-valued case.}\label{tab:para of distr}
	\begin{center}
		\begin{tabular}{|c|c|c|}\hline
			Species& $\beta_1$ & $K$ \\\hline
			\bf the Gauss distribution&2& $\sqrt{8/3}$\\
			\bf the uniform distribution&4/5& less than $\sqrt{2.5}$\\\hline
		\end{tabular}
	\end{center}
\end{table}


\section{Discussion}\label{sec:discussion}
This paper addresses the complex phase retrieval problem from subgaussian measurements, considering scenarios where the measurements either obey a fourth-moment condition or the target signals exhibit non-peakiness. Under each assumption, by employing the leave-one-out arguments, we demonstrate that the Wirtinger flow method,  with a modified spectral initialization, converges linearly to the target signals with high probability, provided the sampling complexity   \(m\ge O(n \log^3 m)\). Our results hold for arbitrary noise, with subexponential noise representing a specific instance.

Several intriguing issues require investigation in future research. Firstly, due to the utilization of Bernstein's inequality, the sampling complexity is \(O(n \log^3 m)\). It would be valuable to investigate whether more potent statistical tools exist to reduce the  sampling complexity $O(n\log m)$ or $O(n)$. Secondly, the assumption of subgaussian measurement vectors could be extended to diverse problem domains such as low-rank matrix recovery, blind deconvolution, the coded diffraction pattern (CDP) model, and beyond.

\section{Appendix A: Proofs of Technical Results}

\subsection{The expectations}
\begin{lemma}\label{le:e}
	Suppose that  $\va_1,...,\va_m \in \C^n$ are i.i.d. subguassian random vectors with $\E a_1=0$, $\E\abs{a_1}^2=1$, $\E\abs{a_1}^4= 1+\beta_1$ and $\abs{\E a_1^2}^2 = 1-\beta_2$ for some $\beta_1\ge 0$ and $0<\beta_2\le 1$. For the function $f(\vz):\C^n\rightarrow\R$ defined in \eqref{def:f}, we have
	\begin{align*}
		\E \nabla f_{\vz}(\vz) =& \xkh{2\norm{\vz}^2-\norm{\vzx}^2}\vz-\vzxh \vz \vzx+\xkh{1-\beta_2}\xkh{\bar{\vz} \vz^\T-\vzxc\vzxt}\vz\\
		&-\xkh{2-\beta_1-\beta_2}\xkh{\md_1(\vz)-\md_1(\vzx)}\vz-\frac{\vone^\T \vxi}{m}\vz,    \\
		\E \nabla^2 f(\vz) =& \begin{bmatrix}
			\ma & \mb \\ \mb^\hh & \bar{\ma}
		\end{bmatrix},
	\end{align*}
	where
	\begin{align*}
		\ma =& \xkh{2\norm{\vz}^2-\norm{\vzx}^2}\mi + \vz\vz^\hh - \vzx\vzxh+\xkh{1-\beta_2}\xkh{2\bar{\vz}\vz^\T-\vzxc\vzxt}  \\
		& - \xkh{2-\beta_1-\beta_2}\xkh{2\md_1(\vz)-\md_1(\vzx)}-\frac{\vone^\T \vxi}{m}\mi,   \\
		\mb =& 2\vz\vz^\T + (1-\beta_2)\xkh{\vz^\T \vz}\mi - \xkh{2-\beta_1-\beta_2}\md_2(\vz).
	\end{align*}
	Here, $\md_1(\vz), \md_2(\vz)$ are defined in \eqref{def:md}.
\end{lemma}

\begin{proof}
        From definition  \eqref{def:f},  one can observe that
	\begin{align*}
		f(\vz) =& \frac{1}{2m} \sum_{j=1}^{m} \xkh{\aabs{\va_j^\hh \vzx}^2 - \aabs{\va_j^\hh \vz}^2 - \xi_j}^2\\
		       =& \frac{1}{2m} \sum_{j=1}^{m} \left[\vz^\hh \xkh{\aabs{\va_j^\hh \vz}^2 \va_j\va_j^\hh-2\aabs{\va_j^\hh \vzx}^2 \va_j\va_j^\hh} \vz + \vzxh \xkh{\aabs{\va_j^\hh \vzx}^2 \va_j \va_j^\hh} \vzx\right.\\
		       &\left.+\xi_j^2 - 2 \xi_j\xkh{\aabs{\va_j^\hh \vzx}^2-\aabs{\va_j^\hh \vz}^2}\right].
	\end{align*}
To obtain the expression of $\E f(\vz)$,  it suffices to compute $\E \abs{\va_1^\hh \vz}^2\va_1\va_1^\hh$, $\E \xkh{\va_1^\hh \vz}^2\va_1\va_1^\T$ and $\E \sum_{i = 1}^{m}\xi_i \va_i\va_i^\hh$.  For the first term $\E \abs{\va_1^\hh \vz}^2\va_1\va_1^\hh$, one has
\[
\E \abs{\va_1^\hh \vz}^2 a_{1j}\bar{a}_{1j} = \E \zkh{\abs{a_{1j}}^4\abs{z_j}^2 + \sum_{k\ne j}\abs{a_{1j}}^2\abs{a_{1j}}^2\abs{z_j}^2 }=  \norm{\vz}^2 + \xkh{\E \abs{a_{11}}^4-1} \abs{z_j}^2
\]
and
\[
\E \abs{\va_1^\hh \vz}^2 a_{1j}\bar{a}_{1k} = \E \zkh{\abs{a_{1j}}^2\abs{a_{1k}}^2 z_j\bar{z}_k + a_{1j}^2 \bar{a}_{1k}^2 z_k\bar{z}_j} = z_j \bar{z}_k + \abs{\E a_{11}^2}^2\bar{z}_j z_k
\]
for $j\ne k$. Therefore, it holds
\begin{eqnarray*}
\E \abs{\va_1^\hh \vz}^2\va_1\va_1^\hh &=& \norm{\vz}^2\mi + \xkh{\E \abs{a_{11}}^4-1}\md_1(\vz)+\vz\vz^\hh-\md_1(\vz)+\abs{\E a_{11}^2}^2\xkh{\bar{\vz}\vz^\T-\md_1(\vz)}\\
			& =&\norm{\vz}^2\mi + \vz\vz^\hh + \xkh{1-\beta_2}\bar{\vz}\vz^\T - \xkh{2-\beta_1-\beta_2}\md_1(\vz).
\end{eqnarray*}
For the second term $\E \xkh{\va_1^\hh \vz}^2\va_1\va_1^\hh$,  we see that
\[
			\E \xkh{\va_1^\hh \vz}^2 a_{1j}^2 = \E \zkh{\abs{a_{1j}}^4z_j^2 + \sum_{k\ne j}a_j^2 \bar{a}_k^2 z_k^2 }= (\abs{\E a_{11}^2}^2)\xkh{\vz^\T \vz} + \xkh{\E \abs{a_{11}}^4-\abs{\E a_{11}^2}^2}z_j^2
			\]
			and
			\[
			\E \xkh{\va_1^\hh \vz}^2 a_{1j}a_{1k} = 2 \E \abs{a_{1j}}^2\abs{a_{1k}}^2z_jz_k
			= 2z_j z_k
			\]
		for $j\ne k$. This reveals that
		\begin{align*}
\E \xkh{\va_1^\hh \vz}^2\va_1\va_1^\T  =& (\abs{\E a_{11}^2}^2)\xkh{\vz^\T \vz}\mi + \xkh{\E \abs{a_{11}}^4-\abs{\E a_{11}^2}^2}\md_2(\vz)+2\xkh{\vz\vz^\T-\md_2(\vz)}\\
			=& 2\vz\vz^\T + (1-\beta_2)\xkh{\vz^\T \vz}\mi - \xkh{2-\beta_1-\beta_2}\md_2(\vz).
		\end{align*}
For the third term, it is easy to check
		\begin{equation*}
			\E \sum_{i = 1}^{m}\xi_i \va_i\va_i^\hh = \frac{\sum_{i=1}^{m}\xi_i}{m} = \frac{\vone^\T \vxi}{m}.
		\end{equation*}
Gathering the above results, we obtain
	\begin{align*}
		\E f(\vz) =& \norm{\vz}^4+\norm{\vzx}^4-\norm{\vz}^2\norm{\vzx}^2-\abs{\vzxh\vz}^2+\frac{\vone^\T \vxi}{m}\xkh{\norm{\vzx}^2-\norm{\vz}^2}+\frac{\norm{\vxi}^2}{2m}\\
		&+\frac{1-\beta_2}{2}\xkh{\abs{\vz^\T\vz}^2+\abs{\vzxt\vzx}^2-2\abs{\vzxt\vz}^2}\\
		&-\frac{2-\beta_1-\beta_2}{2m} \sum_{j=1}^m \xkh{\abs{z_j}^4+\abs{z_j^*}^4-2\abs{z_j}^2\abs{z_j^*}^2}.
	\end{align*}
Finally, $\E \nabla f_{\vz}(\vz)$ and $\E \nabla^2 f(\vz)$ can be calculated by taking Wirtinger derivative with respect to $\E f(\vz)$.
\end{proof}

\subsection{Technical lemmas}
\begin{lemma}  \cite[Theorem 3.1.1]{Vershynin2018} \label{le:x2}
	Let $X=(x_1,\ldots, x_n) \in \C^n$ be a random vector with independent, subgaussian coordinates $x_j$ that satisfy $\E|x_j|^2=1$, then
	\[
	\PP\dkh{\abs{\norm{X}-\sqrt n}\ge t} \le 2\exp(-ct^2/K^4)
	\]
	where $K=\max_{1\le j\le n} \norms{x_j}_{\psi_2}$.
\end{lemma}

\begin{lemma} \cite[Theorem 4.6.1]{Vershynin2018} \label{le:aat}
	Let $\va_1,\ldots,\va_m$ be the independent, mean zero, subgaussian isotropic random vectors  in $\C^n$ with parameter $K>0$. Then for any $t\ge 0$, with probability at least $1-2\exp(-t^2)$, it holds
	\[
	\norm{\frac1m \sum_{j=1}^m \va_j\va_j^\hh - \mi} \le K^2 \max\xkh{\delta,\delta^2},
	\]
	where $\delta=C\xkh{\sqrt{\frac nm}+\frac t{\sqrt m}}$.
	
\end{lemma}

\begin{lemma}  \label{le:order:statistics}
	Let $\va_1,\ldots,\va_m$ be the independent, mean zero, subgaussian isotropic random vectors  in $\C^n$ with parameter $K \ge 1$. Then for any fixed $\vz\in\C^n$,  there exists a universal constant $c_2>0$ such that with probability at least $1 - O\xkh{m^{-50}}$, it holds
	\[
	\maxm{1\le j\le m} \aabs{\va_j^\hh \vz} \le c_2 K \sqrtlm\norm{\vz}.
	\]
	In addition, with probability at least $1 - m\exp\xkh{-c_0 n}$,  we have
	\[
	\maxm{1\le j\le m} \norm{\va_j} \le c_3 K^2 \sqrtn,
	\]
	where $c_0,c_3>0$ are universal constants.
\end{lemma}
\begin{proof}
For any $1\le j \le m$, the subgaussian tail inequality \eqref{eq:sgtail} gives
	\begin{align*}
		\PP\xkh{\aabs{\va_j^\hh \vz}\ge t}\le 2\exp\xkh{-\frac{c_1 t^2}{ K^2 \norm{\vz}^2}},\quad \forall\ t\ge 0
	\end{align*}
for some universal constant $c_1>0$. Setting $t=c_2 K  \sqrt{\log m}\norm{\vz}$ and taking a union bound, we have
\[
\PP\xkh{ \maxm{1\le j\le m} \aabs{\va_j^\hh \vz}\ge c_2 K  \sqrt{\log m}\norm{\vz} } \le  2 m \exp\xkh{-c_1 c_2^2 \log m} \le O\xkh{m^{-50}},
\]
provided $c_2 \ge \sqrt{51/c_1}$. This gives the conclusion of the first part. The second part is a direct consequence of \eqref{le:x2} by taking $t=O(K^2 \sqrtn)$.

\end{proof}

\begin{lemma}\label{le:pre:uniform:bound}
	Suppose that  $\va_1,\ldots,\va_m \in \C^n $ are i.i.d. subguassian random vectors with parameter $K$. For any constant $0<\delta<1$, with probability at least $1-m\exp(-c_0 n)-O(m^{-20})$, it holds
	\begin{align}
		\label{eq:nabal11}\norm{\frac1m \sum_{j=1}^m \aabs{\va_j^\hh \vz}^2\va_j\va_j^\hh - \E \zkh{\aabs{\va^\hh \vz}^2\va\va^\hh }} \le \delta \norm{\vz} \\
		\label{eq:nabal21}\norm{\frac1m \sum_{j=1}^m (\va_j^\hh \vz)^2 \va_j \va_j^{\T} - \E\zkh{ ( \va^\hh \vz)^2 \va \va^{\T} }} \le \delta \norm{\vz}
	\end{align}
	simultaneously for all $\vz\in \C^n$ obeying \(\max_{1\le j\le m} |\va_j^\hh \vz| \le c_2 K \sqrt{\log m}\norm{\vz}\), provided that $m\ge C_0  \delta ^{-2} K^8n \log^3 m$. Here,  $\va \in \C^n$ is a independent copy of $\va_1$, and  $c_0, c_2, C_0>0$ are some universal constants.
\end{lemma}
\begin{proof}
	See Section \ref{pf:pre:uniform:bound}.
\end{proof}
\begin{lemma}\label{le:noise}
	Suppose that  $\va_1,\ldots,\va_m \in \C^n$ are i.i.d. subgaussian isotropic random vectors with parameter $K$. Let  ${\rm \xi}=(\xi_1,\ldots,\xi_m)^\T \in \R^m$ be a fixed vector. Then with probability at least $1-2\exp\xkh{-c_0 n }$, it holds
	\[
	\norm{ \sum_{j=1}^m \xi_j \va_j \va_j^\hh} \le \aabs{\vone^\T \vxi} + C_{3}K^2\xkh{\sqrt{n}\norm{\vxi} + n\norms{\vxi}_\infty}.
	\]
 Here, $c_0, C_{3}>0$ are  universal constants.
\end{lemma}
\begin{proof}
	For any fixed $\vx\in \C^n$ with $\norm{\vx}=1$, the Bernstein's inequality gives
		\[
		\PP\dkh{\abs{\sum_{j=1}^m \xkh{\xi_j \abs{\va_j^\hh \vx}^2-\E\zkh{\xi_j \abs{\va_j^\hh \vx}^2}}  } \ge t } \le 2\exp\xkh{-c \min\xkh{\frac{t^2}{K^4 \norm{\vxi}^2},  \frac{t}{K^2 \norms{\vxi}_\infty}} }
		\]
	for any $t \ge 0$.
	Set $t=\frac{1}{2} \cdot C_{3}K^2(\sqrt{n}\norm{\vxi} + n\norms{\vxi}_\infty)$ with
	$C_{3}>1$.  Then one has
	\[
	\PP\dkh{\abs{\sum_{j=1}^m \xi_j \xkh{\aabs{\va_j^\hh \vx}^2-1} } \ge t } \le 2\exp\xkh{- \frac{cC_{3}}4 \cdot n}.
	\]
	Here,  we use the fact that $\E\zkh{\aabs{\va_j^\hh \vx}}^2=1$ due to $\va_j$ are isotropic random vectors.
	
	Let $\mathcal{N}_{\frac14}$ be a $1/4$-net of the unit sphere $\mathcal{S}_{\C}^{n-1}$, which has cardinality at most $81^n$. Then \cite[Exercise 4.4.3]{Vershynin2018} implies
	\begin{equation*}
		\norm{ \sum_{j=1}^m \xi_j \xkh{\va_j \va_j^\hh-I}}\le 2 \sup_{\vx \in \mathcal{N}_{\frac14}} \abs{\sum_{j=1}^m \xi_j \xkh{\abs{\va_j^\hh \vx}^2-1} }\le C_{3}K^2(\sqrt{n}\norm{\vxi} + n\norms{\vxi}_\infty)
	\end{equation*}
	with probability at least $1-81^n\cdot 2 \exp\xkh{-cC_{3} n/4 }\ge 1-2 \exp\xkh{-c_0 n }$ for some universal constant $c_0$ as long as $C_3 > (4\log 81)/c$. This gives the conclusion that
	\begin{equation*}
		\norm{ \sum_{j=1}^m \xi_j \va_j \va_j^\hh}\le \aabs{\vone^\T \vxi} + C_{3}K^2(\sqrt{n}\norm{\vxi} + n\norms{\vxi}_\infty).
	\end{equation*}
\end{proof}

\begin{lemma}\label{le:uniform:bound}
	Suppose that $\va_1,\ldots,\va_m$ are i.i.d. subguassian random vectors with parameter $K \ge 1$, and the noise vector $\vxi$ satisfies $\norm{\vxi}\lesssim \sqrtm \norm{\vzx}^2$ and $\norminf{\vxi}\lesssim \log m \norm{\vzx}^2$. For any $0<\delta<1$, with probability at least $1-m\exp(-c_0 n)-O(m^{-20})$, it holds
	\begin{align*}
		\norm{\nabla^2 f(\vz) -\E \nabla^2 f(\vz) } \le \delta \norm{\vzx}^2
	\end{align*}
	simultaneously for all $\vz\in \C^n$ obeying
	\begin{align*}
		\max_{1\le j\le m} |\va_j^\hh \vz| &\le c_2 K \sqrt{\log m}\norm{\vz};\\
		\norm{\vz-\vzx} &\le \frac12 \norm{\vzx};
	\end{align*}
	provided that $m\ge C_0  \delta ^{-2} K^8n \log^3 m$. Here, $c_0, c_2, C_0>0$ are some universal constants.
\end{lemma}
\begin{proof}
Recall the Hessian matrix \eqref{def:Hessian} and note that $y_j= \aabs{\va_j^\hh \vzx}^2+\xi_j$. Therefore, we have
 \begin{equation*}
		\nabla^2 f(\vz)=
		\begin{bmatrix}
			\frac 1m\sum\limits_{j=1}^m \xkh{2\aabs{\va_j^\hh \vz}^2-\aabs{\va_j^\hh \vzx}^2-\xi_j}\va_j\va_j^\hh & \frac 1m\sum\limits_{j=1}^m \xxkh{\va_j^\hh \vz}^2\va_j\va_j^\top \\
			\frac 1m\sum\limits_{j=1}^m \xkh{\vz^\hh \va_j}^2\bar{\va}_j\va_j^\hh &\frac 1m\sum\limits_{j=1}^m \xkh{2\aabs{\va_j^\hh \vz}^2-\aabs{\va_j^\hh \vzx}^2-\xi_j} \bar{\va}_j\va_j^\top
		\end{bmatrix}.
	\end{equation*}
	Applying Lemma \ref{le:pre:uniform:bound}, with high probability, one has
	\[
	\norm{\frac 1m \sum_{j=1}^m \aabs{\va_j^\hh \vz}^2\va_j\va_j^\hh - \E \aabs{\va^\hh \vz}^2\va\va^\hh } \le \frac\delta 8 \norm{\vz}
	\]
and
\[
\norm{\frac 1m \sum_{j=1}^m \xxkh{\va_j^\hh \vz}^2 \va_j \va_j^{\T} - \E \xxkh{\va^\hh \vz}^2 \va\va ^{\T}} \le \frac \delta 8 \norm{\vz}.
\]	
Here, $\va$ is a random vector which is a independent copy of $\va_1$.
Since $\max_{1\le j\le m} |\va_j^\hh \vz^*| \le c_2 K \sqrt{\log m}\norm{\vz^*}$ holds with probability exceeding $1-O(m^{-50})$ by Lemma \ref{le:order:statistics}.  Therefore, applying Lemma \ref{le:pre:uniform:bound} again gives
\[
		\norm{\frac 1m \sum_{j=1}^m \aabs{\va_j^\hh \vzx}^2\va_j\va_j^\hh - \E \aabs{\va^\hh \vzx}^2\va\va^\hh  } \le \frac\delta 8 \norm{\vzx}
\]
Lemma \ref{le:noise} implies that with probability at least $1-m\exp(-c_0 n)-O(m^{-20})$, it holds
	\begin{equation*}
		\norm{\frac 1m \sum_{j=1}^m \xi_j \xkh{\va_j \va_j^\hh-I}}\le \frac{C_{3}K^2(\sqrt{n}\norm{\vxi} + n\norms{\vxi}_\infty)}{m}\le \frac\delta 8 \norm{\vzx},
	\end{equation*}
	provided that $m\ge C_0  \delta ^{-2} K^8n \log^3 m$.
Here, we use the fact that $\norm{\vxi}\lesssim \sqrtm \norm{\vzx}^2$ and $\norminf{\vxi}\lesssim \log m \norm{\vzx}^2$.
Collecting all the above estimators together, we arrive at the conclusion that
	\begin{eqnarray*}
		& &\norm{\nabla^2 f(\vz) -\E \nabla^2 f(\vz) } \\
		& \le& 2\norm{\frac 1m \sum_{j=1}^m \aabs{\va_j^\hh \vz}^2\va_j\va_j^\hh -\E \aabs{\va^\hh \vz}^2\va\va^\hh  }+\norm{\frac 1m \sum_{j=1}^m \xkh{\va_j^\hh \vz}^2 \va_j \va_j^{\T} - \E \xxkh{\va^\hh \vz}^2 \va\va ^{\T} }\\
		& &+\norm{\frac 1m \sum_{j=1}^m \aabs{\va_j^\hh \vzx}^2\va_j\va_j^\hh - \E \aabs{\va^\hh \vzx}^2\va\va^\hh  }+\norm{\frac 1m \sum_{j=1}^m \xi_j \xkh{\va_j \va_j^\hh-I}}\\
		& \le& \delta \norm{\vzx},
	\end{eqnarray*}
where we use the fact that $\norm{\vz} \le \norm{\vz^*}+\norm{\vz-\vz^*} \le 2\norm{\vz^*}$ in the last inequality.
\end{proof}

\begin{lemma} \label{le:align}
	For any $\vz_1,\vz_2\in\C^n$, denote
	\begin{align*}
		\phi_1 = \argmin{\phi\in\R} \norm{\e^{i\phi}\vz_1-\vz_2}\quad \mbox{and}\quad \tz_1 = \e^{i\phi_1} \vz_1.
	\end{align*}
	Then it holds that
	\[(\tz_1-\vz_2)^\hh \tz_1 = \tz_1^\hh(\tz_1-\vz_2).
	\]
\end{lemma}
\begin{proof}
	Observe that
	\[g(\phi) := \norm{\e^{i\phi}\vz_1-\vz_2}^2 = \norm{\vz_1}^2+\norm{\vz_2}^2-2\real{\e^{i\phi}\vz_2^\hh\vz_1}.\]
	In order to minimize \( g(\phi) \), $\e^{i\phi}\vz_2^\hh\vz_1$ should be a positive real number, which means
	\begin{align*}
		\e^{i\phi}\vz_2^\hh\vz_1 = \e^{-i\phi}\vz_1^\hh\vz_2.
	\end{align*}
	Therefore, one has
	\begin{equation*}
		\vz_2^\hh \tz_1 = \tz_1^\hh \vz_2.
	\end{equation*}
	This gives the desired result.
\end{proof}

\begin{lemma} \label{le:noalign}
	For any fixed $\vzx\in\C^n$. Suppose that there exist two vectors \(\vz_1,\vz_2\) satisfying
	\begin{equation*}
		\max \dkh{\norm{\vz_1-\vzx},\norm{\vz_2-\vzx}}\le \delta\le \frac 14
	\end{equation*}
	for some sufficiently small constant $\delta>0$. Denote
	\begin{align*}
		\phi_1 = \argmin{\phi\in\R} \norm{\e^{i\phi}\vz_1-\vzx}\quad \mbox{and}\quad \phi_2 = \argmin{\phi\in\R} \norm{\e^{i\phi}\vz_2-\vzx}.
	\end{align*}
	Then one has
	\[\norm{\e^{i\phi_1}\vz_1-\e^{i\phi_2}\vz_2}\lesssim \norm{\vz_1-\vz_2}.
	\]
\end{lemma}
\begin{proof}
	The proof is essentially identical to \cite[Lemma 55]{Macong2020}.
\end{proof}

\section{Appendix B: Proofs with Assumption A}

\subsection{Proof of Lemma \ref{le:ric:a} }\label{pf:ric:a}
For any $\vz \in \C^n$ near $\vz^*$, we decompose the Wirtinger Hessian $\nabla^2 f(\vz)$ into the following two parts:
\begin{equation*}
	\nabla^2 f(\vz) = \nabla^2 F(\vzx) + \xkh{\nabla^2 f(\vz) - \nabla^2 F(\vzx)},
\end{equation*}
where $\nabla^2 F(\vzx)$ denotes the expectation of $\nabla^2 f(\vz^*)$. Our proof then proceeds by showing that  $\nabla^2 F(\vzx)$ satisfies the smoothness and restricted strong convexity properties,  and  \(\norm{\nabla^2 f(\vz)-\nabla^2 F(\vzx)}\) is a perturbation term which is sufficiently small under our assumptions.
The following lemma demonstrates  the smoothness and restricted strong convexity of \(\nabla^2 F(\vzx)\).

\begin{lemma}\label{le:subric1:a}
	Assume that all conditions in \eqref{le:ric:a} are satisfied. Then it holds
	\begin{equation*}
		\norm{\nabla^2 f(\vzx)}\le 10+3\beta_1\quad and \quad \vu^\hh \nabla^2 f(\vzx) \vu\ge \frac{7\betat}{8}\norm{\vu}^2
	\end{equation*}
   for all $\vu \in \C^{2n}$ defined in Lemma \ref{le:ric:a}.
\end{lemma}

The next lemma shows  that \(\norm{\nabla^2 f(\vz)-\nabla^2 F(\vzx)}\) is a small perturbation.

\begin{lemma}\label{le:subric2:a}
	Suppose the sample complexity satisfies $m\ge C_0 \betat^{-2}\beta_2^{-2} $ for some universal constant $ C_0 > 0 $. Then with probability at least $1-O(m^{-10})-m\exp (-c_0 n)$, one has
	\begin{equation*}
		\norm{\nabla^2 f(\vz)-\nabla^2 F(\vzx)} \le \frac{5\betat}{8}
	\end{equation*}
	for all $\vz$ satisfies \eqref{cond1:ric:a} and \eqref{cond2:ric:a}, where $c_0>0$ is some universal constant.
\end{lemma}

With the above two lemmas, we see that for all $\vz$ satisfies \eqref{cond1:ric:a} and \eqref{cond2:ric:a}, it holds
\begin{align*}
	\norm{\nabla^2 f(\vz)} \le \norm{\nabla^2 F(\vzx)} + \norm{\nabla^2 f(\vz) - \nabla^2 F(\vzx)} \le 12 + 3\beta_1
\end{align*}
due to $\betat=\min\dkh{\beta_1, \beta_2}\le 1$.  In addition,  for all $\vu \in \C^{2n}$ defined in Lemma \ref{le:ric:a}, it holds
\begin{eqnarray*}
	\vu^\hh \nabla^2 f(\vz) \vu & = & \ \vu^\hh\nabla^2 F(\vzx)\vu + \vu^\hh\xkh{\nabla^2 f(\vz) - \nabla^2 F(\vzx)}\vu\\
	&\ge& \ \vu^\hh\nabla^2 F(\vzx)\vu - \norm{\vu}^2 \cdot \norm{\nabla^2 f(\vz) - \nabla^2 F(\vzx)}\\
	& \ge& \ \frac{\betat}{4}\norm{\vu}^2.
\end{eqnarray*}
This give the proof Lemma \ref{le:ric:a}.

\subsubsection{Proof of Lemma \ref{le:subric1:a} }\label{pf:subric1:a}
According to Lemma \ref{le:e}, we formulate that
\begin{equation*}
	\nabla^2 F(\vzx)=
	\begin{bmatrix}
		\ma_1 & \mb_1 \\
		\mb_1^\hh & \bar{\ma}_1
	\end{bmatrix}
\end{equation*}
where
\begin{eqnarray*}
	&&\ma_1 = \norm{\vzx}^2 \mi+\vzx\vzxh+(1-\beta_2)\overline{\vzx}\vzxt-(2-\beta_1-\beta_2) \md_1(\vzx)-\frac{\vone^\T \vxi}{m} \mi,\\
	&&\mb_1 = 2\vzx\vzxt + (1-\beta_2) (\vzxt\vzx)\mi-(2-\beta_1-\beta_2) \md_2(\vzx).
\end{eqnarray*}
Here, $\mi$ represents the identity matrix, while $\md_1(\vz)$ and $\md_2(\vz)$ are defined in \eqref{def:md}.
We proceed to decompose \( \nabla^2 F(\vzx) \) as follows,
\begin{eqnarray} \label{eq:nabf0}
	\nabla^2 F(\vzx) & =& \norm{\vzx}^2\begin{bmatrix}
		\mi\ \ & \mzero \\
		\mzero\ \ & \mi
	\end{bmatrix}
	+ \begin{bmatrix}
		\vzx\vzxh\  & \vzx\vzxt \\
		\vzxc\vzxh\  & \vzxc\vzxt
	\end{bmatrix}
	+\begin{bmatrix}
		\mzero & \vzx\vzxt \\
		\vzxc\vzxh & \mzero
	\end{bmatrix}\\
	&&+(1-\beta_2)\begin{bmatrix}
		\mzero & \xkh{\vzxt\vzx}\mi \\
		\overline{\xkh{\vzxt\vzx}}\mi & \mzero
	\end{bmatrix}
	+(1-\beta_2)\begin{bmatrix}
		\vzxc\vzxt & \mzero \\
		\mzero & \vzx\vzxh
	\end{bmatrix}\notag \\
	& &-(2-\beta_1-\beta_2)\begin{bmatrix}
		\md_1(\vzx)\  & \md_2(\vzx) \\
		\overline{\md_2(\vzx)}\  & \md_1(\vzx)
	\end{bmatrix}-(\frac{\vone^\T \vxi}{m})\begin{bmatrix}
		\mi\  & \mzero \\
		\mzero\  & \mi
	\end{bmatrix}. \notag
\end{eqnarray}
It is easy to see that
\begin{eqnarray*}
	\norm{\begin{bmatrix}
			\mzero & \vzx\vzxt \\
			\vzxc\vzxh & \mzero
	\end{bmatrix}} &= &  \sup_{\substack{\vx \in \C^n, \vy \in \C^n\\
	\norm{\vx}^2+\norm{\vy}^2=1}}
	\norm{\begin{bmatrix} \mzero & \vzx\vzxt \\ \vzxc\vzxh & \mzero\end{bmatrix}\cdot \begin{bmatrix} \vx\\ \vy\end{bmatrix}}\\
	&= & \sup_{\substack{\vx \in \C^n, \vy \in \C^n\\
			\norm{\vx}^2+\norm{\vy}^2=1}}
	\sqrt{\norm{\vzx\vzxt\vy}^2+\norm{\vzxc\vzxh\vx}^2} \notag \\
	&\le & \sup_{\substack{\vx \in \C^n, \vy \in \C^n\\
			\norm{\vx}^2+\norm{\vy}^2=1}}
	\sqrt{\norm{\vzx}^4\norm{\vy}^2+\norm{\vzx}^4\norm{\vx}^2} = \norm{\vzx}^2=1, \notag
\end{eqnarray*}
where we use the Cauchy-Schwarz inequality in the last line.  Similarly, one can derive
\begin{equation*}
	\norm{(1-\beta_2)\begin{bmatrix}
			\mzero & \xkh{\vzxt\vzx}\mi \\
			\overline{\xkh{\vzxt\vzx}}\mi & \mzero
	\end{bmatrix}}\le (1-\beta_2)\norm{\vzx}^2 \le 1
\end{equation*}
and
\begin{equation*}	
	\norm{(1-\beta_2)\begin{bmatrix}
			\vzxc\vzxt & \mzero \\
			\mzero & \vzx\vzxh
	\end{bmatrix}}\le (1-\beta_2)\norm{\vzx}^2 \le 1.
\end{equation*}
Here, we use the fact that $0< \beta_2\le 1$. Observe that
\begin{equation*}
	\norm{\begin{bmatrix}
			\vzx\vzxh\  & \vzx\vzxt \\
			\vzxc\vzxh\  & \vzxc\vzxt
	\end{bmatrix}}=\norm{\begin{bmatrix}\vzx \\ \vzxc\end{bmatrix}\cdot\begin{bmatrix}\vzxh&\vzxt\end{bmatrix}}\le 2\norm{\vzx}^2=2.
\end{equation*}
In addition, denote the i-th component of $\vzx$ as $z_i^*$ for $1\le i\le n$.  Since
\begin{align*}
	\norm{\md_k(\vzx)}\le \max_{1\le j\le n} \abs{z_j^*}^2\le \norm{\vzx}^2 =1, \quad (k=1,2)
\end{align*}
and $\abs{2-\beta_1-\beta_2}\le 2+\beta_1$ for $0< \beta_2\le 1, \beta_1 > 0$, we obtain
\begin{eqnarray*}
	&& \norm{(2-\beta_1-\beta_2)\begin{bmatrix}
			\md_1(\vzx)\  & \md_2(\vzx) \\
			\overline{\md_2(\vzx)}\  & \md_1(\vzx)
	\end{bmatrix}}\\
	 & \le& (2+\beta_1)\norm{\begin{bmatrix}
			\md_1(\vzx)\  & \mzero \\
			\mzero\  & \md_1(\vzx)
	\end{bmatrix}}+(2+\beta_1)\norm{\begin{bmatrix}
			\mzero\  & \md_2(\vzx) \\
			\overline{\md_2(\vzx)}\  & \mzero
	\end{bmatrix}} \notag \\
	& \le& 4+2\beta_1. \notag
\end{eqnarray*}
Finally,  it holds
\begin{equation*}
	\norm{(\frac{\vone^\T \vxi}{m})\begin{bmatrix}
			\mi\  & \mzero \\
			\mzero\  & \mi
	\end{bmatrix}}\le \betat=\min\dkh{\beta_1,\beta_2}\le \beta_1
\end{equation*}
as long as $\abs{\vone^T\vxi}\le \betat m$.
Putting all above estimators into \eqref{eq:nabf0}, we obtain the conclusion that
 \[\norm{\nabla^2 F(\vzx)}\le 10+3\beta_1.\]

Next, we turn to  the restricted strong convexity for $\nabla^2 F(\vzx)$.
For convenience, we denote $\vw = \vz_1-\vz_2\in\C^n$, where $\vz_1$ is aligned with $\vz_2$ and obeys \eqref{eq:z1z2}. Then
\begin{equation*}
	\vu = \begin{bmatrix}
		\vw\\\overline{\vw}
	\end{bmatrix}.
\end{equation*}
A simple calculation gives
\begin{equation} \label{eq:un2u}
	\vu^\hh\nabla^2 F(\vzx)\vu=2J_1+J_2+\bar{J_2}+J_3,
\end{equation}
where
\begin{eqnarray*}
	J_1 &=& \norm{\vzx}^2\cdot \norm{\vw}^2+\vwh\vzx\vzxh\vw + (1-\beta_2)\vwh\vzxc\vzxt\vw-(2-\beta_1-\beta_2)\vwh\md_1(\vzx)\vw,\\
	J_2 &=& 2\vwh\vzx\vzxt\vwc+(1-\beta_2)\xkh{\vzxt\vzx}\overline{\xkh{\vwt\vw}}-(2-\beta_1-\beta_2)\vwh\md_2(\vzx)\vwc,\\
	J_3 & =& -2 \cdot \frac{\vone^\T \vxi}{m}\cdot \norm{\vw}^2.
\end{eqnarray*}
We claim that the following holds:
\begin{equation}\label{j1:a}
J_1 \ge  \sum_{i>j} \abs{\zxi\wjc+\zxj\wic}^2 + (\beta_1+\beta_2)\sumin \abs{\zxi}^2\abs{\wi}^2,
\end{equation}
\begin{equation}\label{j1:b}
J_1 \ge  (1-\beta_2)\sum_{i>j} \abs{\zxi\wjc+\zxj\wic}^2 + \beta_2\norm{\vzx}^2 \norm{\vw}^2 + (\beta_1-\beta_2)\sumin \abs{\zxi}^2\abs{\wi}^2
\end{equation}
and
\begin{equation} \label{j2}
J_2 + \bar{J_2} \ge  2(1-\beta_2)\sum_{i>j} \real{\xkh{\zxi\wjc+\zxj\wic}^2} + 2(\beta_1-\beta_2)\sumin \real{\zxis\wics}-10 \norm{\vzx-\vz_1}\norm{\vw}^2.
\end{equation}
It is worth noting that two lower bounds of $J_1$ are given in this claim, which  helps us to establish the desired results.
We divide our proof into the following two cases:

{\bf Case 1:} $\sumin \abs{\zxi}^2\abs{\wi}^2\ge \norm{\vw}^2/2$.  For this case, putting \eqref{j1:a} and \eqref{j2}    into \eqref{eq:un2u}, we obtain
\begin{eqnarray*}
&& \vu^\hh\nabla^2 F(\vzx)\vu \\
 & \ge&2(1-\beta_2)\sum_{i>j} \real{\xkh{\zxi\wjc+\zxj\wic}^2} + 2\sum_{i>j} \abs{\zxi\wjc+\zxj\wic}^2-10 \norm{\vzx-\vz_1}\norm{\vw}^2 \\
		&&+ 2(\beta_1-\beta_2)\sumin \real{\zxis\wics} + 2(\beta_1+\beta_2)\sumin \abs{\zxi}^2\abs{\wi}^2 -2\cdot \frac{\vone^\T \vxi}{m} \norm{\vw}^2 \\
		& \stackrel{\mbox{(i)}}{\ge}&4 \betat\sumin \abs{\zxi}^2\abs{\wi}^2 - 2\frac{\vone^\T \vxi}{m}  \norm{\vw}^2 -10 \norm{\vzx-\vz_1}\norm{\vw}^2\\
		& \stackrel{\mbox{(ii)}}{\ge}&2 \betat \norm{\vw}^2 - 2\frac{\vone^\T \vxi}{m}  \norm{\vw}^2 -10 \norm{\vzx-\vz_1}\norm{\vw}^2 \\
		&\stackrel{\mbox{(iii)}}{\ge} &  \frac 74 \betat \norm{\vw}^2  = \frac 78 \betat \norm{\vu}^2,
	\end{eqnarray*}
	which gives the conclusion for Case 1. Here,   (i) follows from the fact
\begin{eqnarray*}
& & 2(1-\beta_2)\sum_{i>j} \real{\xkh{\zxi\wjc+\zxj\wic}^2} + 2\sum_{i>j} \abs{\zxi\wjc+\zxj\wic}^2 \\
&=&  4(1-\beta_2)\sum_{i>j} \xkh{\real{\zxi\wjc + \zxj\wic}}^2 + 2\beta_2\sum_{i>j} \abs{\zxi\wjc + \zxj\wic}^2 \ge 0
\end{eqnarray*}
due to the identity $\abs{\alpha}^2+\real{\alpha^2}=2\xkh{\real{\alpha}}^2$ for any $\alpha\in\C$, and
\begin{eqnarray*}
&&  2(\beta_1-\beta_2)\sumin \real{\zxis\wics} + 2(\beta_1+\beta_2)\sumin \abs{\zxi}^2\abs{\wi}^2 \\
& = & 4(\beta_1-\beta_2)\sumin \xkh{\real{\zxi\wic}}^2 + 4\beta_2\sumin \abs{\zxi}^2 \abs{\wi}^2 \\
 &\ge & 4 \betat\sumin \abs{\zxi}^2\abs{\wi}^2
\end{eqnarray*}
due to
 \begin{equation} \label{ineq:betat}
4(\beta_1-\beta_2)\sumin \xkh{\real{\zxi\wic}}^2\ge \left\{ \begin{aligned}
				&0, \qquad\qquad\qquad\qquad\qquad\ \ \ \ \; \; \rm{if\ }\beta_1\ge \beta_2   \\
				&4(\beta_1-\beta_2) \sum\nolimits_{i=1}^n \abs{\zxi}^2\abs{\wi}^2,\ \rm{if\ }0<\beta_1<\beta_2.
			\end{aligned}\right.,
\end{equation}
where  $\betat=\min\dkh{\beta_1, \beta_2}\le 1$.  And  (ii) comes from the assumption that  $\sumin \abs{\zxi}^2\abs{\wi}^2\ge \norm{\vw}^2/2$, and  (iii) arises from the fact that  $ \norm{\vzx-\vz_1}\le\delta_1\le \betat/ 80$ and  $\frac{\abs{\vone^\T\vxi}}{m}\le \betat/16$.

{\bf Case 2:} $\sumin \abs{\zxi}^2\abs{\wi}^2< \norm{\vw}^2/2$.  For this case, putting \eqref{j1:b} and \eqref{j2}    into \eqref{eq:un2u} yields
	\begin{eqnarray*}
		& &\vu^\hh\nabla^2 F(\vzx)\vu \\
		& \ge&2(1-\beta_2)\sum_{i>j} \real{\xkh{\zxi\wjc+\zxj\wic}^2} + 2(1-\beta_2)\sum_{i>j} \abs{\zxi\wjc+\zxj\wic}^2\\
		&&+ 2(\beta_1-\beta_2)\sumin \real{\zxis\wics} + 2(\beta_1-\beta_2)\sumin \abs{\zxi}^2\abs{\wi}^2\\
		& &+ 2\beta_2\norm{\vzx}^2 \norm{\vw}^2-2\frac{\vone^\T \vxi}{m} \norm{\vw}^2-10 \norm{\vzx-\vz_1}\norm{\vw}^2\\
		& \stackrel{\mbox{(i)}}{\ge}&4(1-\beta_2)\sum_{i>j} \xkh{\real{\zxi\wjc + \zxj\wic}}^2 + 4(\beta_1-\beta_2)\sumin \xkh{\real{\zxi\wic}}^2\\
		&& + 2\beta_2 \norm{\vzx}^2 \norm{\vw}^2-2\frac{\vone^\T \vxi}{m}  \norm{\vw}^2-10 \norm{\vzx-\vz_1}\norm{\vw}^2\\
		&\stackrel{\mbox{(ii)}}{\ge}& 2 \betat \norm{\vw}^2 - 2\frac{\vone^\T \vxi}{m}  \norm{\vw}^2 -10 \norm{\vzx-\vz_1}\norm{\vw}^2\\
		&\stackrel{\mbox{(iii)}}{\ge} &  \frac 78 \betat \norm{\vu}^2,
	\end{eqnarray*}
	where (i) uses the identity $\abs{\alpha}^2+\real{\alpha^2}=2\xkh{\real{\alpha}}^2$ for any $\alpha\in\C$ again, and (ii) holds since \eqref{ineq:betat} and the assumption $\sumin \abs{\zxi}^2\abs{\wi}^2< \norm{\vw}^2/2$,
and (iii) holds due to the same reason as in the Case 1. Therefore, we arrive at the conclusion for the Case 2.

Finally, it remains to prove the claims \eqref{j1:a}, \eqref{j1:b} and \eqref{j2}. For the claim \eqref{j1:a},  we have
\begin{eqnarray*}
			J_1 &= &  \norm{\vzx}^2   \norm{\vw}^2 + \abs{\vzxh\vw}^2 + (1-\beta_2)\abs{\vzxt\vw}^2-(2-\beta_1-\beta_2)\sumin \abs{\zxi}^2\abs{\wi}^2\\
			&\ge &  \norm{\vzx}^2   \norm{\vw}^2 + \abs{\vzxh\vw}^2 - 2 \sumin \abs{\zxi}^2\abs{\wi}^2 +(\beta_1+\beta_2)\sumin \abs{\zxi}^2\abs{\wi}^2\\
			&=& \sum_{i>j} \abs{\zxi\wjc+\zxj\wic}^2+(\beta_1+\beta_2)\sumin \abs{\zxi}^2\abs{\wi}^2,
		\end{eqnarray*}
where the inequality follows from $1-\beta_2\ge 0$ and $\abs{\vzxt\vw}\ge 0$,  and the last equality comes from the fact that
\begin{equation} \label{eq:zw21}
		\norm{\vzx}^2   \norm{\vw}^2 - \sumin \abs{\zxi}^2\abs{\wi}^2 = \sum_{i>j}\xkh{\abs{\zxi}^2\abs{w_j}^2+\abs{z_j^*}^2\abs{\wi}^2}
\end{equation}
		and
\begin{equation} \label{eq:zw22}
		\aabs{\vzxh\vw}^2 - \sumin \abs{\zxi}^2\abs{\wi}^2 = \sum_{i>j}\xkh{\zxi\wic\zxjc\wj + \zxj\wjc\zxic\wi}.
\end{equation}
To establish claim  \eqref{j1:b}, we decompose $ J_1 $ in an alternative manner, specifically as follows:
	\begin{eqnarray*}
		J_1 & = & (1-\beta_2)\xkh{\norm{\vzx}^2   \norm{\vw}^2 + \aabs{\vzxh\vw}^2 - 2 \sumin \abs{\zxi}^2\abs{\wi}^2 }+ (1-\beta_2)\aabs{\vzxt\vw}^2\\
		& &+ \beta_2\norm{\vzx}^2   \norm{\vw}^2 + \beta_2\aabs{\vzxh\vw}^2 + (\beta_1-\beta_2)\sumin \abs{\zxi}^2\abs{\wi}^2\\
		&\ge &  (1-\beta_2)\sum_{i>j} \abs{\zxi\wjc+\zxj\wic}^2+ \beta_2\norm{\vzx}^2   \norm{\vw}^2 + (\beta_1-\beta_2)\sumin \abs{\zxi}^2\abs{\wi}^2,
	\end{eqnarray*}
	where the inequality comes from combining \eqref{eq:zw21} and \eqref{eq:zw22}, along with the fact that $\aabs{\vzxh\vw}^2\ge 0$, and $\aabs{\vzxt\vw}^2 \ge 0$.

	We finally turn to the claim \eqref{j2}.  It can be seen that
	\begin{eqnarray*}
		J_2 & =& 2\xkh{\vw^\hh \vzx}^2 + (1-\beta_2)\xkh{\vzxt\vzx}\xkh{\overline{\vwt\vw}}-(2-\beta_1-\beta_2)\vwh\md_2(\vzx)\vwc \\
		&=& (1+\beta_2)\xkh{\vw^\hh \vzx}^2 + (1-\beta_2)\xkh{\vw^\hh \vzx}^2 + (1-\beta_2)\xkh{\vzxt\vzx}\xkh{\overline{\vwt\vw}}\\
		& &-2(1-\beta_2)\sumin \wics\zxis+(\beta_1-\beta_2)\sumin \wics\zxis,
	\end{eqnarray*}
	Similar to \eqref{eq:zw21} and \eqref{eq:zw22}, one can check that
	\begin{equation*}
		\xkh{\vzxt\vzx}\xkh{\overline{\vwt\vw}} + \xkh{\vw^\hh \vzx}^2 -2 \sumin \wics\zxis = \sum_{i>j} \xkh{\zxi\wjc+\zxj\wic}^2.
	\end{equation*}
Therefore, we have
	\begin{equation} \label{eq:J2wz}
		J_2 = (1+\beta_2)\xkh{\vw^\hh\vzx}^2 + (1-\beta_2)\sum_{i>j} \xkh{\zxi\wjc+\zxj\wic}^2 + (\beta_1-\beta_2)\sumin \zxis\wics.
	\end{equation}
	Recall the definition of $\vw$ and note that  $\vz_1$ is aligned with $\vz_2$.  Applying Lemma \ref{le:align} gives $\xkh{\vz_1-\vz_2}^\hh \vz_1 = \vz_1^\hh \xkh{\vz_1-\vz_2}$, which means $\vw^\hh \vz_1 = \vz_1^\hh \vw$. Therefore, we have
	\begin{align}
		\xkh{\vw^\hh\vzx}^2 &= \xkh{\vw^\hh\vz_1}^2 + \xkh{\vw^\hh\xkh{\vzx-\vz_1}} \xkh{\vw^\hh\vz_1} + \xkh{\vw^\hh\xkh{\vzx-\vz_1}} \xkh{\vw^\hh\vzx} \notag \\
		&= \abs{\vw^\hh\vz_1}^2 + \underbrace{\xkh{\vw^\hh\xkh{\vzx-\vz_1}} \xkh{\vw^\hh\vz_1} + \xkh{\vw^\hh\xkh{\vzx-\vz_1}} \xkh{\vw^\hh\vzx}}_{:=\tau_1}.  \label{eq:J2wz1}
	\end{align}
Putting \eqref{eq:J2wz1} into \eqref{eq:J2wz}, we get
\[
J_2 = (1+\beta_2)\xkh{\aabs{\vw^\hh\vz_1}^2 + \tau_1} + (1-\beta_2)\sum_{i>j} \xkh{\zxi\wjc+\zxj\wic}^2 + (\beta_1-\beta_2)\sumin \zxis\wics.
\]
It implies
	\begin{eqnarray*}
			J_2 + \bar{J_2} &  \ge & 2(1-\beta_2)\sum_{i>j} \real{\xkh{\zxi\wjc+\zxj\wic}^2} + 2(\beta_1-\beta_2)\sumin \real{\zxis\wics}  + 2(1+\beta_2)\real{\tau_1} \\
			&\ge & 2(1-\beta_2)\sum_{i>j} \real{\xkh{\zxi\wjc+\zxj\wic}^2} + 2(\beta_1-\beta_2)\sumin \real{\zxis\wics}  -10 \norm{\vzx-\vz_1}\norm{\vw}^2, \\
	\end{eqnarray*}
where the first inequality follows from $\aabs{\vw^\hh\vz_1}^2 \ge 0$, and the last inequality arises from
	\begin{eqnarray}\label{ineq:tau1}
			\abs{\real{\tau_1}}&\le&  \norm{\vz_1}  \norm{\vzx-\vz_1}  \norm{\vw}^2 + \norm{\vzx}  \norm{\vzx-\vz_1}  \norm{\vw}^2\\
			&=&  \xkh{\norm{\vz_1}+\norm{\vzx}} \norm{\vzx-\vz_1}  \norm{\vw}^2\notag\\
			&\le & \frac{5}{2} \norm{\vzx-\vz_1}  \norm{\vw}^2 \notag
	\end{eqnarray}
	Here, we use the fact $\norm{\vz_1} \le \norm{\vzx}+\norm{\vz_1-\vzx} \le 1+\delta_1 \le 3/2$ as long as $\delta_1\le 1/2$.
This completes the proof of claim \eqref{j2}.

\subsubsection{Proof of Lemma \ref{le:subric2:a}}\label{pf:subric2:a}
According to the triangle inequality, we note that
\begin{equation*}
	\norm{\nabla^2 f(\vz) - \nabla^2 F(\vzx)}\le \norm{\nabla^2 f(\vz) - \E\zkh{\nabla^2 f(\vz)}} + \norm{\E\zkh{\nabla^2 f(\vz)} - \nabla^2 F(\vzx)}.
\end{equation*}
It follows from Lemma \ref{le:order:statistics} that, with probability at least $1-O\xkh{m^{-20}}$,  it holds
\[\maxm{1\le j\le m} \aabs{\va_j^\hh\vzx}\le \frac{c_2}{2} K\sqrtlm,
\]
where $c_2>0$ is a universal constant. Thus,
\begin{align*}
	\maxm{1\le j\le m} \aabs{\va_j^\hh\vz}\le& \maxm{1\le j\le m} \aabs{\va_j^\hh\vzx} + \maxm{1\le j\le m} \aabs{\va_j^\hh\xkh{\vz-\vzx}}\\
	\le& \frac{c_2}{2} K\sqrtlm + \tilde{C}_{3} K\sqrtlm\\
	\le& c_2 K\sqrtlm
\end{align*}
as long as $c_2 \ge 2\tilde{C}_{3}$. Here, the second inequality comes from \eqref{cond2:ric:a} that  $\maxm{1\le j\le m} \aabs{\va_j^\hh\xkh{\vz-\vzx}}\le \tilde{C}_{3} K\sqrtlm$.
Invoking Lemma \ref{le:uniform:bound}, we obtain that for any $0<\delta_3<1$,
\begin{equation*}
	\norm{\nabla^2 f(\vz) - \E\zkh{\nabla^2 f(\vz)}}\le \delta_3\norm{\vz}\le 2 \delta_3 \norm{\vzx} = 2\delta_3
\end{equation*}
holds with probability at least $1 - m \exp(-c_0 n) - O(m^{-20})$,  provided that $m\ge C_0\delta_3^{-2}K^8 n\log m$ for some universal constant $C_0>0$ . Here, we use the fact that
\[
\norm{\vz}\le \norm{\vzx}+\norm{\vz-\vzx}\le 2\norm{\vzx}
\]
 as long as $\norm{\vz-\vzx}\le \delta_1\le 1 $.
With regard to $\norm{\E\zkh{\nabla^2 f(\vz)} - \nabla^2 F(\vzx)}$, it holds
\begin{eqnarray}\label{cond:diff of E}
	&& \norm{\E\zkh{\nabla^2 f(\vz)} - \nabla^2 F(\vzx)} \\
	 & \stackrel{\mbox{(i)}}{\le}& 2\Big(\abs{\norm{\vz}^2-\norm{\vzx}^2} + (2-\beta_2)\norm{\vz\vz^\hh-\vzx\vzxh}+ \abs{2-\beta_1-\beta_2}\cdot\norm{\md_1(\vz)-\md_1(\vzx)}\Big)\notag\\
	&&+ 2\norm{\vz\vz^\T-\vzx\vzxt} + (1-\beta_2) \aabs{\vz^\T\vz-\vzxt\vzx}+ \abs{2-\beta_1-\beta_2}\cdot\norm{\md_2(\vz)-\md_2(\vzx)}\notag\\
	& \stackrel{\mbox{(ii)}}{\le}& 8(5+\beta_1)\delta_1 \le 40(1+\beta_1)\delta_1. \notag
\end{eqnarray}
Here,  (i) follows Lemma \ref{le:e} and the triangle inequality, and (ii) utilizes the facts that
\[
		\abs{\norm{\vz}^2-\norm{\vzx}^2} = \aabs{\vz^\hh \vz-\vzxh\vzx} \le \norm{\vz-\vzx}\cdot\xkh{\norm{\vz}+\norm{\vzx}}\le \frac 83 \norm{\vz-\vzx}
		\]
		and
		\[
		\norm{\md_1(\vz)-\md_1(\vzx)} \le \maxm{1\le j\le m} \abs{\abs{z_i}^2-\abs{\zxi}^2}\le \norm{\vz-\vzx}\cdot\xkh{\norm{\vz}+\norm{\vzx}}\le \frac 83 \norm{\vz-\vzx}
\]
as long as $\norm{\vz}\le 5\norm{\vzx}/3$. By taking $\delta_1\le \betat (1+\beta_1)^{-1}/80$ and $\delta_3\le \betat /16$ , it holds
{\small \begin{equation*}
		\norm{\nabla^2 f(\vz) - \nabla^2 F(\vzx)}\le \norm{\nabla^2 f(\vz) - \E\zkh{\nabla^2 f(\vz)}} + \norm{\E\zkh{\nabla^2 f(\vz)} - \nabla^2 F(\vzx)}\le \frac{5\betat}{8}
\end{equation*}}
provided that $m\gtrsim \betat^{-2} n \log^3 m$.

\subsection{Proof of Lemma \ref{le:error:a}}\label{pf:error:a}

Recall from the definition of $\dist{\cdot,\cdot}$ and the notation $\phi(t)$ given in \eqref{def:dist} and \eqref{def:ztilde:a}, respectively. It is easy to see that
\begin{equation*}
	\begin{aligned}
		\dist{\vztt,\vzx} &= \min_{\phi\in \R} \norm{\e^{i \phi}\vztt-\vzx}\le \norm{\e^{i \phi(t)}\vztt-\vzx}.
	\end{aligned}
\end{equation*}
Applying the gradient update rule \eqref{def:df}, one has
\begin{equation*}
	\begin{aligned}
		\begin{bmatrix}
			\e^{i \phi(t)}\vztt-\vzx \vspace{0.4em}\\
			\overline{\e^{i \phi(t)}\vztt-\vzx}
		\end{bmatrix}=
		&\begin{bmatrix}
			\e^{i \phi(t)}\vzt \vspace{0.4em}\\
			\overline{\e^{i \phi(t)}\vzt}
		\end{bmatrix}-\eta
		\begin{bmatrix}
			\e^{i \phi(t)}\nabla f_\vz(\vzt) \vspace{0.4em}\\
			\overline{\e^{i \phi(t)}\nabla f_\vz(\vzt)}
		\end{bmatrix}-
		\begin{bmatrix}
			\vzx \vspace{0.4em}\\
			\overline{\vzx}
		\end{bmatrix}\\
		=
		&\underbrace{\begin{bmatrix}
				\vztu-\vzx \vspace{0.4em}\\
				\overline{\vztu-\vzx}
			\end{bmatrix}-\eta
			\begin{bmatrix}
				\nabla f_\vz(\vztu)-\nabla f_\vz(\vzx) \vspace{0.4em}\\
				\overline{\nabla f_\vz(\vztu)-\nabla f_\vz(\vzx)}
		\end{bmatrix}}_{:=I_1}-\underbrace{\eta
			\begin{bmatrix}
				\nabla f_\vz(\vzx) \vspace{0.4em}\\
				\overline{\nabla f_\vz(\vzx)}
		\end{bmatrix}}_{:=I_2},
	\end{aligned}
\end{equation*}
where $\tz^t = e^{i \phi(t)}\vz^t$ as in \eqref{def:ztilde:a}. For the term $I_1$, the fundamental theorem of calculus \cite[Chapter XIII, Theorem 4.2]{Calculus} states that
\begin{equation*}
	\begin{aligned}
		\begin{bmatrix}
			\nabla f_\vz(\vztu)-\nabla f_\vz(\vzx) \vspace{0.4em}\\
			\overline{\nabla f_\vz(\vztu)-\nabla f_\vz(\vzx)}
		\end{bmatrix}
		=
		\underbrace{\int_{0}^{1} \nabla^2 f(\vz(\tau))\dd \tau}_{:=\ma}
		\begin{bmatrix}
			\vztu-\vzx \vspace{0.4em}\\
			\overline{\vztu-\vzx}
		\end{bmatrix},
	\end{aligned}
\end{equation*}
where $\vz(\tau):=\vzx + \tau \xkh{\vztu-\vzx}$ and $\nabla^2 f$ is the Wirtinger Hessian matrix. Therefore,
\begin{eqnarray}
		\norm{I_1}^2 &= &  \deltaztx^\hh \xkh{I-\eta A}^\hh \xkh{I-\eta A}\deltaztx \label{eq:I143}\\
		&\le&  2(1+\eta^2 \norm{\ma}^2 )\norm{\vztu-\vzx}- 2 \eta\deltaztx^\hh \ma \deltaztx. \notag
\end{eqnarray}
For all $0\le \tau\le 1$, it follows from the assumption \eqref{eq:le32} that it holds
\begin{align*}
	\norm{\vz(\tau)-\vzx}
	\le \norm{\vztu-\vzx}
	\le \delta_1,
\end{align*}
and
\begin{align*}
	\max_{1\le j\le m}\abs{\va_j^\hh \xkh{\vz(\tau)-\vzx}}
	\le& \tilde{C}_{3} K\sqrtlm .
\end{align*}
In addition,  note that $\norm{\vztu-\vzx} = \min\limits_{\phi\in \R} \norm{\e^{i \phi}\vz^t-\vz^*}$. Thus,  $\norm{\vz(\tau)-\vzx}= \tau \norm{\vztu-\vx}$ implies $\vz(\tau)$ is aligned with $\vzx$. We are ready to apply Lemma \ref{le:ric:a} to derive
\begin{equation*}
	\norm{\ma}\le \lambda\quad {\rm and } \quad  \deltaztx^\hh \ma \deltaztx \ge 2\nu \norm{\vztu-\vzx}^2.
\end{equation*}
Here, we set
\begin{equation}\label{simply:para:a}
	\lambda = 12+3\beta_1\quad {\rm and}\quad \nu = \frac{\betat}{4}
\end{equation}
for convenience.  Substitution into \eqref{eq:I143} gives that when $\eta\le \nu\lambda^{-2}$, it holds
\begin{equation*}
	\norm{I_1}^2 \le 2(1+\eta^2\lambda^2-2\nu \eta)\norm{\vztu-\vzx}^2 \le 2(1-\nu \eta)\norm{\vztu-\vzx}^2,
\end{equation*}
and hence
\begin{equation}\label{i1:errorA}
	\norm{I_1} \le \sqrttwo (1-\nu \eta)^{1/2} \norm{\vztu-\vzx} \le \sqrttwo (1-\frac{\nu}{2}\eta) \norm{\vztu-\vzx}.
\end{equation}

We next turn to the term $I_2$. Based on the Wirtinger gradient as in \eqref{def:df}, the Cauchy-Schwarz together with Lemma \ref{le:noise} yields
\begin{eqnarray*}
		\norm{\nabla f_{\vz} (\vzx)}& = &  \norm{ \frac 1m \sum_{i=1}^{m} \xi_i \va_i \va_i^\hh \vzx} \le   \frac{\abs{\vone^\T \vxi}+C_{9} K^2\xkh{\sqrt{n} \norm{\vxi} + n \norminf{\vxi}}}{m},
\end{eqnarray*}
and hence
\begin{equation}\label{i2:errorA}
	\norm{I_2} = \sqrttwo \eta\norm{\nabla f_{\vz} (\vzx)}\le \sqrttwo \eta\frac{\abs{\vone^\T \vxi}+C_{9} K^2\xkh{\sqrt{n} \norm{\vxi} + n \norminf{\vxi}}}{m}.
\end{equation}

Finally, combining \eqref{i1:errorA} and \eqref{i2:errorA}, one has
\begin{align}
	\dist{\vztt,\vzx} &\le \sqrtonetwo \norm{I_1}+\sqrtonetwo \norm{I_2} \notag \\
	\label{errorContraction:a}& \le (1-\frac{\nu}{2}\eta) \dist{\vzt,\vzx}+\eta\frac{\abs{\vone^\T \vxi}+C_{9} K^2\xkh{\sqrt{n} \norm{\vxi} + n \norminf{\vxi}}}{m}
\end{align}
with probability at least $1-O\xkh{m^{-20}}$. This completes the proof.

\subsection{Proof of Lemma \ref{le:ztl:a}}\label{pf:ztl:a}
Recall the definition $\tz^t = e^{i \phi(t)}\vz^t$ as in \eqref{def:ztilde:a}.  For any $\vztl$, denote
\begin{equation} \label{eq:zhatl}
		\phih(t) = \argmin{\phi\in \R} \norm{\e^{i \phi}\vztl-\vztu},  \qquad  \vztlh   = e^{i \phih(t)}\vztl.
\end{equation}
Then we have
\begin{align*}
	\dist{\vzttl,\vzttu}  =& \min_{\phi\in \R} \norm{\e^{i \phi}\vzttl-\e^{i \phi(t+1)}\vztt}\\
	=& \min_{\phi\in \R} \norm{ \frac{\e^{i\phi(t+1)}}{ \e^{i\phi(t)}} \xkh{\e^{i\phi} \cdot \frac{\e^{i\phi(t)}}{ e^{i\phi(t+1)}}\vzttl-\e^{i\phi(t)}\vztt}}\\
	\le& \norm{\e^{i\phih(t)}\vzttl-\e^{i\phi(t)}\vztt},
\end{align*}
where the inequality follows by setting $\phi = \phih(t)+\phi(t+1)-\phi(t)$. Applying the gradient update rules \eqref{def:df} and \eqref{eq:uprulel}, we obtain
\begin{eqnarray}
	&&\begin{bmatrix}
		\e^{i\phih(t)}\vzttl-\e^{i\phi(t)}\vztt  \vspace{0.4em}\\
		\overline{\e^{i\phih(t)}\vzttl-\e^{i\phi(t)}\vztt}
	\end{bmatrix} \notag\\
	&=&\begin{bmatrix}
		\e^{i\phih(t)}\xkh{\vztl-\eta\nabla f_\vz^\lkh\xkh{\vztl}} \vspace{0.4em}\\
		\overline{\e^{i\phih(t)}\xkh{\vztl-\eta\nabla f_\vz^\lkh\xkh{\vztl}}}
	\end{bmatrix} - \begin{bmatrix}
		\e^{i\phi(t)}\xkh{\vzt-\eta\nabla f_\vz\xkh{\vzt}} \vspace{0.4em}\\
		\overline{\e^{i\phi(t)}\xkh{\vzt-\eta\nabla f_\vz\xkh{\vzt}}}
	\end{bmatrix}\notag \\
	&=&\underbrace{\begin{bmatrix}
			\vztlh-\vztu \vspace{0.4em}\\
			\overline{\vztlh-\vztu}
		\end{bmatrix}-\eta
		\begin{bmatrix}
			\nabla f_\vz(\vztlh)-\nabla f_\vz(\vztu) \vspace{0.4em}\\
			\overline{\nabla f_\vz(\vztlh)-\nabla f_\vz(\vztu)}
	\end{bmatrix}}_{:=I_1}-\underbrace{\frac{\eta}{m}
		\begin{bmatrix}
			\xkh{\abs{\va_l^\hh \vztlh}^2 - y_l}\va_l\va_l^\hh \vztlh \vspace{0.4em}\\
			\xkh{\abs{\va_l^\hh \vztlh}^2 - y_l}\overline{\va_l\va_l^\hh \vztlh}
	\end{bmatrix}}_{:=I_2}. \label{eq:I1I2}
\end{eqnarray}
We first estimate the part $I_1$. The fundamental theorem of calculus  states that
\begin{equation*}
	\begin{aligned}
		\begin{bmatrix}
			\nabla f_\vz(\vztlh)-\nabla f_\vz(\vztu)\\
			\overline{\nabla f_\vz(\vztlh)-\nabla f_\vz(\vztu)}
		\end{bmatrix}
		=
		\int_{0}^{1} \nabla^2 f(\vz(\tau))\dd \tau
		\begin{bmatrix}
			\vztlh-\vztu\\
			\overline{\vztlh-\vztu}
		\end{bmatrix},
	\end{aligned}
\end{equation*}
where $\vz(\tau):=\vztu + \tau \xkh{\vztlh-\vztu}$ and $\nabla^2 f$ is the Wirtinger Hessian matrix.
In order to apply Lemma \ref{le:ric:a}, we need to upper bound the following two terms:
\[\norm{\vz(\tau)-\vzx}\quad \mbox{and}\quad \max_{1\le j\le m}\aabs{\va_j^\hh \xkh{\vz(\tau)-\vzx}}.\]
Under the induction hypotheses \eqref{induct:err:a}-\eqref{induct:inh:a},  with probability at least $1-m\exp(-c_0 n)-O(m^{-20})$, one has
\begin{eqnarray*}
	\norm{\vz(\tau)-\vzx}&\le& \norm{\vztu-\vzx} + \maxm{1\le l\le m} \norm{\vztlh-\vztu}\\
	& \le& C_{3}+ C_{4} \frac{\abs{\vone^\T \vxi}}{m\betat} + C_{5} K^2\frac{\sqrt{n}\norm{\vxi} + n\norminf{\vxi}}{m\betat}  \\
	&&  + C_{6} \frac{K^7\sqrtnlmmm}{m\beta_2}\xkh{1+\frac{\abs{\vone^\T \vxi}+\sqrtn \norm{\vxi}}{m} + \frac{\norminf{\vxi}}{K^4 \log m}}\\
	&\le& \delta_1,
\end{eqnarray*}
and
\begin{eqnarray*}
		\max_{1\le j\le m}\abs{\va_j^\hh \xkh{\vz(\tau)-\vzx}}
		&\le& \max_{1\le j\le m} \abs{\va_j^\hh \xkh{\vztlh-\vztu}} + \max_{1\le j\le m} \abs{\va_j^\hh \xkh{\vztu-\vzx}} \\
		&\le& \max_{1\le j\le m} \norm{\va_j}\cdot \max_{1\le l\le m}\norm{\vztlh-\vztu} + \max_{1\le j\le m} \abs{\va_j^\hh \xkh{\vztu-\vzx}} \\
		&\le& c_3K^2\sqrtn  C_{6} \frac{K^7\sqrtnlmmm}{m\beta_2}\xkh{1+\frac{\abs{\vone^\T \vxi}+\sqrtn\norm{\vxi}}{m} + \frac{\norminf{\vxi}}{K^4 \log m}} \\
		&&+ C_{7}K\sqrtlm \xkh{1+\frac{\abs{\vone^\T \vxi}+ K^2\sqrtn\norm{\vxi} + K^4 n\norminf{\vxi}}{m\betat}}\\
		&\le& \tilde{C}_{3} K\sqrtlm,
\end{eqnarray*}
provided $m\gtrsim \betat^{-1} K^8 n\log^3 m$.  In particular,  we have \(\norm{\vztlh-\vzx}\le\delta_1\). This thereby gives
\[
\norm{\vztlh} \le \norm{\vztlh-\vzx}+\norm{\vzx}  \le 2,
\]
as long as $\delta_1 \le 1$.  Using the same notations  \(\lambda\),  \(\nu\) as in \eqref{simply:para:a} and applying Lemma \ref{le:ric:a}, the $I_1$ can be upper bounded with the same argument as  \eqref{i1:errorA}. Specifically,  when \(\eta\le\nu\lambda^{-2}\), with probability at least $1-m\exp (-c_0 n)-O(m^{-20})$ for some universal some constant $c_0>0$, one has
\begin{equation} \label{eq:II1}
	\norm{I_1} \le \sqrttwo (1-\frac{\nu}{2}\eta) \norm{\vztlh-\vztu}.
\end{equation}

For  the second part \(I_2\),  recall that $y_l=\aabs{\va_l^\hh\vzx}^2+\xi_l$. Then it follows from the triangle inequality and the Cauchy-Schwarz inequality that, with probability at least $1-m\exp (-c_0 n)-O(m^{-20})$, one has
\begin{eqnarray}
	&&\norm{\xkh{\aabs{\va_l^\hh \vztlh}^2 - y_l}\va_l\va_l^\hh \vztlh}  \label{eq:II2}\\
	&\le& \abs{\aabs{\va_l^\hh \vztlh}^2-\aabs{\va_l^\hh\vzx}^2}\cdot \norm{\va_l}\cdot \aabs{\va_l^\hh \vztlh} + \abs{\xi_l}\cdot \norm{\va_l}\cdot \aabs{\va_l^\hh\vztlh} \notag \\
	&\stackrel{\mbox{(i)}}{\le}& \xkh{\abs{\aabs{\va_l^\hh \vztlh}^2-\aabs{\va_l^\hh\vzx}^2}+\norminf{\vxi}} c_2 c_3 K^3\sqrtnlm \norm{\vztl} \notag\\
	&\stackrel{\mbox{(ii)}}{\le}& \xkh{3 c_2^2 K^2\log m\norm{\vztlu-\vzx} +\norminf{\vxi}}2 c_2 c_3 K^3\sqrtnlm\notag\\
	&\stackrel{\mbox{(iii)}}{\le}& \xkh{3 c_2^2 K^2\log m\xkh{\norm{\vztlu-\vztu}+\norm{\vztu-\vzx}} +\norminf{\vxi}}2 c_2 c_3 K^3\sqrtnlm \notag\\
	&\stackrel{\mbox{(iv)}}{\le}& 8c_2^3 c_3 K^5\sqrtnlmmm\xkh{C_{3}+ C_{4} \frac{\abs{\vone^\T \vxi}}{m\betat} + C_{5} K^2\frac{\sqrt{n}\norm{\vxi} + n\norminf{\vxi}}{m\betat}} \notag\\
	&&+ 2 c_2 c_3 K^3\sqrtnlm\norminf{\vxi}. \notag
\end{eqnarray}
for some universal constants $c_2, c_3>0$. Here,  (i) arises from the fact that \(\abs{\va_l^\hh\vztlh}=\abs{\va_l^\hh\vztl}\) and
\[
	\maxm{1\le l\le m} \abs{\va_l^\hh \vztl} \le c_2 K \sqrtlm\norm{\vztl}, \qquad \maxm{1\le l\le m} \norm{\va_l} \le c_3 K^2 \sqrtn
\]
due to the independence between $\va_l$ and $\vztl$ together with Lemma \ref{le:order:statistics}. And (ii) comes from
\begin{align*}
	\abs{\aabs{\va_l^\hh \vztlh}^2-\aabs{\va_l^\hh\vzx}^2} &\le \xkh{\aabs{\va_l^\hh \vztl}+\aabs{\va_l^\hh \vzx}} \abs{ \va_l^\hh\xxkh{\vztlu-\vzx}}\\
	&\le c_2^2 K^2 \log m \xkh{\norm{\vztl}+1}\norm{\vztlu-\vzx},
\end{align*}
where we use Lemma \ref{le:order:statistics} again and the fact  that $\norm{\vztl} \le 2$.
Besides, (iii) follows from the triangle inequality, and (iv) utilizes the induction hypotheses \eqref{induct:a} that
\begin{align*}
	\max_{1\le l\le m}\norm{\vztlu-\vztu}\lesssim C_{6} \frac{K^7\sqrtnlmmm}{m\beta_2}\xkh{1+\frac{\abs{\vone^\T \vxi}+\sqrtn\norm{\vxi}}{m} + \frac{\norminf{\vxi}}{\log m}}
\end{align*}
and
\[\norm{\vztu-\vzx}\le C_{3}+ C_{4} \frac{\abs{\vone^\T \vxi}}{m\betat} + C_{5} K^2\frac{\sqrt{n}\norm{\vxi} + n\norminf{\vxi}}{m\betat}\]
with $m\gtrsim \beta_2^{-1} K^8 n\log^3 m$.  Putting \eqref{eq:II1} and \eqref{eq:II2} into \eqref{eq:I1I2}, we obtain that with probability at least $1-m\exp\xkh{-c_0 n}-O(m^{-20})$, it holds
\begin{align*}
	&\max_{1\le l\le m} \dist{\vzttl,\vzttu}\\
	\le& \max_{1\le l\le m}\dkh{\sqrtonetwo\norm{I_1} + \sqrtonetwo\norm{I_2}}\\
	\le& \xkh{1-\frac{\nu\eta}{2}}\max_{1\le l\le m}\dist{\vztl,\vztu} +2 c_2 c_3 K^3\eta\frac{\sqrtnlm}{m}\norminf{\vxi}\\
	&+  8c_2^3 c_3 K^5\eta\frac{\sqrtnlmmm}{m}\xkh{C_{3}+ C_{4} \frac{\abs{\vone^\T \vxi}}{m\betat} + C_{5} K^2\frac{\sqrt{n}\norm{\vxi} + n\norminf{\vxi}}{m\betat}}\\
	\le& C_{6} \frac{K^7\sqrtnlmmm}{m\beta_2}\xkh{1+\frac{\abs{\vone^\T \vxi}+\sqrtn\norm{\vxi}}{m} + \frac{\norminf{\vxi}}{K^4 \log m}},
\end{align*}
 provided $m\gtrsim \betat^{-1}n\log^3 m$ and \(C_{6} > 0\) is a large constant. This gives the conclusion of the first part of Lemma \ref{le:ztl:a}.

Next,  we turn to upper bound $\max_{1\le l\le m}\norm{\vzttlu-\vzttu}$.  Recall the definition of $\vzttlh$ as in \eqref{eq:zhatl} and define
\[\tilde{\phi}(t+1) = \argmin{\phi\in \R} \norm{\e^{i \phi}\vzttlh-\vzx}.
\]
Set
\[
\vzttlu = \tilde{\phi}(t+1) \vzttlh.
\]
From the results of  \eqref{cond:result:1} and \eqref{eq:le331}, with high probability, it holds
\[\dist{\vztt,\vzx}\le \frac 18 \quad \mbox{and} \quad \dist{\vzttl,\vzttu}\le \frac 18.\]
This gives
 \[
 \max \dkh{\norm{\vzttlh-\vzx},\norm{\vzttu-\vzx}}\le 1/4.
 \]
Invoking Lemma \ref{le:noalign}, we have
\begin{align*}
	\norm{\vzttlu-\vzttu}\lesssim \norm{\vzttlh - \vzttu} = \dist{\vzttl,\vzttu} .
\end{align*}
Here, the last equality comes from \eqref{eq:zhatl}.
Therefore, one has
\begin{eqnarray}
	\max_{1\le l\le m}\norm{\vzttlu-\vzttu}&\lesssim&  \maxm{1\le l\le m}\dist{\vzttl,\vzttu} \notag \\
	&\le &  C_{6} \frac{K^7\sqrtnlmmm}{m\beta_2}\xkh{1+\frac{\abs{\vone^\T \vxi}+\sqrtn\norm{\vxi}}{m} + \frac{\norminf{\vxi}}{\log m}} \label{eq:ztlzti}
\end{eqnarray}
with probability exceeding $1-m\exp\xkh{-c_0 n}-O(m^{-20})$.

\subsection{Proof of Lemma \ref{le:initial1:a}}\label{pf:initial1:a}
Recall the definition of the matrix $\mm$ given in Algorithm \ref{alg:1} is
\[
\mm = \frac{\mm_0}{\beta_2}+\frac{2-\beta_1-\beta_2}{\beta_1(2-\beta_2)}\mdiag\xkh{\mm_0}+ \frac{\beta_1-1}{\beta_1} \gamma\mi -\frac{2-2\beta_2}{\beta_2(2-\beta_2)}\real{\mm_0}
\]
with
\[
\mm_0 = \frac{1}{m}\sum_{j=1}^{m} y_j\va_j\va_j^\hh \quad \mbox{and} \quad \gamma = \frac{1}{m}\sum_{j=1}^{m} y_j.
\]
According to Lemma  \ref{le:e}, it is easy to check that
\begin{equation} \label{eq:EM0}
\E \mm = \norm{\vzx}^2\mi + \vzx\vzxh + \frac{\vone^\T \vxi}{m} \mi.
\end{equation}
Note that $\check{\vz}^0$ and $\vzx$ are the leading singular vectors of $\mm$ and $\E\mm$, respectively.
 A variant of Wedin's sin$\Theta$ theorem \cite[Theorem 2.1]{Wedin} gives
\begin{equation} \label{eq:wendinineua}
	\minm{\phi\in\R} \norm{\e^{i\phi} \check{\vz}^0-\vzx}\le \frac{\sqrttwo \norm{\mm-\E\mm}}{\sigma_1\xkh{\E\mm}- \sigma_2\xkh{\mm}}.
\end{equation}
 For any constant $\delta_0>0$, as we have proved in Lemma \ref{le:aat} and Lemma \ref{le:uniform:bound},  it holds
\begin{equation}\label{cond:y-ey}
		\norm{\mm_0 - \E \mm_0}\le\delta_0\norm{\vzx}^2 \quad \mbox{and} \quad
		\norm{\gamma - \E \gamma}\le\delta_0\norm{\vzx}^2
\end{equation}
 with probability at least $1 - m\exp (-c_0 n)-O(m^{-20})$, provided $m \ge C_0\delta_0^{-2}K^8n\log^3 m$.
This gives
\begin{align*}
	\norm{\real{\mm_0} - \E \real{\mm_0}}\le\delta_0\norm{\vzx}^2\quad \mbox{and}\quad \norm{\mdiag\xkh{\mm_0} - \E \mdiag\xkh{\mm_0}}\le\delta_0\norm{\vzx}^2.
\end{align*}
 In fact, for any complex symmetric matrix $\ma\in\C^n$, one has
\[
	\norm{\real{\ma}}=\,\norm{\frac{\ma+\bar{\ma}}{2}}
	\,\le\frac{\norm{\ma}}{2} + \frac{\norm{\bar{\ma}}}{2}
	= \norm{\ma}
	\]
	and
	\[
	\norm{\mdiag\xkh{ \ma}}= \maxm{1\le k\le m} \abs{ a_{kk}}
	= \maxm{1\le k\le m} \abs{ \ve_k^\hh \ma\ve_k}
	\le \norm{ \ma},
	\]
where $a_{kk}$ represents the k-th diagonal element of $\ma$.
Therefore,  the numerator of \eqref{eq:wendinineua} can be upper bounded as
\begin{eqnarray}
	\norm{\mm-\E\mm} & \le& \xkh{\frac{1}{\beta_2} + \frac{\abs{2-\beta_1-\beta_2}}{\beta_1(2-\beta_2)}+\frac{2-2\beta_2}{\beta_2(2-\beta_2)} + \frac{\abs{\beta_1-1}}{\beta_1}} \delta_0\norm{\vzx}^2 \label{eq:MEM}\\
	& \le& \frac{5+2\beta_1}{\betat} \delta_0\norm{\vzx}^2 \le \frac{5(1+\beta_1)}{\betat} \delta_0\norm{\vzx}^2, \notag
\end{eqnarray}
where we use the fact $\beta_1>0, 0<\beta_2\le 1, \betat=\min\dkh{\beta_1,\beta_2}$, and $\abs{2-\beta_1-\beta_2}\le 2+\beta_1, \abs{\beta_1-1}\le \beta_1+1, 0\le (2-2\beta_2)/(2-\beta_2)<1$.  Regarding the denominator of \eqref{eq:wendinineua},  due to the assumption \(\abs{\vone^\T \vxi}/m \le 1/2\), it follows from \eqref{eq:EM0} that
\begin{equation} \label{eq:EM1}
	1 \le \sigma_1\xkh{\E\mm}=2+\frac{\vone^\T \vxi}{m}<3
	\end{equation}
and
\begin{equation} \label{eq:EMn}
	\frac 12\le \sigma_k\xkh{\E\mm}=1+\frac{\vone^\T \vxi}{m}<2
\end{equation}
for all $2\le k\le m$. Invoking the Weyl's inequality, we have
\begin{eqnarray}
	\sigma_1\xkh{\E\mm}- \sigma_2\xkh{\mm}&\ge &  \sigma_1\xkh{\E\mm} - \sigma_2\xkh{\E\mm} - \norm{\mm-\E\mm} \label{eq:s1sem}\\
	&\ge&  \xkh{2+\frac{\vone^\T \vxi}{m}} - \xkh{1+\frac{\vone^\T \vxi}{m}} - \frac{5(1+\beta_1)\delta_0}{\betat} \notag \\
	&\ge&  \frac{1}{2},  \notag
\end{eqnarray}
where the second inequality arises from \eqref{eq:EM1}, \eqref{eq:EMn} and \eqref{eq:MEM}, and the last inequality follows by taking $\delta_0 \le \betat/(10(1+\beta_1))$.
Putting \eqref{eq:MEM} and \eqref{eq:s1sem} into \eqref{eq:wendinineua}, we get
\begin{equation} \label{eq:distzczx}
	\dist{\check{\vz}^0,\vzx} = \minm{\phi\in\R} \norm{\e^{i\phi} \check{\vz}^0-\vzx}\le\frac{10\sqrttwo\xkh{1+\beta_1}}{\betat}\delta_0.
\end{equation}

Next, we connect the preceding bound \eqref{eq:distzczx} with the scaled singular vector $\vz^0=\sqrt{\gamma}\check{\vz}^0$. The triangle inequality implies
\begin{align*}
	\dist{\vz^0,\vzx} =& \minm{\phi\in\R} \norm{\sqrt{\gamma}\e^{i\phi} \check{\vz}^0-\vzx} \\
	\le& \minm{\phi\in\R} \dkh{\norm{\sqrt{\gamma}\e^{i\phi} \check{\vz}^0-\sqrt{\gamma}\vzx} + \norm{\sqrt{\gamma}\vzx-\vzx}}\\
	=&   \minm{\phi\in\R} \sqrt{\gamma}\norm{\e^{i\phi} \check{\vz}^0-\vzx} + \abs{\sqrt{\gamma}-1}.
\end{align*}
Note that $\E \gamma= \norm{\vzx}^2+\frac{\vone^\T \vxi}{m}$. This together with  \eqref{cond:y-ey} reveal that
\begin{align} \label{eq:gammalup}
	\abs{\gamma - 1}  \le \abs{\frac{\vone^\T \vxi}{m}} + \delta_0 \le 2\delta_0    \quad \mbox{and}\quad \frac 12 \le \gamma\le 2
\end{align}
as long as \(\abs{\vone^\T \vxi}/m \le \delta_0\) and $\delta_0<1/4$. This further implies that
\begin{equation*}
	\abs{\sqrt{\gamma}-1} = \frac{ \abs{\gamma-1}}{\sqrt{\gamma}+1}\le 2\delta_0.
\end{equation*}
To summarize, we arrive at the conclusion that for any $0<\delta<1$, with probability at least $1 - m\exp (-c_0 n) - O(m^{-10})$, it holds
\begin{eqnarray*}
	\dist{\vz^0,\vzx} &\le& \sqrt{\gamma}\dist{\check{\vz}^0,\vzx} + \abs{\sqrt{\gamma}-1}\\
	&\le& \frac{20(1+\beta_1)}{\betat}\delta_0 + 2\delta_0\\
	&\le& \frac{22(1+\beta_1)}{\betat}\delta_0\\
	&\le & \delta,
\end{eqnarray*}
provided $m\ge C_0(1+\beta_1)^2\betat^{-2}\delta^{-2} K^8 n \log^3 m$. Here,  we use the fact that $\betat=\min\dkh{\beta_1,\beta_2}\le 1+\beta_1$ in the second inequality and take $\delta_0=\frac{\betat \delta}{22(1+\beta_1)}$ in the last inequality.
Repeating the same arguments, we see that
\begin{equation*}
	\maxm{1\le l\le m}\dist{\vz^{0,\lkh},\vzx} \le \delta.
\end{equation*}

\subsection{Proof of Lemma \ref{le:initial2:a}}\label{pf:initial2:a}
Recall that $\check{\vz}^0$ and $\check{\vz}^{0,\lkh}$ are the leading singular vectors of $\mm$ and $\mm^\lkh$, respectively.
Invoke Wedin's sin$\Theta$ theorem \cite[Theorem 2.1]{Wedin} to obtain
\begin{equation}\label{ineq:initial2:1:a}
	\dist{\check{\vz}^0,\check{\vz}^{0,\lkh}} = \minm{\phi\in\R} \norm{\e^{i\phi}\check{\vz}^0-\check{\vz}^{0,\lkh}}\le \frac{\sqrttwo \norm{\xkh{\mm-\mm^\lkh}\check{\vz}^{0,\lkh}}}{\sigma_1\xkh{\mml}-\sigma_2\xkh{\mm}}.
\end{equation}
To prove the main result, we need an upper bound for $\norm{\xkh{\mm-\mm^\lkh}\check{\vz}^{0,\lkh}}$ and a lower bound for $\sigma_1\xkh{\mml}-\sigma_2\xkh{\mm}$. From the definitions of $\mm$ and $\mml$, one has
\begin{equation}\label{express:zt-ztl}
	\begin{aligned}
		\mm-\mml =& \frac{\abs{\va_l^\hh\vzx}^2+\xi_l}{m\beta_2}\xkh{\va_l\va_l^\hh-\frac{2-2\beta_2}{2-\beta_2}\real{\va_l\va_l^\hh}}+ \frac{\beta_1-1}{\beta_1} \frac{\abs{\va_l^\hh\vzx}^2+\xi_l}{m}\mi \\
		&+\frac{2-\beta_1-\beta_2}{\beta_1(2-\beta_2)}\frac{\abs{\va_l^\hh\vzx}^2+\xi_l}{m}\mdiag\xkh{\va_l\va_l^\hh}.
	\end{aligned}
\end{equation}
It comes from Lemma \ref{le:order:statistics} that with probability at least $1-m\exp\xkh{-c_0 n}-O(m^{-50})$, one has
\[
\maxm{1\le l\le m}\aabs{\va_l^\hh\vzx}^2 \le c_2^2 K^2 \log m
\]
and
\[
	\maxm{1\le l\le m} ~ \norm{\va_l\va_l^\hh \check{\vz}^{0,\lkh}} \le \maxm{1\le l\le m}\norm{\va_l}\cdot \maxm{1\le l\le m}\abs{\va_l^\hh \check{\vz}^{0,\lkh}}  \le c_2c_3 K^3 \sqrtnlm,
\]
due to the statistical independence between $\va_l$ and $\check{\vz}^{0,\lkh}$. This implies
\begin{eqnarray*}
	\norm{\real{\va_l\va_l^\hh}\check{\vz}^{0,\lkh}} &= &  \norm{\frac{\va_l\va_l^\hh
	+\overline{\va_l\va_l^\hh}}{2}\check{\vz}^{0,\lkh}} \\
	&\le &  \frac{\norm{\va_l\va_l^\hh\check{\vz}^{0,\lkh}}+\norm{\va_l\va_l^\hh\bar{\check{\vz}}^{0,\lkh}}}{2}\\
	&\le &  c_2c_3 K^3 \sqrtnlm.
\end{eqnarray*}
Moreover, with probability at least $1-O(m^{-50}n^{-50})$, it holds
\begin{align*}
	\maxm{1\le l\le m}\norm{\mdiag\xkh{\va_l\va_l^\hh}\check{\vz}^{0,\lkh}}\stackrel{\mbox{(i)}}{\le}& \maxm{1\le l\le m}\norm{\mdiag\xkh{\va_l\va_l^\hh}} \\
	\stackrel{\mbox{(ii)}}{=}& \maxm{1\le l\le m} \xkh{\maxm{1\le j\le n}\abs{a_{lj}}^2} \\
	\stackrel{\mbox{(iii)}}{\le}& c_2^2K^2\log (mn).
\end{align*}
Here, (i) follows from Cauchy-Schwarz and $\norm{\check{\vz}^0} = 1$,  (ii) comes from the fact $\mdiag\xkh{\va_l\va_l^\hh}$ is a diagonal matrix whose $j$-th diagonal element is \(\abs{a_{lj}}^2\), and (iii) arises from Lemma\ref{le:order:statistics}.
 Combining the above estimators together, we obtain the following upper bound
\begin{eqnarray}
	& &\maxm{1\le l\le m}\norm{\xkh{\mm-\mm^\lkh}\check{\vz}^{0,\lkh}} \label{eq:MMlnum}\\
	& \le &\frac{c_2^2K^2\log m+\norminf{\vxi}}{m}\xkh{\frac{2c_2c_3K^3\sqrtnlm}{\beta_2} + \frac{\beta_1+1}{\beta_1} + \frac{(2+\beta_1)c_2^2 K^2 \log\xkh{mn}}{\beta_1}} \notag \\
	& \le &\frac{3c_2^3c_3 K^5\sqrtnlmmm}{m\beta_2} + \frac{3c_2c_3 K^3\sqrtnlm\norminf{\vxi}}{m\beta_2}  \notag
\end{eqnarray}
with probability at least $1 - m\exp(-c_0 n) - O(m^{-50})$. Here, the last inequality holds true as long as
\begin{align*}
	n\log m\gg \frac{(\beta_1+1)^2\beta_2^2}{\beta_1^2}\quad \mbox{and}\quad n\log m\gg \frac{(\beta_1+2)^2\beta_2^2\log^2 (mn)}{\beta_1^2}.
\end{align*}
We now turn to estimate the denominator $\sigma_1\xkh{\mml}-\sigma_2\xkh{\mm}$ of \eqref{ineq:initial2:1:a}. As shown in \eqref{eq:MEM} that with probability at least $1 - m\exp (-c_0 n) - O(m^{-20})$, it holds
\[\norm{\mm-\E\mm}\le \frac{1}{12},
\]
provided $m\gtrsim (1+\beta_1)^2 \betat^{-2} K^8 n\log^3 m$.  Repeating the same arguments, one has
\[\maxm{1\le l\le m} \norm{\mml-\E\mml}\le \frac{1}{12}.
\]
Thus, the triangle inequality gives
\begin{eqnarray}
		\sigma_1\xkh{\mml}-\sigma_2\xkh{\mm} & \ge&\sigma_1\xkh{\E\mml}-\norm{\mml-\E\mml}-\sigma_2\xkh{\E\mm}-\norm{\mm-\E\mm} \notag \\
		&\ge& \xkh{\frac 74+\frac{\sum_{j\ne l}\xi_j}{m}}-\norm{\mml-\E\mml}-\xkh{1+\frac{\vone^\T \vxi}{m}}-\norm{\mm-\E\mm}\notag \\
		& \ge& \frac{7}{12} - \frac{\norminf{\vxi}}{m}  \ge  \frac 12,  \label{eq:sml2m}
\end{eqnarray}
where the second line comes from
\[\sigma_1\xkh{\E\mml} \ge \frac 74 + \frac{\sum_{j\ne l}\xi_j}{m},
\]
 and the last inequality follows from the fact $\norminf{\vxi}/m\le 1/12$ due to $\norminf{\vxi}\lesssim \log m$.  Putting \eqref{eq:MMlnum} and \eqref{eq:sml2m} into \eqref{ineq:initial2:1:a},  we derive
\begin{equation} \label{eq:zhzol}
	\maxm{1\le l\le m}\dist{\check{\vz}^0,\check{\vz}^{0,\lkh}}\le \frac{6\sqrttwo c_2^3c_3 K^5\sqrtnlmmm}{m\beta_2} + \frac{6\sqrttwo c_2c_3 K^3\sqrtnlm\norminf{\vxi}}{m\beta_2}.
\end{equation}

Next, we translate the preceding bounds to the scaled version, namely, $\maxm{1\le l\le m} \dist{\vz^{0,\lkh},\tilde{\vz}^0}$. Recalling \eqref{eq:gammalup}, one sees that with high probability, it holds
\begin{align*}
	\frac 12\le \gamma \le 2\quad \mbox{and}\quad \frac 12\le \gl \le 2
\end{align*}
for each $1\le l\le m$. Here, $\gamma = \frac{1}{m}\sum_{j=1}^{m} y_j$ and $\gl = \frac{1}{m}\sum_{j=1,j \ne l}^{m} y_j$. From the definitions of $\vz^{0,\lkh}$ and $\tilde{\vz}^0$,  we have
\begin{align*}
	\dist{\vz^{0,\lkh},\tilde{\vz}^0} =& \minm{\phi\in\R} \norm{\sqrt{\gl}\e^{i\phi} \check{\vz}^{0,\lkh}-\tilde{\vz}^0} \\
	\le& \minm{\phi\in\R} \left\{\norm{\sqrt{\gl}\e^{i\phi}\check{\vz}^{0,\lkh}-\sqrt{\gl}\e^{i\phi_0}\check{\vz}^0} + \norm{\sqrt{\gl}\e^{i\phi_0}\check{\vz}^0-\sqrt{\gamma}\e^{i\phi_0}\check{\vz}^0}\right\}\\
	=& \minm{\phi\in\R} \sqrt{\gl}\norm{\e^{i\phi}\check{\vz}^{0,\lkh}-\e^{i\phi_0}\check{\vz}^0} + \abs{\sqrt{\gamma}-\sqrt{\gl}}\\
	=& \sqrt{\gl}\;\dist{\check{\vz}^0,\check{\vz}^{0,\lkh}} + \abs{\sqrt{\gamma}-\sqrt{\gl}},
\end{align*}
where $\phi_0 = \argmin{\phi\in\R} \norm{\e^{i\phi}\vz^0-\vzx}$. Observe that with probability at least $1-O(m^{-50})$, it holds
\begin{align*}
	\maxm{1\le l\le m} \abs{\sqrt{\gamma}-\sqrt{\gl}}\le& \maxm{1\le l\le m} \frac{\abs{\abs{\va_l^\hh\vzx}^2+\xi_l}/m}{\sqrt{\gamma} + \sqrt{\gl}}\\
	\le& \maxm{1\le l\le m} \frac{\abs{\va_l^\hh\vzx}^2}{m} + \frac{\norminf{\vxi}}{m}
	\le \frac{c_2^2 K^2\log m}{m} + \frac{\norminf{\vxi}}{m}.
\end{align*}
This immediately gives that
\begin{align*}
	&\maxm{1\le l\le m} \dist{\vz^{0,\lkh},\tilde{\vz}^0}\\
	\le& \sqrttwo\maxm{1\le l\le m} \dist{\check{\vz}^0,\check{\vz}^{0,\lkh}} + \frac{c_2^2 K^2\log m}{m} + \frac{\norminf{\vxi}}{m}\\
	\le& \frac{12 c_2^3 c_3 K^5\sqrtnlmmm + c_2^2 K^2 \log m}{m\beta_2} + \frac{12 c_2 c_3 K^3\sqrtnlm + 1}{m\beta_2}\norminf{\vxi}\\
	\le& C_{6} \frac{K^7\sqrtnlmmm + K^7\sqrtnlm\norminf{\vxi}}{m\beta_2}
\end{align*}
holds with probability exceeding $1 - m\exp (-c_0 n) - O(m^{-10})$ as long as $K\ge 1$ and $C_{6}\ge \max\dkh{12 c_2^3 c_3 + c_2^2, 12 c_2 c_3 + 1}$. Here, the second inequality arises from \eqref{eq:zhzol}.

Finally, using the same arguments as in deriving \eqref{eq:ztlzti}, one can apply Lemma \ref{le:noalign} to obtain
\begin{align*}
		\max_{1\le l\le m}\norm{\tilde{\vz}^{0,\lkh}-\tilde{\vz}^0}\lesssim \maxm{1\le l\le m}\dist{\vz^{0,\lkh},\tilde{\vz}^0}
		\le C_{6} \frac{K^7\sqrtnlmmm + K^7\sqrtnlm\norminf{\vxi}}{m\beta_2}.
\end{align*}

\section{Appendix C: Proof with Assumption B}
Apply Lemma \ref{le:e} with $\beta_1 = 0$, one has
\begin{align*}
	\E \nabla f_{\vz}(\vz) =& \xkh{2\norm{\vz}^2-\norm{\vzx}^2}\vz-\vzxh \vz \vzx+\xkh{1-\beta_2}\xkh{\bar{\vz} \vz^\T-\vzxc\vzxt}\vz\\
	&-\xkh{2-\beta_2}\xkh{\md_1(\vz)-\md_1(\vzx)}\vz-\frac{\vone^\T \vxi}{m}\vz,    \\
	\E \nabla^2 f(\vz) =& \begin{bmatrix}
		\ma & \mb \\ \mb^\hh & \bar{\ma}
	\end{bmatrix},
\end{align*}
where
\begin{align*}
	\ma =& \xkh{2\norm{\vz}^2-\norm{\vzx}^2}\mi + \vz\vz^\hh - \vzx\vzxh+\xkh{1-\beta_2}\xkh{2\bar{\vz}\vz^\T-\vzxc\vzxt}  \\
	& - \xkh{2-\beta_2}\xkh{2\md_1(\vz)-\md_1(\vzx)}-\frac{\vone^\T \vxi}{m}\mi,   \\
	\mb =& 2\vz\vz^\T + (1-\beta_2)\xkh{\vz^\T \vz}\mi - \xkh{2-\beta_2}\md_2(\vz).
\end{align*}

\subsection{Proof of Lemma \ref{le:ric:b}}\label{pf:ric:b}
We decompose \( \nabla^2 f(\vz) \) into the following two components,
\begin{equation*}
	\nabla^2 f(\vz) = \nabla^2 F(\vzx) + \xkh{\nabla^2 f(\vz) - \nabla^2 F(\vzx)},
\end{equation*}
and subsequently validate our conclusion using the following two lemmas.

\begin{lemma}\label{le:subric1:b}
	Assume that all conditions stated in Lemma \ref{le:ric:b} are satisfied.  It holds
	\begin{equation*}
		\norm{\nabla^2 f(\vzx)}\le 8\quad and \quad \vu^\hh \nabla^2 f(\vzx) \vu\ge \frac{7\beta_2}{10}\norm{\vu}^2.
	\end{equation*}
\end{lemma}

\begin{lemma}\label{le:subric2:b}
	Suppose that  $m\ge C_{10} \beta_2^{-2} K^8 n \log^3 m$ for some universal constant $ C_{10} > 0 $. Then with probability at least $1-O(m^{-10})-m\exp (-c_0 n)$, it holds
	\begin{equation*}
		\norm{\nabla^2 f(\vz)-\nabla^2 F(\vzx)} \le \frac{9\beta_2}{20}
	\end{equation*}
	for all $\vz$ satisfies \eqref{cond1:ric:b} and \eqref{cond2:ric:b}, where $c_0>0$ is some universal constant.
\end{lemma}

Through the aforementioned two lemmas, one has
\begin{align*}
	\norm{\nabla^2 f(\vz)} \le \norm{\nabla^2 F(\vzx)} + \norm{\nabla^2 f(\vz) - \nabla^2 F(\vzx)} \le 9,
\end{align*}
and
\begin{align*}
	\vu^\hh \nabla^2 f(\vz) \vu \ge& \ \vu^\hh\nabla^2 F(\vzx)\vu + \vu^\hh\xkh{\nabla^2 f(\vz) - \nabla^2 F(\vzx)}\vu\\
	\ge& \ \vu^\hh\nabla^2 F(\vzx)\vu - \norm{\vu}^2 \cdot \norm{\nabla^2 f(\vz) - \nabla^2 F(\vzx)}\\
	\ge& \ \frac{\beta_2}{4}\norm{\vzx}^2\cdot\norm{\vu}^2.
\end{align*}

\subsubsection{Proof of Lemma \ref{le:subric1:b}}
Use Lemma \ref{le:e} to obtain that
\begin{equation}
	\nabla^2 F(\vzx)=
	\begin{bmatrix}
		\ma_2 & \mb_2 \\
		\overline{\mb_2} & \overline{\ma_2}
	\end{bmatrix}
\end{equation}
where
\begin{eqnarray*}
	&&\ma_2 = \norm{\vzx}^2 \mi+\vzx\vzxh+(1-\beta_2)\overline{\vzx}\vzxt-(2-\beta_2) \md_1(\vzx),\\
	&&\mb_2 = 2\vzx\vzxt + (1-\beta_2) (\vzxt\vzx)\mi-(2-\beta_2) \md_2(\vzx).
\end{eqnarray*}
The i-th component of $\vzx$ is denoted as $z_i^*$ for $1\le i\le n$. First, \( \nabla^2 F(\vzx) \) can be decomposed as follows,
\begin{align*}
	\nabla^2 F(\vzx) =& \norm{\vzx}^2\begin{bmatrix}
		\mi\ \ & \mzero \\
		\mzero\ \ & \mi
	\end{bmatrix}
	+ \begin{bmatrix}
		\vzx\vzxh\  & \vzx\vzxt \\
		\vzxc\vzxh\  & \vzxc\vzxt
	\end{bmatrix}
	+\begin{bmatrix}
		\mzero & \vzx\vzxt \\
		\vzxc\vzxh & \mzero
	\end{bmatrix}\\
	&+(1-\beta_2)\begin{bmatrix}
		\mzero & \xkh{\vzxt\vzx}\mi \\
		\overline{\xkh{\vzxt\vzx}}\mi & \mzero
	\end{bmatrix}
	+(1-\beta_2)\begin{bmatrix}
		\vzxc\vzxt & \mzero \\
		\mzero & \vzx\vzxh
	\end{bmatrix}\\
	&-(2-\beta_2)\begin{bmatrix}
		\md_1(\vzx)\  & \md_2(\vzx) \\
		\overline{\md_2(\vzx)}\  & \md_1(\vzx)
	\end{bmatrix}-(\frac{\vone^\T \vxi}{m})\begin{bmatrix}
		\mi\  & \mzero \\
		\mzero\  & \mi
	\end{bmatrix}.
\end{align*}
As we have demonstrated in Section  \ref{pf:subric1:a}, it holds
{\small \begin{align*}
		&\norm{\begin{bmatrix}
				\mzero & \vzx\vzxt \\
				\vzxc\vzxh & \mzero
		\end{bmatrix}} \le 1,\
		\norm{(1-\beta_2)\begin{bmatrix}
				\mzero & \xkh{\vzxt\vzx}\mi \\
				\overline{\xkh{\vzxt\vzx}}\mi & \mzero
		\end{bmatrix}}\le 1,\\
		&\norm{(1-\beta_2)\begin{bmatrix}
				\vzxc\vzxt & \mzero \\
				\mzero & \vzx\vzxh
		\end{bmatrix}}\le 1,\
		\norm{\begin{bmatrix}
				\mzero & \vzx\vzxt \\
				\vzxc\vzxh & \mzero
		\end{bmatrix}} \le 2,\
		\norm{(\frac{\vone^\T \vxi}{m})\begin{bmatrix}
				\mi\  & \mzero \\
				\mzero\  & \mi
		\end{bmatrix}}\le 2,
\end{align*}}
as long as $\abs{\vone^T\vxi}\le \beta_2 m\le 2 m$.
In addition,  under the assumption B, one can check that the matrix
\[\begin{bmatrix}
	\md_1(\vzx)\  & \md_2(\vzx) \\
	\overline{\md_2(\vzx)}\  & \md_1(\vzx)
\end{bmatrix}\] is positive semi-definite. In fact, for any \( \vx=(x_1,x_2,\cdots,x_n)^\T,\vy=(y_1,y_2,\cdots,y_n)^\T\in\C^n \), we have
\begin{align*}
	\begin{bmatrix}
		\vx^\hh\ \vy^\hh
	\end{bmatrix}
	\begin{bmatrix}
		\md_1(\vzx) & \md_2(\vzx) \\
		\overline{\md_2(\vzx)} & \md_1(\vzx)
	\end{bmatrix}
	\begin{bmatrix}
		\vx\\ \vy
	\end{bmatrix} = \sum_{i=1}^{n} \xkh{\abs{x_i}^2\abs{z_i^*}^2+\abs{y_i}^2\abs{z_i^*}^2 + 2\real{\bar{x}_i y_i z_i^{*2}}}\ge 0,
\end{align*}
due to the fact that $\abs{x_i}^2+\abs{y_i}^2\ge 2\abs{x_i}\abs{y_i}$ and $\real{\bar{x}_i y_i z_i^{*2}}\ge -\abs{x_i}\abs{y_i}\abs{z_i^*}^2$.
Therefore,
 \[\norm{\nabla^2 F(\vzx)}\le 8.\]

Next, we examine the restricted strong convexity for $\nabla^2 F(\vzx)$. For convenience, we denote $\vw = \vz_1-\vz_2\in\C^n$ and
the i-th component of $\vw$ as $w_i$. Then
\begin{equation*}
	\vu = \begin{bmatrix}
		\vw\\\overline{\vw}
	\end{bmatrix}.
\end{equation*}
Recall the claims \eqref{j1:a}, \eqref{j1:b} and \eqref{j2} and note that $\beta_1 = 0$ in Assumption B.  One has
 \begin{align*}
		2J_1 \ge& 2(1-\beta_2)\sum_{i>j} \abs{\zxi\wjc+\zxj\wic}^2 + 2\beta_2\norm{\vzx}^2 \norm{\vw}^2-2\beta_2\sumin \abs{\zxi}^2\abs{\wi}^2\\
		J_2 + \bar{J_2} \ge & 2(1-\beta_2)\sum_{i>j} \real{\xkh{\zxi\wjc+\zxj\wic}^2} - 2\beta_2\sumin \real{\zxis\wics} - 10 \norm{\vzx-\vz_1}\norm{\vw}^2,
\end{align*}
and therefore,
\begin{align*}
	\vu^\hh\nabla^2 F(\vzx)\vu =&\, 2J_1 + J_2 + \bar{J_2} - 2\frac{\vone^\T \vxi}{m} \norm{\vw}^2 \\
	\stackrel{\mbox{(i)}}{\ge}&4(1-\beta_2)\sum_{i>j} \xkh{\real{\zxi\wjc + \zxj\wic}}^2 - 4\beta_2\sumin \xkh{\real{\zxi\wic}}^2\\
	& + 2\beta_2 \norm{\vzx}^2 \norm{\vw}^2 - 10 \norm{\vzx-\vz_1}\norm{\vw}^2 - 2 \frac{\vone^\T \vxi}{m} \norm{\vw}^2\\
	\stackrel{\mbox{(ii)}}{\ge}& \frac 85 \beta_2 \norm{\vw}^2 - 10 \norm{\vzx-\vz_1} \norm{\vw}^2 - 2 \frac{\vone^\T \vxi}{m} \norm{\vw}^2\\
	\stackrel{\mbox{(iii)}}{\ge}& \frac 75 \beta_2 \norm{\vw}^2.
\end{align*}
Here, the first inequality (i) uses the identity $\abs{\alpha}^2+\real{\alpha^2}=2\xkh{\real{\alpha}}^2$ for any $\alpha\in\C$ again. The second inequality (ii) arises from
\[(1-\beta_2)\sum_{i>j} \xkh{\real{\zxi\wjc + \zxj\wic}}^2 \ge 0\]
and
\begin{equation*}
	\sumin \xkh{\real{\zxic\wi}}^2\le \sumin \abs{\zxi}^2\abs{\wi}^2\le \norminf{\vzx}^2\sumin \abs{\wi}^2\le \mu \norm{\vzx}^2\norm{\vw}^2\le \frac{1}{10}\norm{\vw}^2
\end{equation*}
with $\mu\le 1/10$.
And the last inequality (iii) utilizes the facts that
\begin{align*}
	\norm{\vzx-\vz_1}\le \delta_1\le \frac{\beta_2}{100}\quad \mbox{and}\quad \frac{\abs{\vone^\T\vxi}}{m}\le \frac{\beta_2}{20}
\end{align*}
as long as $\abs{\vone^\T \vxi}\lesssim \beta_2 m$ and $m\gtrsim \beta_2^{-1} n \log^3 m$.
As a result, it holds
\[\vu^\hh \nabla^2 F (\vzx) \vu\ge \frac{7\beta_2}{10} \norm{\vu}^2\]
due to $\norm{\vu}^2 = 2\norm{\vw}^2$.

\subsubsection{Proof of Lemma \ref{le:subric2:b}}
 The proof resembles that of Lemma \ref{le:subric2:a}, and we omit it here.

\section{Proof of Lemma \ref{le:pre:uniform:bound} }\label{pf:pre:uniform:bound}
We first establish \eqref{eq:nabal11}. For any $\vz \in \C^n$ obeying \(\max_{1\le j\le m} |\va_j^\hh \vz| \le c_2 K \sqrt{\log m}\norm{\vz}\), it holds
\[
\norm{\frac1m \sum_{j=1}^m \aabs{\va_j^\hh \vz}^2\va_j\va_j^\hh - \E \zkh{\aabs{\va^\hh \vz}^2\va\va^\hh }}  = \norm{\frac1m \sum_{j=1}^m \aabs{\va_j^\hh \vz}^2\va_j\va_j^\hh\1_{\dkh{|\va_j^\hh \vz| \le c_2 K \sqrt{\log m}}}  - \E \zkh{\aabs{\va^\hh \vz}^2\va\va^\hh }}.
\]
In what follows, we shall prove
\begin{equation} \label{eq:aj2tru}
\norm{\frac1m \sum_{j=1}^m \aabs{\va_j^\hh \vz}^2\va_j\va_j^\hh\1_{\dkh{|\va_j^\hh \vz| \le c_2 K \sqrt{\log m}}}  - \E \zkh{\aabs{\va^\hh \vz}^2\va\va^\hh }}  \le \delta \norm{\vz}^2
\end{equation}
holds simutaneously  for all $\vz \in \C^n$ by the standard covering argument. To this end,  without loss of generality, we assume that  $\norm{\vz}=1$.  Then  the triangle inequality to gives
	\begin{eqnarray*}
		&&\norm{\frac1m \sum_{j=1}^m \aabs{\va_j^\hh \vz}^2\va_j\va_j^\hh\1_{\dkh{|\va_j^\hh \vz| \le c_2 K \sqrt{\log m}}}  - \E \zkh{\aabs{\va^\hh \vz}^2\va\va^\hh }}  \\
		&\le& \underbrace{\norm{\frac1m \sum_{j=1}^m \abs{\va_j^\hh \vz}^2 \va_j\va_j^\hh \1_{\dkh{|\va_j^\hh \vz| \le c_2 K \sqrt{\log m}}} - \E\zkh{ \abs{\va_1^\hh \vz}^2 \va_1\va_1^\hh \1_{\dkh{|\va_1^\hh \vz| \le c_2 K \sqrt{\log m}}}}}}_{:=\theta_1} \\
		&& + \underbrace{\norm{\E\zkh{ \abs{\va_1^\hh \vz}^2 \va_1\va_1^\hh \1_{\dkh{|\va_1^\hh \vz| \le c_2 K \sqrt{\log m}}}} - \E \zkh{\aabs{\va^\hh \vz}^2\va\va^\hh } }}_{:=\theta_2}.
	\end{eqnarray*}	
For the second term $\theta_2$,  recall that $\va $ is an independent copy of $\va_1$. Thus,
	\begin{eqnarray*}
		\theta_2 &=& \norm{ \E\zkh{ \aabs{\va^\hh \vz}^2 \va\va^\hh \1_{\dkh{|\va^\hh \vz| > c_2 K \sqrt{\log m}}}}} \\
		&= & \max_{\vy \in \mathcal{S}_{\C}^{n-1}} \abs{ \E\zkh{ \aabs{\va^\hh \vz}^2 \aabs{\va^\hh \vy}^2 \1_{\dkh{|\va_1^\hh \vz| > c_2 K \sqrt{\log m}}}} } \\
		&\le & \max_{\vy \in \mathcal{S}_{\C}^{n-1}}  \xkh{\E\abs{\va^\hh \vz}^8 }^{1/4}  \xkh{\E\abs{\va^\hh \vy}^8 }^{1/4} \xkh{\E\zkh{\1_{\dkh{|\va^\hh \vz| > c_2 K \sqrt{\log m}}}} }^{1/2} \\
		& \lesssim &\ K^4 m^{-5}
	\end{eqnarray*}
	where the third line follows from the Cauchy-Schwartz inequality and the last inequality comes from \eqref{eq:subEp} that  $\xkh{\E\abs{\va^\hh \vz}^8}^{\frac 18}\lesssim K$ and $\xkh{\E\abs{\va^\hh \vy}^8}^{\frac 18}\lesssim K$ and the fact
	\begin{equation}\label{eq:esti of expe}
	\E\zkh{\1_{\dkh{|\va^\hh \vz| > c_2 K \sqrt{\log m}}}}=\PP\xkh{ \aabs{\va^\hh \vz} > c_2 K \sqrtlm }   \lesssim m^{-50}.
	\end{equation}
	Here, the last inequality arises from Lemma \ref{le:order:statistics}.
%
	
	Regarding $\theta_1$, we first fix some $\vz\in \mathcal{S}_{\C}^{n-1}$. Let $\mathcal{N}_{\frac14}$ be a $1/4$-net of the unit sphere $\mathcal{S}_{\C}^{n-1}$, which has cardinality at most $81^n$. Then \cite[Exercise 4.4.3]{Vershynin2018} implies
	{\footnotesize \begin{align*}
			& \norm{\frac1m \sum_{j=1}^m \abs{\va_j^\hh \vz}^2 \va_j\va_j^\hh \1_{\dkh{|\va_j^\hh \vz| \le c_2 K \sqrt{\log m}}} - \E\zkh{ \abs{\va_1^\hh \vz}^2 \va_1\va_1^\hh \1_{\dkh{|\va_1^\hh \vz| \le c_2 K \sqrt{\log m}}}}} \\
			\le & 2 \sup_{\vy \in \mathcal{N}_{\frac14}} \abs{\frac1m \sum_{j=1}^m \abs{\va_j^\hh \vz}^2 \abs{\va_j^\hh\vy}^2 \1_{\dkh{|\va_j^\hh \vz| \le c_2 K \sqrt{\log m}}} - \E\zkh{ \abs{\va_1^\hh \vz}^2 \abs{\va_1^\hh\vy}^2 \1_{\dkh{|\va_1^\hh \vz| \le c_2 K \sqrt{\log m}}}}}.
	\end{align*}}
	For any fixed $\vy \in \mathcal{N}_{\frac14}$, note that
	\[
	\norms{ \abs{\va_j^\hh \vz}^2\abs{\va_j^\hh\vy}^2 \1_{\dkh{|\va_j^\hh \vz| \le c_2 K \sqrt{\log m}}}}_{\psi_1} \le  c_2^2 K^2 \log m \norms{\abs{\va_j^\hh\vy}^2}_{\psi_1} \lesssim K^4 \log m
	\]
	where $\norms{\cdot}_{\psi_1}$ denotes the sub-exponential norm. This gives
	{\small\[
		\norms{ \abs{\va_j^\hh \vz}^2\abs{\va_j^\hh\vy}^2 \1_{\dkh{|\va_j^\hh \vz| \le c_2 K \sqrt{\log m}}}- \E \zkh{\abs{\va_1^\hh \vz}^2\abs{\va_1^\hh\vy}^2\1_{\dkh{|\va_1^\hh \vz| \le c_2 K \sqrt{\log m}}} } }_{\psi_1}  \lesssim K^4 \log m .
		\]}
	Here, we use the fact $\norms{X-\E X}_{\psi_2} \lesssim \norms{X}_{\psi_2}$ for any subgaussian random variable $X$ \cite[Theorem 2.6.8]{Vershynin2018}. Applying the Bernstein's inequality, one obtains that for any $0<t \le 1$, with probability at least $1-2\exp(-c t^2 m)$ it holds
	{\footnotesize \begin{align*}
			\abs{\frac1m \sum_{j=1}^m \abs{\va_j^\hh \vz}^2\abs{\va_j^\hh\vy}^2  \1_{\dkh{|\va_j^\hh \vz| \le c_2 K \sqrt{\log m}}} - \E\zkh{ \abs{\va_1^\hh \vz}^2 \abs{\va_1^\hh\vy}^2 \1_{\dkh{|\va_1^\hh \vz| \le c_2 K \sqrt{\log m}}}}}
			\le  t K^4 \log m
	\end{align*}}
	where $c>0$ is a universal constant.
	Setting $t=\sqrt{C_{3} n\log m/m}$ for some sufficiently large constant $C_{3}>0$, with probability exceeding $1-2\exp(-c C_{3} n \log m)$ it holds
	\begin{align*}
		\abs{\frac1m \sum_{j=1}^m \abs{\va_j^\hh \vz}^2\abs{\va_j^\hh\vy}^2  \1_{\dkh{|\va_j^\hh \vz| \le c_2 K \sqrt{\log m}}} - \E\zkh{ \abs{\va_1^\hh \vz}^2 \abs{\va_1^\hh\vy}^2 \1_{\dkh{|\va_1^\hh \vz| \le c_2 K \sqrt{\log m}}}}} \\
		\lesssim K^4 \sqrt{\frac{n \log^3 m}{m}}.
	\end{align*}
	Taking the union bound over $\mathcal{N}_{\frac14}$ gives
	\begin{align}
		\norm{\frac1m \sum_{j=1}^m \abs{\va_j^\hh \vz}^2 \va_j\va_j^\hh \1_{\dkh{|\va_j^\hh \vz| \le c_2 K \sqrt{\log m}}} - \E\zkh{ \abs{\va_1^\hh \vz}^2 \va_1\va_1^\hh \1_{\dkh{|\va_1^\hh \vz| \le c_2 K \sqrt{\log m}}}}} \notag\\
		\lesssim K^4 \sqrt{\frac{n \log^3 m}{m}} \label{eq:u14z}
	\end{align}
	with probability at least $1-81^n\cdot 2\exp(-c C_{3} n\log m)$ for the fixed $\vz$.
	
	Now let $\mathcal{N}_\epsilon$ be an $\epsilon$-net of the unit sphere with $\epsilon=O(m^{-3})$, which has cardinality at most $(1+\frac{2}{\epsilon})^{2n}$. Then for any $\vz\in\C^n$ with unit norm, there exist $\vz_0 \in \mathcal{N}_\epsilon$ such that $\norm{\vz-\vz_0}\le \epsilon$. The triangle inequality reveals that
	{\small \begin{eqnarray*}
			&&\norm{\frac1m \sum_{j=1}^m \abs{\va_j^\hh \vz}^2\va_j\va_j^\hh  \1_{\dkh{|\va_j^\hh \vz| \le c_2 K \sqrt{\log m}}} - \E\zkh{ \abs{\va_1^\hh \vz}^2\va_1\va_1^\hh\1_{\dkh{|\va_1^\hh \vz| \le c_2 K \sqrt{\log m}}}}}\\
			& \le& \underbrace{\norm{\frac1m \sum_{j=1}^m \abs{\va_j^\hh \vz}^2 \va_j\va_j^\hh  \1_{\dkh{|\va_j^\hh \vz| \le c_2 K \sqrt{\log m}}} - \frac1m \sum_{j=1}^m \abs{\va_j^\hh \vz_0}^2 \va_j\va_j^\hh  \1_{\dkh{|\va_j^\hh \vz_0| \le c_2 K \sqrt{\log m}}}}}_{:=I_1} \\
			&& + \underbrace{\norm{\frac1m \sum_{j=1}^m \abs{\va_j^\hh \vz_0}^2 \va_j\va_j^\hh  \1_{\dkh{|\va_j^\hh \vz_0| \le c_2 K \sqrt{\log m}}}- \E\zkh{ \abs{\va_1^\hh \vz_0}^2 \va_1\va_1^\hh  \1_{\dkh{|\va_1^\hh \vz_0| \le c_2 K \sqrt{\log m}}}}}}_{:=I_2}\\
			&& + \underbrace{\norm{\E\zkh{ \abs{\va_1^\hh \vz_0}^2 \va_1\va_1^\hh  \1_{\dkh{|\va_1^\hh \vz_0| \le c_2 K \sqrt{\log m}}}}- \E\zkh{ \abs{\va_1^\hh \vz}^2 \va_1\va_1^\hh  \1_{\dkh{|\va_1^\hh \vz| \le c_2 K \sqrt{\log m}}}}}}_{:=I_3}.
	\end{eqnarray*}}
	
For $I_2$, it follows from \eqref{eq:u14z} that
\[
I_2\lesssim K^4 \sqrt{\frac{n \log^3 m}{m}}
\]
holds  for all  $\vz_0\in\mathcal{N}_\ep$ with probability at least $1-\abs{\mathcal{N}_\ep}\cdot 81^n\cdot 2\exp(-cC_{3} n\log m)\ge 1 - O\xkh{m^{-20}}$ as long as $C_3$ is sufficiently large.
	With regard to $I_3$, one sees that
	\begin{align*}
		I_3\le & \underbrace{\norm{\E\zkh{ \abs{\va_1^\hh \vz_0}^2 \va_1\va_1^\hh \1_{\dkh{|\va_1^\hh \vz_0| \le c_2 K \sqrt{\log m}}}}- \E\zkh{ \abs{\va_1^\hh \vz_0}^2 \va_1\va_1^\hh}}}_{:=J_1} \\
		& + \underbrace{\norm{\E\zkh{ \abs{\va_1^\hh \vz_0}^2 \va_1\va_1^\hh}- \E\zkh{ \abs{\va_1^\hh \vz}^2 \va_1\va_1^\hh}}}_{:=J_2} \\
		& + \underbrace{\norm{\E\zkh{ \abs{\va_1^\hh \vz}^2 \va_1\va_1^\hh}- \E\zkh{ \abs{\va_1^\hh \vz}^2 \va_1\va_1^\hh \1_{\dkh{|\va_1^\hh \vz| \le c_2 K \sqrt{\log m}}}}}}_{:=J_3}.
	\end{align*}
	Similar to $\theta_2$, we have that $J_1,J_3\lesssim K^4 m^{-5}$. For the second term, it follows from \eqref{cond:diff of E} that $J_2\lesssim (1+\beta_1)\ep\lesssim m^{-2}$. These reveal that
	\[
	I_3\lesssim K^4 m^{-5} + m^{-2}.
	\]
	Finally, we turn to  $I_1$. For convenience, we denote $S_j=\dkh{\vz | \abs{\va_j^\hh \vz}\le c_2 K\sqrtlm}$. Then $I_1$ can be upper bounded as
	\begin{align*}
		I_1\le& \underbrace{\norm{\frac 1m \sum_{j=1}^{m} \xkh{\abs{\va_j^\hh \vz}^2-\abs{\va_j^\hh \vz_0}^2}\va_j\va_j^\hh\1_{\dkh{\vz,\vz_0\in S_j}}}}_{:=L_1} \\
		&+ \underbrace{\norm{\frac 1m \sum_{j=1}^{m} \abs{\va_j^\hh \vz}^2\va_j\va_j^\hh \1_{\dkh{\vz\in S_j,\vz_0\notin S_j}}}}_{:=L_2} + \underbrace{\norm{\frac 1m \sum_{j=1}^{m} \abs{\va_j^\hh \vz_0}^2\va_j\va_j^\hh \1_{\dkh{\vz\notin S_j,\vz_0\in S_j}}}}_{:=L_3}.
	\end{align*}
 We will consider each item separately.
	\begin{itemize}
		\item First, the Cauchy-Schwarz yields
		\begin{align*}
			L_1\le& \maxm{1\le j\le m} \norm{ \xkh{\abs{\va_j^\hh \vz}-\abs{\va_j^\hh \vz_0}} \xkh{\abs{\va_j^\hh \vz}+\abs{\va_j^\hh \vz_0} }\va_j\va_j^\hh\1_{\dkh{\vz,\vz_0\in S_j}}} \\
			\le& \maxm{1\le j\le m} 2 \norm{\va_j}^4 \norm{\vz-\vz_0} \\
			\lesssim& K^8 m^{-1}
		\end{align*}
		where the last line holds due to $\norm{\vz-\vz_0}\lesssim m^{-3}$,  $m\gtrsim n$, and the fact that  $\maxm{1\le j\le m} \norm{\va_j}\le c_3K^2 \sqrtn$ with probability at least $1-m\exp\xkh{-c_0 n}$.
		
		\item Next,  for  $L_2$, we have
		\begin{align*}
			L_2 \le& \underbrace{\norm{\frac 1m \sum_{j=1}^{m} \abs{\va_j^\hh \vz_0}^2\va_j\va_j^\hh \1_{\dkh{\vz\in S_j,\vz_0\notin S_j}}}}_{:=\tau_1} \\
			&+ \underbrace{\norm{\frac 1m \sum_{j=1}^{m} \xkh{\abs{\va_j^\hh \vz}^2-\abs{\va_j^\hh \vz_0}^2}\va_j\va_j^\hh \1_{\dkh{\vz\in S_j,\vz_0\notin S_j}}}}_{:=\tau_2}.
		\end{align*}
		Using the same technique as  $L_1$, one has $\tau_2\lesssim K^8 m^{-1}$.
		To control $\tau_1$, we claim that
		\begin{equation}\label{ineq:indicator 1}
			\1_{\dkh{\vz\in S_j,\vz_0\notin S_j}}\le \1_{\dkh{c_2K\sqrtlm\le \abs{\va_j^\hh\vz_0}\le c_2K \sqrtlm + c_3 K^2 \sqrtn \ep}}.
		\end{equation}
		Indeed, if we assume $\aabs{\va_j^\hh\vz_0}\ge c_2K \sqrtlm + c_3 K^2 \sqrtn \ep$, then one has
		\begin{align*}
			\aabs{\va_j^\hh\vz} \ge \aabs{\va_j^\hh\vz_0} - \maxm{1\le j\le m}\norm{\va_j} \norm{\vz_0-\vz} \ge c_2K\sqrtlm
		\end{align*}
		which is contradictory to $\vz\in S_j$.
		By taking $\ep=O(m^{-3})$, it then follows from \eqref{ineq:indicator 1} that for any $\vz\in S_j,\vz_0\notin S_j$, it holds
		\begin{equation} \label{eq:ajz0}
			\aabs{\va_j^\hh\vz_0}\le c_2K \sqrtlm + c_3 K^2 \sqrtn \ep \le 2c_2 K\sqrtlm.
		\end{equation}
		Therefore, the triangle inequality gives
		\begin{eqnarray*}
		\tau_1 &\le & \abs{ \tau_1- \E \tau_1}+ \E \tau_1 \lesssim  K^4 \sqrt{\frac{n\log ^3 m}{m}}+ K^4 m^{-5},
		\end{eqnarray*}
		where the last inequality comes from \eqref{eq:u14z} together with the conditon \eqref{eq:ajz0}, and the fact
		\begin{align*}
			\E \zkh{\abs{\va_1^\hh \vz_0}^2\va_1\va_1^\hh \1_{\dkh{\vz\in S_1,\vz_0\notin S_1}}}\lesssim K^4 m^{-5}
		\end{align*}
		due to $\1_{\dkh{\vz\in S_1,\vz_0\notin S_1}}\le \1_{\dkh{\vz_0\notin S_1}}\le \1_{\dkh{\abs{\va_1^\hh \vz_0}\ge c_2 K \sqrtlm}}$. 	
		\item The last item $L_3$ can be controlled by the same argument for $L_2$, namely
		\begin{align*}
			L_3\lesssim K^4 m^{-5} + K^4 \sqrt{\frac{n\log ^3 m}{m}},
		\end{align*}
		where we utilize that
		\begin{equation}\label{ineq:indicator 2}
			\1_{\dkh{\vz\notin S_j,\vz_0\in S_j}}\le \1_{\dkh{c_2K\sqrtlm - c_3 K^2 \sqrtn \ep\le \abs{\va_j^\hh\vz_0}\le c_2K \sqrtlm}}.
		\end{equation}
		The proof of \eqref{ineq:indicator 2} is similar to \eqref{ineq:indicator 1}.
		
		\item Putting the previous bounds of $L_1, L_2$ and $L_3$ together yields
		\begin{align*}
			I_1 \lesssim K^8 m^{-1} + K^4 m^{-5} + K^4 \sqrt{\frac{n \log^3 m}{m}}
		\end{align*}
		with probability exceeding $1 - m\exp\xkh{-c_0 n} - O(m^{-20})$.
		
	\end{itemize}

	Combining the bounds of $I_1, I_2, I_3 $ and $\theta_2$ together, we obtain that for any $\delta>0$, with probability exceeding $1 - m\exp\xkh{-c_0 n} - O(m^{-20})$, it holds
	\[
\norm{\frac1m \sum_{j=1}^m \aabs{\va_j^\hh \vz}^2\va_j\va_j^\hh - \E \zkh{\aabs{\va^\hh \vz}^2\va\va^\hh }}  \le \delta \norm{\vz}^2
\]
	provided $m\ge C_0  \delta ^{-2} K^8n \log^3 m$ for some sufficiently large constant $C_0>0$.  This completes the proof of \eqref{eq:nabal11}.
	
	The proof of \eqref{eq:nabal21} is similar, so we omit it. 

\end{document}